\documentclass[12pt]{amsart}
\usepackage{format}

\newcommand{\C}{\mathbb{C}}

\title{Unipotent Ideals for Spin and Exceptional Groups}

\author{Lucas Mason-Brown}
\address{Mathematical Institute, University of Oxford, Oxford, UK}
\email{lucas.masonbrown@gmail.com, dmitriy.matveevskiy@gmail.com}

\author{Dmytro Matvieievskyi}
\address{Department of Mathematics, Northeastern University, Boston, MA USA}

\begin{document}


\maketitle

\begin{abstract}
In the monograph \cite{LMBM}, we define the notion of a unipotent representation of a complex reductive group. The representations we define include, as a proper subset, all special unipotent representations in the sense of \cite{BarbaschVogan1985} and form the (conjectural) building blocks of the unitary dual. In \cite{LMBM} we provide combinatorial formulas for the infinitesimal characters of all unipotent representations of linear classical groups. In this paper, we establish analogous formulas for spin and exceptional groups, thus completing the determination of the infinitesimal characters of all unipotent ideals. Using these formulas, we prove an old conjecture of Vogan: all unipotent ideals are maximal. For $G$ a real reductive Lie group (not necessarily complex), we introduce the notion of a unipotent representation attached to a rigid nilpotent orbit (in the complexified Lie algebra of $G$). Like their complex group counterparts, these representations form the (conjectural) building blocks of the unitary dual. Using the {\fontfamily{cmtt}\selectfont atlas} software (and the work of \cite{AdamsMillerVogan}) we show that if $G$ is a real form of a simple group of exceptional type, all such representations are unitary.
\end{abstract}

\tableofcontents

\section{Introduction}

Let $G$ be a complex reductive algebraic group and let $\fg = \mathrm{Lie}(G)$. Inspired by the orbit method for nilpotent and solvable Lie groups, Vogan proposes the following in \cite{Vogan1990}:

\begin{conj}\label{conj:quantization}
For each (finite connected) cover $\widetilde{\OO}$ of a nilpotent co-adjoint $G$-orbit, there is a canonically defined completely prime primitive ideal $I_0(\widetilde{\OO})$ in the universal enveloping algebra $U(\fg)$.
\end{conj}

The properties of the ideal $I_0(\widetilde{\OO})$ and its relation to $\widetilde{\OO}$ are described in some detail in \cite{Vogan1990} and also in \cite[Chp 9]{Vogan1987}. In particular, the ideals $I_0(\widetilde{\OO})$ should be maximal and should include, as a proper subset, all special unipotent ideals in the sense of \cite{BarbaschVogan1985}. Conjecture \ref{conj:quantization} is called the \emph{quantization problem for nilpotent covers} and the conjectured ideals are called \emph{unipotent}. We note that Conjecture \ref{conj:quantization} is intimately related to the problem of classifying unitary representations. If we can define unipotent ideals in $U(\fg)$, we can define unipotent representations (of the complex group $G$)---a $G$-representation is \emph{unipotent} if it is irreducible and annihilated (on both sides) by a unipotent ideal. It is conjectured in \cite[Chp 9]{Vogan1987} that unipotent representations are unitary, and in fact form the `building blocks' of the unitary dual.

Conjecture \ref{conj:quantization} has generated a large body of research over the past several decades, with contributions from Barbasch (\cite{Barbasch1989}), Vogan (\cite{Vogan1986a},\cite{Vogan1988},\cite{Vogan1990}), Joseph (\cite{Joseph1976}), McGovern (\cite{McGovern1994}), Brylinski (\cite{Brylinski2003}), and others. Previous approaches have made use of specialized constructions to resolve Conjecture \ref{conj:quantization} in certain special cases (e.g. for $\widetilde{\OO}$ equal to the \emph{minimal nilpotent orbit}, as in \cite{Joseph1976}, or for $G$ equal to a linear classical group, as in \cite{Barbasch1989},\cite{McGovern1994}, and \cite{Brylinski2003}). In the monograph \cite{LMBM}, we give a systematic solution to Conjecture \ref{conj:quantization} in all cases, using the theory of filtered quantizations of conical symplectic singularities. Our definition of $I_0(\widetilde{\OO})$ will be recalled in Section \ref{sec:inflchars}. 

Write $\fz(U(\fg))$ for the center of $U(\fg)$. If $I \subset U(\fg)$ is a primitive ideal, then $I \cap \fz(U(\fg))$ is the kernel of an algebra homomorphism  $\fz(U(\fg)) \to \CC$, called the \emph{infinitesimal character} of $I$. If we fix a Cartan subalgebra $\fh \subset \fg$, we can use the Harish-Chandra isomorphism $\mathfrak{z}(U(\fg)) \simeq \CC[\fh^*]^{W}$ to identify infinitesimal characters with $W$-orbits on $\fh^*$. Write $\gamma_0(\widetilde{\OO}) \in \fh^*/W$ for the infinitesimal character of the unipotent ideal $I_0(\widetilde{\OO})$. One drawback of the approach taken in \cite{LMBM} is that the infinitesimal character $\gamma_0(\widetilde{\OO})$ is not easy to determine from the definition of $I_0(\widetilde{\OO})$. The computation of $\gamma_0(\widetilde{\OO})$ requires a detailed analysis of the birational geometry of $\widetilde{\OO}$. In \cite{LMBM}, we deduce  combinatorial formulas for all $\gamma_0(\widetilde{\OO})$ in the case of linear classical groups. In this paper, we establish analogous formulas for spin and exceptional groups, thus completing the computation of all unipotent ideals.

We will now provide a more detailed overview of the results contained in this paper. One of the crucial ideas herein is \emph{birational induction}. This is an operation which takes nilpotent covers for a Levi subgroup $L \subset G$ to nilpotent covers for $G$. A nilpotent cover is called \emph{birationally rigid} if it cannot be obtained via birational induction from a proper Levi subgroup. In Sections \ref{subsec:birigidorbits} and \ref{subsec:semirigid}, we give a classification of birationally rigid nilpotent covers for simple exceptional groups, see Propositions \ref{prop:listofbirigid} and \ref{prop:listofsemirigid}.

One important property of unipotent ideals is that their infinitesimal characters are preserved under birational induction. More precisely, if $\widetilde{\OO}_G$ is birationally induced from $\widetilde{\OO}_L$, then
\begin{equation}\label{eq:preservationofinflchar}\gamma_0(\widetilde{\OO}_G) = \gamma_0(\widetilde{\OO}_L).\end{equation}
In view of (\ref{eq:preservationofinflchar}), it is enough to compute the infinitesimal characters attached to birationally rigid covers. For simple exceptional groups, we provide in Section \ref{subsec:inflcharexceptional} a complete list of such infinitesimal characters (see Tables \ref{table:orbitG2}-\ref{table:coversE8}). For spin groups, we provide in Section \ref{subsec:inflcharspin} a combinatorial formula (see Proposition \ref{prop:inflcharspin}).

These computations allow us to complete the proof of the following conjecture of Vogan (see \cite[Conj 9.18]{Vogan1987}).

\begin{theorem}[See Theorem \ref{thm:maximality} below]\label{thm:maximalityintro}
Let $G$ be a complex connected reductive algebraic group and let $\widetilde{\OO}$ be a (finite connected) covering of a nilpotent co-adjoint $G$-orbit. Then $I_0(\widetilde{\OO})$ is a maximal ideal in $U(\fg)$.
\end{theorem}

In Section \ref{sec:realgroups}, we investigate the applicability of these ideas to real reductive groups. Suppose $G$ is a real reductive Lie group (not necessarily linear) and let $\OO \subset \fg^*$ be a rigid nilpotent orbit for the complexified Lie algebra of $G$. Let $K \subset G$ be a maximal compact subgroup. We propose the following definition.

\begin{definition}[See Definition \ref{def:unipotentreal} below]
A unipotent representation of $G$ attached to $\OO$ is an irreducible $(\fg,K)$-module $\cB$ such that $\Ann_{U(\fg)}(\cB) = I_0(\OO)$. 
\end{definition}

If $G$ is a real form of a simple exceptional group and $\OO$ is a special nilpotent orbit, then it will be shown in \cite{AdamsMillerVogan} that all such representations are unitary. In Section \ref{subsec:unitarity}, we do the same for non-special orbits, proving the following result.

\begin{theorem}[See Theorem \ref{thm:unitarity} below]
Suppose $G$ is a real form of a simple exceptional group and $\OO \subset \fg^*$ is a rigid nilpotent orbit. Then all unipotent representations attached to $\OO$ are unitary.
\end{theorem}

The proof of this result is a straightforward (albeit time-consuming) computation using the {\fontfamily{cmtt}\selectfont atlas} software. 

The methods used in this paper mostly belong to the theory of filtered quantizations of conical symplectic singularities. We review what is needed from this theory in Section \ref{sec:quant}. The reader who is not interested in our methods may prefer to regard this paper as a sort of `treasure map' for interesting unitary representations of spin and exceptional groups. Many of the representations we describe have never been studied in the literature.

\subsection{Acknowledgements} The authors would like to thank Ivan Losev and David Vogan for many helpful conversations. Many of the computations in this paper were assisted, and in some cases wholly carried out, using the {\fontfamily{cmtt}\selectfont atlas} software. We would like to acknowledge the whole {\fontfamily{cmtt}\selectfont atlas} team, especially Jeffrey Adams, Annegret Paul, Marc van Leeuwen, and David Vogan, for creating and maintaining such fantastically useful software. The work of D.M. was partially supported by the NSF
under grant DMS-2001139.

\section{Quantizations of conical symplectic singularities}\label{sec:quant}

In this section, we review the theory of filtered quantizations of conical symplectic singularities. The results in this section come from a variety of sources, including \cite{Beauville2000}, \cite{Kaledin2006}, \cite{Losev4}, and \cite{LMBM}. Our exposition roughly follows Section 4 of \cite{LMBM}. 

\subsection{Filtered quantizations}\label{subsec:filteredquant}

Let $A$ be a \emph{graded Poisson algebra} of degree $-d \in \ZZ_{<0}$. By this, we will mean a finitely-generated commutative associative unital algebra equipped with two additional structures: an algebra grading 
$$A = \bigoplus_{i=-\infty}^{\infty} A_i$$
and a Poisson bracket $\{\cdot, \cdot\}$ of degree $-d$
$$\{A_i, A_j\}\subset A_{i+j-d}, \qquad i,j \in \ZZ.$$
For any algebra of this form, one can define the notion of a \emph{filtered quantization}.

\begin{definition}\label{def:filteredquant}
A \emph{filtered quantization} of $A$ is a pair $(\cA,\theta)$ consisting of
\begin{itemize}
    \item[(i)] an associative algebra $\cA$ equipped with a complete and separated filtration by subspaces
    $$\cA = \bigcup_{i=-\infty}^{\infty} \cA_{\leq i}, \qquad ... \subseteq \cA_{\leq -1} \subseteq \cA_{\leq 0} \subseteq \cA_{\leq 1} \subseteq ...$$
    such that
    $$[\cA_{\leq i}, \cA_{\leq j}] \subseteq \cA_{\leq i+j-d} \qquad i,j \in \ZZ,$$
    and
    \item[(ii)] an isomorphism of graded Poisson algebras
    $$\theta: \gr(\cA) \xrightarrow{\sim} A,$$
    where the Poisson bracket on $\gr(\cA)$ is defined by
    $$\{a+\cA_{\leq i-1}, b+\cA_{\leq j-1}\}=[a,b]+\cA_{\leq i+j-d-1}, \qquad a \in \cA_{\leq i}, \ b \in \cA_{\leq j}.$$
\end{itemize}
An isomorphism of filtered quantizations $(\cA_1, \theta_1) \xrightarrow{\sim} (\cA_2, \theta_2)$ is an isomorphism of filtered algebras $\phi: \cA_1 \xrightarrow{\sim} \cA_2$ such that $\theta_1 = \theta_2 \circ \gr(\phi)$. Denote the set of isomorphism classes of quantizations of $A$ by $\mathrm{Quant}(A)$.
\end{definition}

Often, the isomorphism $\theta$ is clear from the context, and will be omitted from the notation. However, the reader should keep in mind that a filtered quantization $(\cA,\theta)$ is \emph{not} determined up to isomorphism by $\cA$ alone.

Now, suppose $G$ is an algebraic group which acts rationally on $A$ by graded Poisson automorphisms. Write $\mathrm{Der}(A)$ for the Lie algebra of derivations. The $G$-action on $A$ induces by differentiation a Lie algebra homomorphism
$$\fg \to \mathrm{Der}(A), \qquad \xi \mapsto \xi_A,$$
We say that $A$ is \emph{Hamiltonian} if there is a $G$-equivariant map $\varphi: \mathfrak{g} \to A_d$ (called a \emph{classical co-moment map}) such that
$$\{\varphi(\xi), a\} = \xi_A(a), \qquad \xi \in \fg, \quad a \in A.$$ 
A filtered quantization $(\cA,\theta)$ is $G$-\emph{equivariant} if $G$ acts rationally on $\cA$ by filtered algebra automorphisms and the isomorphism $\theta: \gr(\cA) \xrightarrow{\sim} A$ is $G$-equivariant. In this setting (as above) we get a Lie algebra homomorphism
$$\fg \to \mathrm{Der}(\cA), \qquad \xi \mapsto \xi_{\cA}.$$

\begin{definition}\label{def:hamiltonian}
Suppose $A$ is a graded Poisson algebra equipped with a Hamiltonian $G$-action. A \emph{Hamiltonian} quantization of $A$ is a triple $(\cA,\theta,\Phi)$ consisting of
 \begin{itemize}
     \item[(i)] a $G$-equivariant filtered quantization $(\cA,\theta)$ of $A$, and
     \item[(ii)] a $G$-equivariant map $\Phi: \fg \to \cA_{\leq d}$ (called a \emph{quantum co-moment map}) such that
     $$[\Phi(\xi),a] = \xi_{\cA}(a), \qquad \xi \in \fg, \quad a \in \cA.$$
 \end{itemize}
An isomorphism $(\mathcal{A}_1,\theta_1,\Phi_1) \xrightarrow{\sim} (\mathcal{A}_2,\theta_2,\Phi_2)$ of Hamiltonian quantizations of $A$ is a $G$-equivariant isomorphism of filtered algebras $\phi: \cA_1 \to \cA_2$ such that $\theta_1 = \theta_2 \circ \gr(\phi)$ and $\Phi_2 = \phi \circ \Phi_1$. Denote the set of isomorphism classes of Hamiltonian quantizations of $A$ by $\mathrm{Quant}^G(A)$.
\end{definition}

\subsection{Conical symplectic singularities}
Let $X$ be a normal Poisson variety. 

\begin{definition}[\cite{Beauville2000}, Def 1.1]\label{def:symplecticsing}
We say that $X$ has \emph{symplectic singularities} if
\begin{itemize}
    \item[(i)] the regular locus $X^{\mathrm{reg}} \subset X$ is symplectic; denote the symplectic form by $\omega^{\mathrm{reg}}$. 
    \item[(ii)] there is a resolution of singularities $\rho: Y \to X$ such that $\rho^*(\omega^{\mathrm{reg}})$ extends to a regular (not necessarily symplectic) $2$-form on $Y$.
\end{itemize}
\end{definition}

In this paper, we will consider symplectic singularities of a very special type. Let $d \in \ZZ_{>0}$ as in Section \ref{subsec:filteredquant}.

\begin{definition}\label{def:conicalsymplecticsing}
A \emph{conical symplectic singularity} is a normal affine Poisson variety $X$ with symplectic singularities and a contracting rational $\CC^{\times}$-action such that $\CC[X]$ is a graded Poisson algebra of degree $-d$. 
\end{definition}

\begin{example}\label{example:symplecticsingularity}
The following are examples of conical symplectic singularities
		\begin{itemize}
		    \item[(i)] Let $\Gamma\subset \mathrm{Sp}(2)$ be a finite subgroup. Then the Kleinian singularity $\Sigma = \CC^2/\Gamma$ is a conical symplectic singularity, see  \cite[Prop 2.4]{Beauville2000}. For $\rho$ we take the minimal resolution $\mathfrak{S} \to \Sigma$.
		    \item[(ii)] Let $\fg$ be a complex reductive Lie algebra and let $\mathbb{O} \subset \fg^*$ be a nilpotent orbit. Then $\Spec(\CC[\mathbb{O}])$ is a conical symplectic singularity, see  \cite[Sec 2.5]{Beauville2000}. 
		    \item[(iii)] In the setting of $(ii)$, let $\widetilde{\mathbb{O}} \to \mathbb{O}$ be a connected finite \'{e}tale cover. Then $\Spec(\CC[\widetilde{\mathbb{O}}])$ is a conical symplectic singularity, see \cite[Lem 2.5]{LosevHC}.
		\end{itemize}
	 \end{example}	
	 
For an arbitrary variety $X$, define the subvarieties $X_0,X_1,X_2,...$ as follows: $X_0:=X$ and  $X_{k+1}:=X_k- X_k^{\mathrm{reg}}$. If $X$ is Poisson, then all $X_k$ are Poisson subvarieties of $X$.

\begin{definition}\label{def:fin_many_leaves} We say that $X$ has finitely many (symplectic) leaves if  $X_k^{\mathrm{reg}}$ is a symplectic variety for all $k$. By a symplectic leaf of $X$ we mean an irreducible (i.e. connected) component of $X_k^{\mathrm{reg}}$ for some $k$.
\end{definition}

\begin{prop}[Thm 2.3, \cite{Kaledin2006}]\label{Prop:fin_many_leaves_Kaledin}
Suppose $X$ has symplectic singularities. Then $X$ has finitely many leaves. 
\end{prop}

\subsection{Namikawa space}

Let $X$ be a normal Poisson variety with symplectic singularities. Recall that a normal variety $Y$ is $\QQ$-\emph{factorial} if every Weil divisor has a (nonzero) integer multiple which is Cartier. The following is a consequence of \cite{BCHM} (see  \cite[Prop 2.1]{LosevSRA} for a proof).

\begin{prop}\label{prop:terminalization}
There is a birational projective morphism $\rho: Y\to X$ such that
\begin{itemize}
    \item[(i)] $Y$ is an irreducible, normal, Poisson variety  (in particular, $Y$ has symplectic singularities).
    \item[(ii)] $Y$ is $\QQ$-factorial.
    \item[$(iii)$] $Y$ has terminal singularities.
\end{itemize}
\end{prop}

\begin{rmk}\label{rmk:terminalization}
Modulo (i), (iii) is equivalent to the condition that the singular locus of $Y$ is of codimension $\geq 4$, see \cite[Main Thm]{Namikawa_note}. In practice, the latter condition is often easier to check.
\end{rmk}

The map $\rho: Y\to X$ in the proposition above (or the variety $Y$ itself, if the map is understood) is called a $\QQ$-\emph{factorial terminalization}. If $X$ is conical, then $Y$ admits a $\CC^{\times}$-action such that $\rho$ is $\CC^{\times}$-equivariant, see \cite[A.7]{Namikawa3}. 

\begin{example}\label{Ex:Springer_resolution}
Let $\fg$ be a complex reductive Lie algebra and let $\cN \subset \fg^*$ be its nilpotent cone. By Example \ref{example:symplecticsingularity}(ii) (and the normality of $\cN)$, $X:=\cN$ is a conical symplectic singularity. For $\rho: Y \to X$ we take the Springer resolution $T^*(G/B) \to X$. 
\end{example}

\begin{definition}\label{defi:Namikawa_space}
Let $X$ be a conical symplectic singularity and $Y$ a $\QQ$-factorial terminalization of $X$.
The \emph{Namikawa space} associated to $X$ is the complex vector space 
$$\fP^X := H^2(Y^{\mathrm{reg}},\CC).$$
\end{definition}
\begin{rmk}
It was shown in \cite{LMBM} that $\fP^X$ depends only on $X$ up to canonical isomorphism (i.e. it is independent of the choice of $\QQ$-terminalization), see \cite[Lem 4.6.6]{LMBM} and the discussion following it. In particular, the notation $\fP^X$ is justified.
\end{rmk}

\subsection{Structure of Namikawa space}\label{subsec:structureNamikawa}

Following \cite{Namikawa} and \cite{Losev4}, we will provide a description of $\fP^X$ in terms of the geometry of $X$. For this discussion, it is convenient to fix a $\QQ$-terminalization $\rho: Y \to X$. For each codimension 2 leaf $\fL_k \subset X$, the formal slice to $\fL_k \subset X$ is identified with the formal neighborhood at $0$ in a Kleinian singularity $\Sigma_k = \CC^2/\Gamma_k$, see \cite{Namikawa}. Under the McKay correspondence, $\Gamma_k$ corresponds to a complex simple Lie algebra $\fg_k$ of type A, D, or E. Fix a Cartan subalgebra $\fh_k \subset \fg_k$. Write $\Lambda_k \subset \fh_k^*$ for the weight lattice and $W_k$ for the Weyl group. If we choose a point $x \in \fL_k$, there is a natural identification $H^2(\rho^{-1}(x),\ZZ) \simeq \Lambda_k$, and $\pi_1(\fL_k)$ acts on $\Lambda_k$ by diagram automorphisms. The \emph{partial Namikawa space} corresponding to $\fL_k$ is the subspace
$$\fP^X_{k} := (\fh_{k}^*)^{\pi_1(\fL_k)}.$$
The embedding $\rho^{-1}(x) \subset Y^{\mathrm{reg}}$ induces a map on cohomology
$$\fP^X := H^2(Y^{\mathrm{reg}},\CC) \to  H^2(\rho^{-1}(x),\CC)^{\pi_1(\fL_k)} \simeq (\fh_{k}^*)^{\pi_1(\fL_k)} =: \fP^X_k.$$
Also, define
$$\fP^X_{0} := H^2(X^{\mathrm{reg}},\CC).$$
The embedding $X^{\mathrm{reg}} \subset Y^{\mathrm{reg}}$ induces a map on cohomology
$$\fP^X := H^2(Y^{\mathrm{reg}},\CC) \to H^2(X^{\mathrm{reg}},\CC) =: \fP^X_{0}.$$
\begin{prop}[\cite{Losev4}, Lem 2.8]\label{prop:partialdecomp}
The maps $\fP^X \rightarrow \fP^X_{k}$ assemble into a linear isomorphism
\begin{equation}\label{eq:partialdecomp}\fP^X \simeq\bigoplus_{k=0}^t \fP^X_{k}, \qquad \lambda \mapsto (\lambda_0,\lambda_1,...,\lambda_t).\end{equation}
\end{prop}

Finally, we define the \emph{Namikawa Weyl group} of $X$. The $\pi_1(\fL_k)$-action on $\Lambda_k$ induces a $\pi_1(\fL_k)$-action on $W_k$. Consider the subgroup $W_k^{\pi_1(\fL_k)} \subset W_k$. There is a natural action of $W_k^{\pi_1(\fL_k)}$ on $\fP^X_{k} = (\fh_{k}^*)^{\pi_1(\fL_k)}$. The \emph{Namikawa Weyl group} associated to $X$ is the product 
$$W := \prod_k W_k^{\pi_1(\fL_k)}.$$
We let $W$ act on $\fP$ via the isomorphism (\ref{eq:partialdecomp}) (the action on $\fP^X_{0}$ is trivial).

\subsection{Finite covers of conical symplectic singularities}

Let $X$ be a conical symplectic singularity. In this section, we will define the notion of a \emph{finite cover} of $X$. Let $p': \widetilde{X}' \to X^{\mathrm{reg}}$ be a finite \'{e}tale cover of the regular locus $X^{\mathrm{reg}} \subset X$. Rescaling if necessary, we can arrange so that the $\CC^{\times}$-action on $X^{\mathrm{reg}}$ lifts to $\widetilde{X}'$. Consider the composition $\widetilde{X}' \overset{p'}{\to} X^{\mathrm{reg}} \hookrightarrow X$ and its Stein factorization
\begin{center}
\begin{tikzcd}
\widetilde{X}' \ar[r,hookrightarrow] \ar[d,"p'"] & \widetilde{X} \ar[d, "p"]\\
X^{\mathrm{reg}} \ar[r,hookrightarrow] & X
\end{tikzcd}
\end{center}
Note that $\widetilde{X}$ is affine and $\widetilde{X}'$ embeds into $\widetilde{X}$ as an open subvariety. Since $\operatorname{codim}(X^{\mathrm{sing}},X)\geqslant 2$ and $p:\widetilde{X} \to X$ is finite, we have that $\codim(\widetilde{X}-\widetilde{X}',\widetilde{X}) \geq 2$. Thus the algebra $\CC[\widetilde{X}']$ is finitely generated and
$ \widetilde{X}= \Spec(\CC[\widetilde{X}'])$. In particular, the $\CC^{\times}$-action on $\widetilde{X}'$ extends to $\widetilde{X}$. In fact, $\widetilde{X}$ is a conical symplectic singularity, see \cite[Lemma 2.5]{LosevHC}. A map $p:\widetilde{X} \to X$ obtained in this fashion is called a \emph{finite cover} of $X$. A finite cover $p: \widetilde{X} \to X$ is \emph{Galois} if it is Galois over the regular locus $X^{\mathrm{reg}}$. 

For each codimension 2 leaf $\fL_k \subset X$, choose a system of fundamental weights
$$\omega_1(k), \omega_2(k), ..., \omega_{n(k)}(k) \in \fh_k^*$$
To each fundamental weight $\omega_i(k)$ we assign a coefficient $a_i$, which is the multiplicity of the corresponding simple root in the highest weight for $\fg_k$ (in type $A$, $a_i=1$ for every $i$, and this is the only case that will concern us). Define the element
$$\epsilon_k := |\Gamma_k|^{-1} \sum_i a_i\omega_i \in \fh_k^*.$$
\begin{lemma}[Prop 5.3.1,\cite{LMBM}]\label{lem:barycenter}
Suppose $X$ admits a finite Galois cover $\widetilde{X} \to X$ such that $\widetilde{X}$ has no codimension 2 leaves. Then for each $\fL_k \subset X$, the element $\epsilon_k$ is a fixed point for $\pi_1(\fL_k)$ and hence an element of the partial Namikawa space $\fP_k = (\fh_k^*)^{\pi_1(\fL_k)}$.
\end{lemma}
In the setting of Lemma \ref{lem:barycenter}, define
\begin{equation}\label{eq:defofepsilon}\epsilon := (0, \epsilon_1,\epsilon_2, ..., \epsilon_t) \in \bigoplus_i \fP^X_i \simeq \fP^X\end{equation}
We call $\epsilon$ the \emph{weighted barycenter parameter} for $X$. We note that the Galois group $\Pi$ of $\widetilde{X} \to X$ acts on the canonical quantization $\cA_0^{\widetilde{X}}$, and $(\cA_0^{\widetilde{X}})^{\Pi} \simeq \cA_{\epsilon}^X$, see \cite[Prop 5.3.1]{LMBM}.

\subsection{Filtered quantizations of conical symplectic singularities}

In this section, we will recall the classification of filtered quantizations of conical symplectic singularities. Let $X$ be a conical symplectic singularity and choose a $\QQ$-terminalization $\rho: Y \to X$. For any graded smooth symplectic variety $V$, there is a (non-commutative) \emph{period map}
$$\mathrm{Per}: \mathrm{Quant}(V) \to H^2(V,\CC),$$
see \cite[Sec 4]{BK}, \cite[Sec 2.3]{Losev_isofquant}.

\begin{prop}[\cite{Losev4}, Prop 3.1(1)]\label{prop:P=Quant(Y)}
The maps 
$$\mathrm{Quant}(Y) \overset{|_{Y^{\mathrm{reg}}}}{\to} \mathrm{Quant}(Y^{\mathrm{reg}}) \overset{\mathrm{Per}}{\to} H^2(Y^{\mathrm{reg}},\CC) = \fP^{\widetilde{X}}$$
are bijections.
\end{prop}

For $\lambda \in \fP$, let  $\mathcal{D}_{\lambda}$ denote the corresponding filtered quantization of $Y$ and let $\cA_{\lambda} := \Gamma(Y,\mathcal{D}_{\lambda})$. 

\begin{theorem}[Prop 3.3, Thm 3.4, \cite{Losev4}]\label{thm:quantssofsymplectic}
The following are true:

\begin{itemize}
\item[(i)] For every $\lambda \in \fP$, the algebra $\cA_\lambda$ is a filtered quantization of $\CC[X]$.
    \item[(ii)] Every filtered quantization of $\CC[X]$ is isomorphic to $\cA_{\lambda}$ for some $\lambda \in \fP$.
    \item[(iii)] For every $\lambda, \lambda' \in \fP$, we have $\cA_{\lambda} \simeq \cA_{\lambda'}$ if and only if $\lambda' \in W\cdot \lambda$.
\end{itemize}
Hence, the map $\lambda \mapsto \cA_{\lambda}$ induces a bijection
$$\fP/W \simeq \mathrm{Quant}(\CC[X]), \qquad W \cdot \lambda \mapsto \cA_{\lambda}.$$
\end{theorem}

There is an equivariant version of this result, which we will now state. Let $G$ be a connected reductive algebraic group and suppose $A:=\CC[X]$ admits a Hamiltonian $G$-action, see Section \ref{subsec:filteredquant}. Define the \emph{extended Namikawa space}
$$\overline{\fP}^X := \fP^X \oplus \fz(\mathfrak{g})^*$$
This space should be viewed as an equivariant version of $\fP^X$. Let $W^X$ act on $\overline{\fP}^X$ via the decomposition above (the $W$-action on the second factor is defined to be trivial). 

\begin{prop}[Lem 4.11.2, \cite{LMBM}]\label{prop:Hamiltonian}
Let $G$ be a connected reductive algebraic group and suppose $A:=\CC[X]$ admits a Hamiltonian $G$-action. Then the following are true:
\begin{itemize}
    \item[(i)] There is a unique classical co-moment map $\varphi: \fg \to A_d$.
    \item[(ii)] Every filtered quantization $\cA \in \mathrm{Quant}(A)$ has a unique $G$-equivariant structure.
    \item[(iii)] For every $\cA \in \mathrm{Quant}(A)$ and $\chi \in \fz(\fg)^*$, there is a unique quantum co-moment map $\Phi_{\chi}: \fg \to \cA_{\leq d}$ such that $\Phi|_{\fz(\fg)} = \chi$.
\end{itemize}
In particular, there is a canonical bijection
$$\overline{\fP}^X/W^X \xrightarrow{\sim} \mathrm{Quant}^G(A) \qquad W^X (\lambda, \chi) \mapsto (\cA_{\lambda}^X, \Phi_{\chi}).$$
\end{prop}

One consequence of this proposition is that there is always a distinguished quantization of $X$.

\begin{definition}[Def 5.0.1, \cite{LMBM}]\label{def:canonical}
The \emph{canonical Hamiltonian quantization} of $\CC[X]$ is the Hamiltonian quantization corresponding to the parameter $0 \in \overline{\fP}^X$. 
\end{definition}

\section{Nilpotent orbits and covers}

In this section, we collect some basic facts about nilpotent orbits and their (finite connected) covers. 

\subsection{Classification of nilpotent orbits and covers}

In classical types, nilpotent orbits are classified by (decorated) integer partitions.

\begin{definition}
A partition $p$ is of \emph{type C} (resp type \emph{B/D}) if every odd part (resp. even part) occurs with even multiplicity.  
\end{definition}

The following result is well-known.

\begin{prop}[Section 5.1, \cite{CM}]\label{prop:orbitstopartitions}
Suppose $\fg$ is classical. The set of nilpotent orbits $\mathbb{O} \subset \fg^*$ is parameterized by (decorated) partitions as follows
\begin{enumerate}[label=(\alph*)]
    \item If $\fg = \mathfrak{sl}(n)$, the set of nilpotent orbits is in one-to-one correspondence with partitions of $n$.
    \item If $\fg = \mathfrak{so}(2n+1)$, the set of nilpotent orbits is in one-to-one correspondence with partitions of $2n+1$ of type B/D.
    \item If $\fg = \mathfrak{sp}(2n)$, the set of nilpotent orbits is in one-to-one correspondence with partitions of $2n$ of type C.
    \item If $\fg = \mathfrak{so}(2n)$, the set of nilpotent orbits is in one-to-one correspondence with partitions of $2n$ of type B/D, except that each \emph{very even} partition (i.e. a partition containing only even parts) corresponds to two nilpotent orbits, labeled $\mathbb{O}^I$ and $\mathbb{O}^{II}$.
\end{enumerate}
\end{prop}

If $p$ is a partition of the appropriate type, we will denote the corresponding nilpotent orbit by $\OO_p$. In exceptional types, we will use the Bala-Carter classification to label nilpotent orbits, see \cite[Sec 8.4]{CM} for an explanation. 

By a `nilpotent cover' we will mean a finite connected $G$-equivariant cover of a nilpotent co-adjoint orbit. Up to isomorphism, nilpotent covers of $\OO$ are paramterized by conjugacy classes of subgroups of the (necessarily finite) $G$-equivariant fundamental group $\pi_1^G(\OO)$. A description of these fundamental groups can be found in \cite[Sec 6.1]{CM} (for classical groups) and \cite[Sec 8.4]{CM} (for simply connected exceptional groups). We will occasionally need a description of $\pi_1^G(\OO)$ for
more general classes of groups (for example, the Levi subgroups of simply connected exceptional groups, which need not be simply connected). In these cases, we use the {\fontfamily{cmtt}\selectfont atlas} software to compute $\pi_1^G(\OO)$.

\subsection{Geometry of nilpotent orbits and covers}\label{subsec:geometrynilpotent}

In this section, we will collect some basic facts about the geometry of nilpotent orbits and covers. Suppose $\OO$ is a nilpotent orbit. The singular locus of $\overline{\OO}$ coincides with the boundary $\partial \OO = \overline{\OO} \setminus \OO$, see \cite[Proof of Prop 2.2]{Namikawa2004}. Let $\OO'$ be a maximal $G$-orbit in $\partial \OO$. For each point $e' \in \OO'$, there is a transverse slice $S_{\OO,\OO'}$ to $\OO'$ in $\overline{\OO}$ (obtained, for example, by intersecting the Slodowy slice to $\OO'$ at $e'$ with $\overline{\OO}$). The variety $S_{\OO,\OO'}$ has an isolated singularity at $e'$ and a natural rational $\CC^{\times}$-action which is contracting onto $e'$, see \cite[Sec 4]{GanGinzburg}. We note that $S_{\OO,\OO'}$ is independent, up to algebraic isomorphism, of the choice of $e'$ as well as the $\mathfrak{sl}(2)$-triple $(e',f',h')$ used to define the Slodowy slice. We call $S_{\OO,\OO'}$ the \emph{singularity} of the orbit $ \OO' \subset \overline{\OO}$. The varieties $S_{\OO',\OO'}$ were described in classical types by Kraft and Procesi in \cite{Kraft-Procesi} and in exceptional types by Fu, Juteau, Levy, and Sommers in \cite{fuetal2015}.

In this paper, we will restrict our attention to the singularities corresponding to codimension 2 orbits $\OO' \subset \overline{\OO}$, i.e. to \emph{dimension 2 singularities} in $\overline{\OO}$. The dimension 2 singularities in nilpotent orbit closures can be rather complicated (sometimes non-normal) varieties. We will briefly recall some of the standard conventions for denoting them, from \cite{Slodowy},\cite{Kraft-Procesi} and \cite{fuetal2015}. Fix a dimension 2 singularity $S_{\OO,\OO'}$. If $S_{\OO,\OO'}$ is normal, it is isomorphic to a Kleinian singularity of type $A$, $D$, or $E$. As explained in Section \ref{subsec:structureNamikawa}, the fundamental group $\pi_1(\OO')$ acts on the Dynkin diagram of $S_{\OO,\OO'}$ by a finite group $K$ of diagram automorphisms. Following \cite{Slodowy} and \cite{fuetal2015}, the pair $(S_{\OO,\OO'},K)$ is denoted by
\begin{itemize}
    \item $B_k$, if $S_{\OO,\OO'}$ is of type $A_{2k-1}$, and $K=S_2$,
    \item $C_k$, if $S_{\OO,\OO'}$ is of type $D_{k+1}$, and $K=S_2$,
    \item $F_4$, if $S_{\OO,\OO'}$ is of type $E_6$, and $K=S_2$,
    \item $G_2$, if $S_{\OO,\OO'}$ is of type $D_4$, and $K=S_3$,
    \item $A_{2k}^{+}$, if $S_{\OO,\OO'}$ is of type $A_{2k}$, and $K=S_2$. 
\end{itemize}
If $S_{\OO,\OO'}$ is \emph{non-normal} and $\fg$ is a classical Lie algebra, then $S_{\OO,\OO'}$ is of the following type:
\begin{itemize}
    \item $nA_k$: $n$ copies of the Kleinian singularity of type $A_k$, meeting at the singular point.
\end{itemize}
If $\fg$ is exceptional, there are several additional non-normal singularities which can appear. In the notation of \cite{fuetal2015}, they are:
\begin{itemize}
    \item $nD_k$: $n$ copies of the Kleinian singularity of type $D_k$, meeting at the singular point.
    \item  $m$:  a non-normal 2-dimensional conical singularity admitting an $\mathrm{SL}(2)$-action with an open orbit isomorphic to $\CC^2 \setminus \{0\}$ and normalization isomorphic to $\CC^2$, see \cite[Section 1.8.4]{fuetal2015} for details.
    \item $\mu$: a non-normal 2-dimensional conical singularity with normalization isomorphic to a Kleinian singularity of type $A_3$, see \cite[Sec 1.8.4]{fuetal2015} for details.
\end{itemize}

The next lemma, which asserts that the above singularities are in fact the only possibilities, is immediate from \cite{Kraft-Procesi} and \cite{fuetal2015}.

\begin{lemma}\label{not normal sing}
Let $\OO'\subset \overline{\OO}$ be a codimension 2 orbit, and assume that the singularity $S_{\OO,\OO'}$ is non-normal. Then $S_{\OO,\OO'}$ is of type $nA_k$, $nD_k$ (for some $n$, $k$), $m$, or $\mu$.
\end{lemma}

Now let $\widetilde{\OO}$ be a (finite connected) cover of $\OO$ and let $\widetilde{X} = \Spec(\CC[\widetilde{\OO}])$. Recall, (iii) of Example \ref{example:symplecticsingularity}, that $\widetilde{X}$ is a conical symplectic singularity. Let $\mu:\widetilde{X} \to \overline{\OO} \subset \fg^*$ denote the moment map. If $\fL \subset \widetilde{X}$ is a codimension 2 leaf, then $\mu(\overline{\fL})$ is the closure of a codimension 2 orbit $\OO' \subset \overline{\OO}$, see \cite[Lem 4.6.1]{LMBM}. This defines a map
\begin{equation}\label{eq:leaftoorbit}\{\text{codimension 2 leaves in } \widetilde{X}\} \to \{\text{codimension 2 orbits in }  \overline{\OO}\}\end{equation}
Let $\Sigma$ denote the Kleinian singularity corresponding to the leaf $\fL \subset \widetilde{X}$, see Section \ref{subsec:structureNamikawa}. There is a closed embedding $\Sigma \subset \widetilde{X}$, constructed as follows. Since $S_{\OO',\OO}$ is transverse to $\OO'$, $\mu^{-1}(S_{\OO',\OO})$ is transverse to every leaf in $\mu^{-1}(\OO')$. Thanks to the contracting $\CC^{\times}$-action on $S_{\OO',\OO}$, $\mu^{-1}(S_{\OO',\OO})$ splits into a disjoint union of connected components, indexed by points in $\mu^{-1}(S_{\OO',\OO} \cap \OO')$. Choose a point in this set lying in $\fL$ and let $\Sigma$ be its connected component in $\mu^{-1}(S_{\OO',\OO})$. Then $\Sigma$ is a Kleinian singularity and $\operatorname{Spec}(\C[\Sigma]^{\wedge})$ is a formal slice to $\fL$. 

If we specialize to the case of nilpotent orbits, the map  (\ref{eq:leaftoorbit}) is almost always a bijection.

\begin{lemma}\label{lem:weaklynormal}
Suppose $\widetilde{\OO}=\OO$ is a nilpotent orbit. Then the following are true:
\begin{itemize}
    \item[(i)] The map (\ref{eq:leaftoorbit}) is injective.
    \item[(ii)] The map (\ref{eq:leaftoorbit}) is surjective unless $\overline{\OO}$ contains a dimension 2 singularity of type $m$.
\end{itemize}
\end{lemma}
\begin{proof}
    If $\overline{\OO}$ is normal in codimension 2, then $\mu$ is an isomorphism over every codimension 2 orbit in $\overline{\OO}$. Hence, (\ref{eq:leaftoorbit}) is a bijection. 
    
    If $\OO'\subset \overline{\OO}$ is a codimension 2 orbit, $\mu^{-1}(\OO')$ is a (union of) codimension 2 orbits. Furthermore, by \ref{not normal sing}, $\mu^{-1}(\OO')$ lies in the singular locus of $X$ if and only if the singularity $S_{\OO,\OO'}$ is not of type $m$. This proves (ii) in all cases. To prove (i) we consider all orbits $\OO$ such that $\overline{\OO}$ contains a non-normal dimension 2 singularity not of type $m$. The general idea is that if the singularity corresponds to a partial Namikawa space of dimension $t$, then there are at most $[\frac{\dim \fP^X}{t}]$ leaves with this singularity.
    
    \begin{itemize}
        \item[(i)] Let $\fg$ be a classical rank-$n$ Lie algebra not of type $A$. Suppose that $\OO_1 \subset \overline{\OO}$ corresponds to a minimal degeneration of type (e) in the sense of \cite[Table 1]{Kraft-Procesi}, and let $\alpha$ and $\beta$ be the corresponding partitions. Let $\OO_1, \ldots, \OO_t\subset \overline{\OO}$ be the codimension 2 orbits, and let $k$ be the largest integer such that $\alpha_k>\beta_k$. By loc.cit., $\alpha_{k}=\alpha_{k+1}+2t$ for some $t$. Consider the Levi subalgebra $\fl=\fg\fl(k)^{\times t}\times \fg(n-2kt)$. Let $\OO_L\subset \fl^*$ be the nilpotent orbit corresponding to the partition $\gamma$ given by $\gamma_i=\alpha_i-2t$ for $i\le k$, and $\gamma_i=\alpha_i$ for $i>k$. Set $\fP_L=\fP^{X_L}$. We note that $\OO$ is birationally induced from $(L,\OO_L)$, and for each $j\neq 2$ there is a codimension $2$ orbit $\OO_{L,j}\subset \overline{\OO}_L$ such that $\OO_j$ is birationally induced from $\OO_{L,j}$. By \cite[Lemma 4.16]{Mitya2020}, the singularities of $\OO_{L,j}$ in $X_L$ and of $\OO_j$ in $X$ are equivalent, and $\dim \fP_j^X=\dim \fP_j^{X_L}$. We have $\dim \fP^X=\dim \fP^{X_L}+t$, and therefore $\dim \fP_1^X=t$. It follows that there is only one codimension $2$ leaf in $X$ over the orbit $\OO_1$.
        
        \item[(ii)] Let $\fg$ be of type $F_4$, and set $\OO=C_3(a_1)$. There is one codimension $2$ orbit $\OO'=B_2\subset \overline{\OO}$, and the corresponding singularity is of type $2A_1$. Using \cite[Tables]{deGraafElashvili}, $\dim \fP^X=1$, and therefore there is one codimension $2$ leaf over $\OO'$. 
        
        \item[(iii)] Let $\fg$ be of type $F_4$, and set $\OO=C_3$. There is one codimension $2$ orbit $\OO'=F_4(a_1)\subset \overline{\OO}$, and the corresponding singularity is of type $4G_2$. Using \cite[Tables]{deGraafElashvili}, $\dim \fP^X=2$, and therefore there is one codimension $2$ leaf over $\OO'$. 
        
        \item[(iv)] Let $\fg$ be of type $E_6$, and set $\OO=A_4$. There is one codimension $2$ orbit $\OO'=D_4(a_1)\subset \overline{\OO}$, and the corresponding singularity is of type $3C_2$. Using \cite[Tables]{deGraafElashvili}, $\dim \fP^X=3$, and therefore there is one codimension $2$ leaf over $\OO'$. 
        
        \item[(v)] Let $\fg$ be of type $E_7$, and set $\OO=D_6(a_1)$. There are two codimension $2$ orbits $\OO_1=E_7(a_5)\subset \overline{\OO}$ and $\OO_2=D_5\subset \overline{\OO}$, and the corresponding singularities are of type $3C_2$ and $A_1$. Using \cite[Tables]{deGraafElashvili}, $\dim \fP^X=4$, and therefore there is one codimension $2$ leaf over $\OO_1$. 
        
        \item[(vi)] Let $\fg$ be of type $E_7$, and set $\OO=A_3+A_2$. There is one codimension $2$ orbit $\OO'=D_4(a_1)+A_1\subset \overline{\OO}$, and the corresponding singularity is of type $2A_1$. Using \cite[Tables]{deGraafElashvili}, $\dim \fP^X=1$, and therefore there is one codimension $2$ leaf over $\OO'$. 
        
        \item[(vii)] Let $\fg$ be of type $E_7$, and set $\OO=D_4(a_1)+A_1$. There are two codimension $2$ orbits $\OO_1=D_4(a_1)\subset \overline{\OO}$ and $\OO_2=A_3+2A_1\subset \overline{\OO}$, and the corresponding singularities are of type $3A_1$ and $A_1$. Using \cite[Tables]{deGraafElashvili}, $\dim \fP^X=2$, and therefore there is one codimension $2$ leaf over $\OO_1$. 
         
        \item[(viii)] Let $\fg$ be of type $E_8$, and set $\OO=E_7(a_1)$. There is one codimension $2$ orbit $\OO'=E_8(b_5)\subset \overline{\OO}$, and the corresponding singularity is of type $3C_5$. Using \cite[Tables]{deGraafElashvili}, $\dim \fP^X=5$, and therefore there is one codimension $2$ leaf over $\OO'$. 
        
        \item[(ix)] Let $\fg$ be of type $E_8$, and set $\OO=E_7(a_1)$. There are two codimension $2$ orbits $\OO_1=D_6(a_1)\subset \overline{\OO}$ and $\OO_2=A_6\subset \overline{\OO}$, and the corresponding singularities are of type $2A_1$ and $A_1$. Using \cite[Tables]{deGraafElashvili}, $\dim \fP^X=2$. Since $\dim \fP_2^X=1$, it implies that there is one codimension $2$ leaf over $\OO_1$. 
        
        \item[(x)] Let $\fg$ be of type $E_8$, and set $\OO=D_6(a_1)$. There are two codimension $2$ orbits $\OO_1=E_8(a_7)\subset \overline{\OO}$ and $\OO_2=D_5+A_1\subset \overline{\OO}$, and the corresponding singularities are of type $10G_2$ and $A_1$. Using \cite[Tables]{deGraafElashvili}, $\dim \fP^X=3$. Since $\dim \fP_2^X=1$, there is one codimension $2$ leaf over $\OO_1$. 
        
        \item[(xi)] Let $\fg$ be of type $E_8$, and set $\OO=A_6$. There is one codimension $2$ orbit $\OO'=E_8(a_7)\subset \overline{\OO}$, and the corresponding singularity is of type $5G_2$. Using \cite[Tables]{deGraafElashvili}, $\dim \fP^X=2$, and therefore there is one codimension $2$ leaf over $\OO'$.
        
        \item[(xii)] Let $\fg$ be of type $E_8$, and set $\OO=E_7(a_5)$. There are two codimension $2$ orbits $\OO_1=D_6(a_2)\subset \overline{\OO}$ and $\OO_2=E_6(a_3)+A_1\subset \overline{\OO}$, and the corresponding singularities are of types $2A_1$ and $m$. Using \cite[Tables]{deGraafElashvili}, $\dim \fP^X=1$, and therefore there is one codimension $2$ leaf over $\OO_1$. 
        
        \item[(xiii)] Let $\fg$ be of type $E_8$, and set $\OO=A_3+A_2$. There is one codimension $2$ orbit $\OO'=D_4(a_1)+A_1\subset \overline{\OO}$, and the corresponding singularity is of type $3A_1$. Using \cite[Tables]{deGraafElashvili}, $\dim \fP^X=1$, and therefore there is one codimension $2$ leaf over $\OO'$.

    \end{itemize}
\end{proof}

\begin{rmk}
Let $\OO_k \subset \overline{\OO}$ be a codimension 2 orbit such that the corresponding singularity is not of type $m$. Then by Lemma \ref{lem:weaklynormal}, there is a unique codimension 2 leaf $\fL_k \subset \Spec(\CC[\OO])$ which maps to $\OO_k$ under (\ref{eq:leaftoorbit}).
\end{rmk}

\subsection{Induction of nilpotent orbits}

Let $M \subset G$ be a Levi subgroup, and let $\mathbb{O}_M$ be a nilpotent $M$-orbit. Fix a parabolic subgroup $Q \subset G$ with a Levi decomposition $Q = MU$. The annihilator of $\fq$ in $\fg^*$ is a $Q$-stable subspace $\fq^{\perp} \subset \fg^*$. Choosing a nondegenerate invariant symmetric form on $\fg$, we get a $Q$-invariant identification $\fq^\perp \simeq \fu$. Form the $G$-equivariant fiber bundle $G \times^Q (\overline{\mathbb{O}}_M \times \mathfrak{q}^{\perp})$ over the partial flag variety $G/Q$. There is a proper $G$-equivariant map
$$\mu: G \times^Q (\overline{\mathbb{O}}_M \times \mathfrak{q}^{\perp}) \to \mathfrak{g}^* \qquad \mu(g,\xi) = \Ad^*(g)\xi$$
The image of $\mu$ is a closed irreducible $G$-invariant subset of $\cN$, and hence the closure in $\cN$ of a nilpotent $G$-orbit, denoted $\mathrm{Ind}^G_M\mathbb{O}_M \subset \fg^*$. The correspondence
$$\mathrm{Ind}^G_M: \{\text{nilpotent } M\text{-orbits}\} \to \{\text{nilpotent } G\text{-orbits}\}$$
is called \emph{Lusztig-Spaltenstein} induction. A nilpotent orbit is \emph{rigid} if it cannot be induced from a proper Levi subgroup. 

\begin{prop}[\cite{LusztigSpaltenstein1979} or \cite{CM}, Sec 7]\label{prop:propsofInd}
Lusztig-Spaltenstein induction has the following properties
\begin{itemize}
    \item[(i)] $\mathrm{Ind}^G_M$ depends only on $M$ (and not on $Q$)
    \item[(ii)] If $L \subset M$ is a Levi subgroup of $M$, then
    $$\Ind^G_L = \Ind^G_M \circ \Ind^M_L.$$
    \item[(iii)] If $\OO$ is a nilpotent orbit, there is a Levi sugroup $M \subset G$ and a rigid nilpotent $M$-orbit $\OO_M$ such that
    $$\OO = \Ind^G_M \OO_M$$
    \item[(iv)] If $\mathbb{O}_M \subset \fm^*$ is a nilpotent $M$-orbit and $\mathbb{O} = \Ind^G_M \mathbb{O}_M$, then
    $$\mathrm{codim}(\mathbb{O}_M,\cN_M) = \codim(\mathbb{O},\cN).$$
\end{itemize}
\end{prop}

In classical types, a classification of rigid nilpotent orbits and a description of induction in terms of partitions can be found in \cite[Sec 7.3]{CM}. In exceptional types, this information can be found in the tables appearing in \cite[Sec 4]{deGraafElashvili}. For the explicit computations in Sections \ref{subsec:semirigid} and \ref{subsec:inflcharexceptional}, we will make repeated use of these descriptions.

\subsection{Birational induction of nilpotent covers}\label{subsec:binduction}

Choose a Levi subgroup $M \subset G$, a nilpotent $M$-orbit $\OO_M$, and a (finite, connected) $M$-equivariant cover $\widetilde{\OO}_M$ of $\OO_M$. Let $\OO = \Ind^G_M \OO_M$. Consider the affine variety $\widetilde{X}_M := \Spec(\CC[\widetilde{\OO}_M])$. There is an $M$-action on $\widetilde{X}_M$ (induced from the $M$-action on $\widetilde{\mathbb{O}}_M$) and a finite surjective $M$-equivariant map $\widetilde{X}_M \to \overline{\mathbb{O}}_M$. Let $\widetilde{\mu}$ denote the composition
$$G \times^Q (\widetilde{X}_M \times \mathfrak{q}^{\perp}) \to G \times^Q (\overline{\mathbb{O}}_M \times \mathfrak{q}^{\perp}) \overset{\mu}{\to} \overline{\mathbb{O}}.$$
Note that $\widetilde{\mu}^{-1}(\OO) \to \OO$ is a (finite, connected) $G$-equivariant cover. The correspondence
\begin{align*}\mathrm{Bind}^G_M: \{M\text{-eqvt nilpotent covers}\} &\to \{G\text{-eqvt nilpotent covers}\}\\
\widetilde{\OO}_M &\mapsto \widetilde{\mu}^{-1}(\mathbb{O})
\end{align*}
is called \emph{birational induction}. A nilpotent cover is \emph{birationally rigid} if it cannot be birationally induced from a proper Levi subgroup. 

The main properties of birational induction are catalogued below.

\begin{prop}[Prop 2.4.1, \cite{LMBM}]\label{prop:propsofbind}
Birational induction has the following properties

\begin{enumerate}
    \item[(i)] $\mathrm{Bind}^G_M$ depends only on $M$ (and not on $Q$)
    \item[(ii)] If $L \subset M$ is a Levi subgroup of $M$, then
    $$\mathrm{Bind}^G_L = \mathrm{Bind}^G_M \circ \mathrm{Bind}^M_L.$$
    \item[(iii)] If $\widetilde{\mathbb{O}}$ is a $G$-equivariant nilpotent cover, there is a Levi subgroup $M \subset G$ and a birationally rigid $M$-equivariant nilpotent cover $\widetilde{\OO}_M$ such that 
    $$\widetilde{\mathbb{O}} = \mathrm{Bind}^G_M \widetilde{\mathbb{O}}_M.$$
    The pair $(M,\widetilde{\OO}_M)$ is called a \emph{birationally minimal induction datum} and is unique up to conjugation by $G$.
    \item[(iv)] If we write $\deg(\widetilde{\mathbb{O}}_M)$ for the degree of the covering map $\widetilde{\mathbb{O}}_M \to \mathbb{O}_M$, then
    $$\deg(\widetilde{\mathbb{O}}_M) \text{ divides } \deg(\mathrm{Bind}^G_M(\widetilde{\mathbb{O}}_M)).$$
\end{enumerate}
\end{prop}

\subsection{Filtered quantizations of nilpotent covers}

Let $\widetilde{\OO}$ be a nilpotent cover and consider the affine variety $\widetilde{X} := \Spec(\CC[\widetilde{\OO}])$. By (iii) of Example \ref{example:symplecticsingularity}, $\widetilde{X}$ is a conical symplectic singularity. Fix the notation of Section \ref{subsec:structureNamikawa}, i.e. $\fP^{\widetilde{X}}$, $W^{\widetilde{X}}$, $\fP_k^{\widetilde{X}}$, and so on. By Theorem \ref{thm:quantssofsymplectic}, there is a canonical bijection
\begin{equation}\label{eq:nilpquant} \fP^{\widetilde{X}}/W^{\widetilde{X}}\xrightarrow{\sim} \mathrm{Quant}(\CC[\widetilde{\OO}]) \qquad W^{\widetilde{X}} \cdot \lambda \mapsto \cA^{\widetilde{X}}_{\lambda}\end{equation}
In this section, we will re-interpret the spaces $\fP^{\widetilde{X}}$ and $\fP_k^{\widetilde{X}}$ in terms of purely Lie-theoretic information. 

Fix a birationally minimal induction datum $(L,\widetilde{\OO}_L)$ for $\widetilde{\OO}$ (cf. Proposition \ref{prop:propsofbind}(iii)) and let $\widetilde{X}_L := \Spec(\CC[\widetilde{\OO}_L])$. Choose a parabolic subgroup $P \subset G$ with Levi factor $L$ and consider the map
$$\widetilde{\mu}: \widetilde{Y} := G \times^P(\widetilde{X}_L \times \fp^{\perp}) \to \overline{\OO}$$
defined in Section \ref{subsec:binduction}. Since $\widetilde{\mu}^{-1} \simeq \widetilde{\OO}$, $\widetilde{\mu}$ factors through a partial resolution
\begin{equation}\label{eq:defofrho}\rho: \widetilde{Y} \to \widetilde{X}.\end{equation}
\begin{prop}\label{prop:nilpquant}
The following are true:
\begin{itemize}
    \item[(i)] The map (\ref{eq:defofrho}) is a $\QQ$-factorial terminalization.
    \item[(ii)] There is a linear isomorphism
    $$\eta: \fX(\fl \cap [\fg,\fg]) \xrightarrow{\sim} H^2(G/P,\CC) \xrightarrow{\sim} H^2(\widetilde{Y}^{\mathrm{reg}},\CC) =: \fP^{\widetilde{X}},$$
    where the second map is the pullback along the natural projection $\widetilde{Y}^{\mathrm{reg}} \to G/P.$
    \item[(iii)] $W^{\widetilde{X}}$ is identified with a normal subgroup of $N_G(L)/L$, with its canonical action on $\fX(\fl \cap [\fg,\fg])$. 
    \item[(iv)] Up to the action of $W^{\widetilde{X}}$ on the target, the map $\eta:\fX(\fl \cap [\fg,\fg]) \xrightarrow{\sim} \fP^{\widetilde{X}}$ is independent of the choice of parabolic $P$.
\end{itemize}
\end{prop}

\begin{proof}
(i) is \cite[Cor 4.3]{Mitya2020}. For orbits, (ii) and (iii) are \cite[Prop 4.7]{Losev4}. The proofs there can be easily generalized to arbitrary nilpotent covers. (iv) is \cite[Prop 7.2.5]{LMBM}.
\end{proof}

Combining (\ref{eq:nilpquant}) and (ii) of Proposition \ref{prop:nilpquant}, we obtain a natural bijection
$$\fX(\fl \cap [\fg,\fg])/W^{\widetilde{X}} \xrightarrow{\sim} \mathrm{Quant}(\CC[\widetilde{\OO}]) \qquad W^{\widetilde{X}} \cdot \lambda \mapsto \cA^{\widetilde{X}}_{\eta(\lambda)}$$
Note that $G$ acts on $\CC[\widetilde{\OO}]$ by graded Poisson automorphisms. There is a classical co-moment map $\varphi: \fg \to \CC[\widetilde{\OO}]$ obtained from the map of varieties $\widetilde{\OO} \to \fg^*$. The map $\eta: \fX(\fl \cap [\fg,\fg]) \xrightarrow{\sim} \fP^{\widetilde{X}}$ extends to an isomorphism (still denoted by $\eta$)
$$\eta: \fX(\fl) \xrightarrow{\sim} \fX(\fl \cap [\fg,\fg]) \oplus \fz(\fg)^* \xrightarrow{\sim} \fP^{\widetilde{X}} \oplus \fz(\fg)^* = \overline{\fP}^{\widetilde{X}}.$$
So by Proposition \ref{prop:Hamiltonian}, we obtain a natural bijection
$$\fX(\fl) \xrightarrow{\sim} \mathrm{Quant}^G(\CC[\widetilde{\OO}]) \qquad W^{\widetilde{X}} \cdot (\lambda, \chi) \mapsto (\cA_{\lambda}^{\widetilde{X}}, \Phi_{\chi}^{\widetilde{X}}).$$

\subsection{Description of partial Namikawa spaces}

In this seciton, we will give a Lie-theoretic description of the partial Namikawa spaces $\fP_k^{\widetilde{X}}$ (under some conditions). Passing to a covering group if necessary, we can (and will) assume that $G$ is simply connected. 
Let $R_x$ denote the reductive part of the stabilizer of $x \in \widetilde{\OO}$ and let $\mathfrak{r}$ be its Lie algebra. We note that $\fr$ does not depend on the choice of a point $x$, and the adjoint action of $R_x$ on $\fX(\mathfrak{r})$ factors through $R_x/R_x^{\circ} \simeq \pi^G_1(\widetilde{\OO})$. 

\begin{lemma}\label{lem:computeH2} 
The following are true:
\begin{itemize}
\item[(i)] Restriction along $\widetilde{\OO} \subset \widetilde{X}^{\mathrm{reg}}$ induces a linear isomorphism
$$\fP_0^{\widetilde{X}} = H^2(\widetilde{X}^{\mathrm{reg}},\CC) \xrightarrow{\sim} H^2(\widetilde{\OO},\CC)$$
\item[(ii)] There is a natural identification
$$H^2(\widetilde{\OO},\CC) \xrightarrow{\sim} \fX(\mathfrak{r})^{\pi_1(\widetilde{\OO})}$$
\end{itemize}
\end{lemma}

A description of $\mathfrak{r}$ can be found in \cite[Sec 6.1]{CM} (for classical types) and \cite[Sec 13.1]{Carter1993} (for exceptional types). 

\begin{rmk}\label{rmk:H2}
If $\widetilde{\OO} = \widehat{\OO}$ is the \emph{universal} cover of $\OO$, then $H^2(\widehat{\OO},\CC) \simeq \fz(\mathfrak{r})$ by Lemma \ref{lem:computeH2}. In particular,  $H^2(\widehat{\OO},\CC)=0$ if and only if $\mathfrak{r}$ is semisimple. On the other hand, if $\widetilde{\OO} = \OO$, then $H^2(\widetilde{\OO},\CC) \simeq \fz(\mathfrak{r})^{\pi_1(\OO)}$ was computed in every case by Biswas and Chatterjee in \cite{BISWAS}. 
\end{rmk}

Assume for the remainder of this subsection that $H^2(\widetilde{\OO},\CC)=0$ and fix a birationally minimal induction datum $(L,\widetilde{\OO}_L)$ for $\widetilde{\OO}$.  Suppose $Q \subset G$ is a parabolic subgroup with Levi factor $M$ and $\widetilde{\OO}_M$ is an $M$-equivariant nilpotent cover with $\widetilde{\OO}=\mathrm{Bind}^G_M \widetilde{\OO}_M$. The triple $(Q,M,\widetilde{\OO}_M)$ gives rise to a projective birational morphism (generalizing the map (\ref{eq:defofrho}))
\begin{equation}\label{eq:partialresolution}
\rho: G \times^Q (\Spec(\CC[\widetilde{\OO}_{M}]) \times \fq^{\perp}) \to \Spec(\CC[\widetilde{\OO}])\end{equation}
\begin{prop}[Prop 7.5.6, \cite{LMBM}]\label{prop:adaptedresolutiondatum}
For each codimension 2 leaf $\fL_k \subset \widetilde{X}$, there is a unique pair $(M_k,\widetilde{\OO}_{M_k})$ consisting of a Levi subgroup $M_k \subset G$ and a $M_k$-equivariant nilpotent cover $\widetilde{\OO}_{M_k}$ such that
\begin{itemize}
    \item[(i)] $L \subset M_k$.
    \item[(ii)] $\widetilde{\OO} = \mathrm{Bind}^G_{M_k} \widetilde{\OO}_{M_k}$.
    \item[(iii)] For any parabolic $Q \subset G$ with Levi factor $M_k$, the partial resolution (\ref{eq:partialresolution}) resolves $\Sigma_k$ and preserves $\Sigma_j$ for $j\neq k$.
\end{itemize}
\end{prop}

The pair $(M_k,\widetilde{\OO}_{M_k})$ appearing in Proposition \ref{prop:adaptedresolutiondatum} is called the $\fL_k$-\emph{adapted resolution datum.} 

\begin{prop}
Let $\fL_k \subset \widetilde{X}$ be a codimension 2 leaf and let $(M_k, \widetilde{\OO}_{M_k})$ be the $\fL_k$-adapted resolution datum. Then the following are true:

\begin{itemize}
\item[(i)] The closed embedding $\Sigma_k \subset \widetilde{X}$ (cf. Section \ref{subsec:geometrynilpotent}) lifts to a closed embedding $\mathfrak{S}_k \subset \widetilde{Z}_k$.
\item[(ii)] If $\mathcal{L}$ is a line bundle on $\widetilde{Z}_k$, then $\mathcal{L}|_{\mathfrak{S}_k}$ is a $\pi_1(\fL_k)$-equivariant line bundle on $\mathfrak{S}_k$, i.e. there is a restriction map
\begin{equation}\label{eq:restriction1}\Pic(\widetilde{Z}_k) \to \Pic(\mathfrak{S}_k)^{\pi_1(\fL_k)}\end{equation}
\item[(iii)] There are natural group isomorphisms
$$\Pic(\widetilde{Z}_k) \simeq \fX(M_k), \qquad \Pic(\mathfrak{S}_k) \simeq \Lambda_k,$$
i.e. (\ref{eq:restriction1}) induces a group homomorphism
\begin{equation}\label{eq:restriction2}
\fX(M_k) \to \Lambda_k^{\pi_1(\fL_k)}
\end{equation}
\item[(iv)] The complexification of (\ref{eq:restriction2}) is a linear isomorphism
\begin{equation}\label{eq:etak}\eta_k: \fX(\fm_k) \xrightarrow{\sim} \fP_k^{\widetilde{X}},\end{equation}
\item[(v)] The following diagram commutes
\begin{center}
    \begin{tikzcd}
      \fX(\fl) \ar[r,"\eta"]& \fP^{\widetilde{X}} \\
      \fX(\fm_k) \ar[u,hookrightarrow] \ar[r,"\eta_k"] & \fP_k^{\widetilde{X}} \ar[u, hookrightarrow]
    \end{tikzcd}
\end{center}
\end{itemize}
\end{prop}

\begin{proof}
(i) follows from (iii) of Proposition \ref{prop:adaptedresolutiondatum}. (iii) follows from \cite[Prop 7.1.2]{LMBM}. (ii), (iv), and (v) follow from \cite[Lem 7.5.7]{LMBM}.
\end{proof}

The isomorphism $\eta_k$ was computed in \cite[Sec 7.7]{LMBM} under certain conditions on $\widetilde{\OO}$ and $\fL_k$. The relevant statements will be recalled in Section \ref{subsubsec:etak}.

\subsection{Geometric characterization of birationally rigid covers}

Combining Proposition \ref{prop:partialdecomp} and (ii) of Proposition \ref{prop:nilpquant}, we obtain the following (purely geometric) characterization of birational rigidity. 

\begin{prop}\label{prop:criterionbirigidcover}
Let $\widetilde{\OO}$ be a nilpotent cover. Then $\widetilde{\OO}$ is birationally rigid if and only if the following conditions hold:
\begin{itemize}
    \item[(i)]  $H^2(\widetilde{\OO},\CC)=0$.
    \item[(ii)] $\Spec(\CC[\widetilde{\OO}])$ has no codimension 2 leaves.
\end{itemize}
\end{prop}

Checking condition (i) of Proposition \ref{prop:criterionbirigidcover} is usually easy in view of Lemma \ref{lem:computeH2}.  Checking condition (ii) is a subtler business in general. We will develop some techniques for doing so in Section \ref{subsec:semirigid}.

\subsection{Classification of birationally rigid orbits}\label{subsec:birigidorbits}

Let $\OO$ be a nilpotent orbit. 

\begin{prop}\label{prop:criterionbirigidorbit}
Let $\OO$ be a nilpotent orit. Then $\OO$ is birationally rigid if and only if the following conditions are satisfied:
\begin{itemize}
    \item[(i)] $H^2(\OO,\CC)=0$.
    \item[(ii)] All dimension 2 singularities in $\overline{\OO}$ are of type $m$.
\end{itemize}
\end{prop}

\begin{proof}
By Lemma \ref{lem:weaklynormal}, condition (ii) above is equivalent to condition (ii) of Proposition \ref{prop:criterionbirigidcover}. Now Proposition \ref{prop:criterionbirigidorbit} follows at once from Proposition \ref{prop:criterionbirigidcover}.
\end{proof}

An advantage of this formulation is that condition (ii) above is very easy to check. In classical types, 
there are no singularities of type $m$. So (ii) is equivalent to the condition that there are no codimension 2 orbits in $\overline{\OO}$. The set of codimension 2 orbits in $\overline{\OO}$ was described by Kraft and Procesi in  \cite{Kraft-Procesi} in terms of the partition corresponding to $\OO$. From this description, one easily deduces the following.

\begin{prop}[Prop 7.6.3, \cite{LMBM}]\label{prop:birigidorbitclassical}
Suppose $\fg$ is classical and let $\mathbb{O} \subset \fg^*$ be a nilpotent orbit corresponding to a partition $p$. Then $\OO$ is birationally rigid if and only if one of the following is true:
\begin{itemize}
    \item[(i)] $\fg = \mathfrak{sl}(n)$ and $\mathbb{O} = \{0\}$.
    \item[(ii)] $\fg = \mathfrak{so}(2n+1)$ or $\mathfrak{sp}(2n)$ and $p$ satisfies
    $$p_i \leq p_{i+1}+1 \qquad \forall i.$$
    \item[(iii)] $\fg = \mathfrak{so}(2n)$, $p$ satisfies
    $$p_i \leq p_{i+1} +1 \qquad \forall i,$$
    and $p$ is not of the form $(2^m,1^2)$ for some $m$. 
\end{itemize}
\end{prop}

In exceptional types, condition (ii) of Proposition \ref{prop:criterionbirigidorbit} can be checked by inspecting the incidence tables in \cite[Sec 13]{fuetal2015}. One easily arrives at the following classification.

\begin{prop}\label{prop:listofbirigid}
The following is a complete list of birationally rigid orbits in simple exceptional Lie algebras:

\begin{center}
    \begin{tabular}{|c|l|} \hline
        $\fg$ & Birationally rigid orbits \\ \hline
         $G_2$ & $\{0\}$, $A_1$, $\widetilde{A}_1$\\ \hline
         $F_4$ & $\{0\}$, $A_1$, $\widetilde{A}_1$, $A_1+\widetilde{A}_1$, $A_2+\widetilde{A}_1$, $\widetilde{A}_2+A_1$ \\ \hline
         $E_6$ & $\{0\}$, $A_1$, $3A_1$, $2A_2+A_1$ \\ \hline
         $E_7$ & $\{0\}$, $A_1$, $2A_1$, $(3A_1)'$, $4A_1$, $A_2+A_1$, $A_2+2A_1$, $2A_1+A_1$, $(A_3+A_1)'$, $A_4+A_1$\\ \hline
         $E_8$ & $\{0\}$, $A_1$, $2A_1$, $3A_1$, $4A_1$, $A_2+A_1$, $A_2+2A_1$, $A_2+3A_1$, $2A_2+A_1$, $A_3+A_1$, $2A_2+2A_1$,\\ 
         & $A_3+2A_1$, $D_4(a_1)+A_1$, $A_3+A_2+A_1$, $A_4+A_1$, $2A_3$, $A_4+A_3$, $A_5+A_1$, $D_5(a_1)+A_2$\\ \hline
    \end{tabular}
\end{center}
\vspace{3mm}
Three of these orbits are not rigid, namely: 
$$A_2+A_1, A_4+A_1 \subset E_7, \qquad A_4+A_1 \subset E_8.$$
\end{prop}

\begin{proof}
In exceptional types, the cohomology groups $H^2(\OO,\CC)$ were computed by Biswas and Chatterjee in \cite[Thms 5.11,5.12]{BISWAS}. It was shown there that $H^2(\OO,\CC)=0$ in all cases except for the following nine orbits in type $E_6$:
$$2A_1, A_2+A_1, A_2+2A_1, A_3, A_3+A_1, A_4, A_4+A_1, A_5, D_5(a_1).$$
Thus by Proposition \ref{prop:criterionbirigidorbit}, $\OO$ is birationally rigid if and only if $\OO$ is not one of these nine and all dimension 2 singularities in $\overline{\OO}$ are of type $m$. Inspecting the diagrams in \cite[Sec 13]{fuetal2015}, one arrives at the list given in the statement of the proposition. A list of rigid orbits in exceptional types is provided in \cite{Elashvili}.
\end{proof}

\subsection{Classification of birationally semi-rigid orbits}\label{subsec:semirigid}

For the calculations in Section \ref{sec:inflchars}, birationally rigid nilpotent covers, and the orbits which admit them, will play a central role. Make the following definition.

\begin{definition}
A nilpotent cover $\OO$ is \emph{birationally semi-rigid} if 
\begin{itemize}
    \item[(i)] $\OO$ admits a $G$-equivariant birationally rigid cover.
    \item[(ii)] $\OO$ is not birationally rigid.
\end{itemize}
\end{definition}

Below, we will give a classification of such orbits in simple exceptional types. The following result from \cite{LMBM} narrows the range of possibilities.

\begin{prop}[Prop 7.6.16,\cite{LMBM}]\label{prop:A1}
Suppose $\OO$ is a birationally semi-rigid orbit in a simple exceptional Lie algebra. Then all Kleinian singularities in $\Spec(\CC[\OO])$ are of type $A_1$, with the following four exceptions:
\begin{itemize}
    \item[(i)] $\fg = E_6$ and $\mathbb{O} = 2A_2$. There is a unique codimension 2 leaf, and the corresponding singularity is of type $A_2$.
    \item[(ii)] $\fg=E_6$ and $\mathbb{O}=A_5$. There is a unique codimension 2 leaf, and the corresponding singularity is of type $A_2$.
    \item[(iii)] $\fg=E_6$ and $\mathbb{O}=E_6(a_3)$. There are two codimension 2 leaves, and the corresponding singularities are of types $A_1$ and $A_2$.
    \item[(iv)] $\fg=E_8$ and $\mathbb{O} = E_8(b_6)$. There are two codimension 2 leaves, and the corresponding singularities are of types $A_1$ and $A_2$.
\end{itemize}
\end{prop}

In the simple exceptional Lie algebras, there are 38 nilpotent orbits with nontrivial $\pi_1(\OO)$ which satisfy the $A_1$ condition above. However, not all such orbits are birationally semi-rigid (the simplest example is the distinguished orbit $F_4(a_2)$ in $F_4$). Our task in this subsection is to determine precisely which of them are. For the most part, the techniques we will employ were developed in \cite{LMBM}. We will review some of them here for the reader's convenience.

If $\OO$ is a nilpotent orbit, consider the finite set
\begin{equation}\label{eq:Prig}\mathcal{P}_{\mathrm{rig}}(\OO) := \{(M,\OO_M) \mid \OO = \Ind^G_M \OO_M, \ \OO_M \text{ is rigid}\}/G.\end{equation}
Let $m(\OO)$ denote the maximum value of $\dim \fz(\fm)$ for $(M,\OO_M) \in \mathcal{P}_{\mathrm{rig}}(\OO)$. 

\begin{lemma}[Prop 7.6.15, \cite{LMBM}]\label{lem:criterion2leafless}
Let $\fL \subset \Spec(\CC[\OO])$ be a codimension 2 leaf. Let $\Sigma = \CC^2/\Gamma$ be the corresponding Kleinian singularity, and $\OO' \subset \overline{\OO}$ the corresponding codimension two $G$-orbit. Assume:
\begin{itemize}
    \item[(i)]  $\Gamma$ is a simple group.
    \item[(ii)] There is a strict inequality
    $$|\pi_1(\OO)||\pi_1(\OO')|^{-1} > m(\OO)$$
\end{itemize}
Then $\Sigma \subset \Spec(\CC[\OO])$ is smoothened under the covering map $\Spec(\CC[\widehat{\OO}]) \to \Spec(\CC[\OO])$. 
\end{lemma}

Combining Lemma \ref{lem:criterion2leafless} and Proposition \ref{prop:criterionbirigidcover}, we obtain the following useful criterion.

\begin{cor}\label{cor:criterionbirigidcover}
Suppose
\begin{itemize}
    \item[(i)] All Kleinian singularities in $\Spec(\CC[\OO])$ are of type $A_1$.
    \item[(ii)] For each codimension 2 orbit $\OO' \subset \overline{\OO}$, there is a strict inequality
    $$|\pi_1(\OO)||\pi_1(\OO')|^{-1} > m(\OO)$$
    \item[(iii)] The reductive part of the centralizer of $e \in \OO$ is semisimple.
\end{itemize}
Then the universal cover $\widehat{\OO}$ of $\OO$ is birationally rigid.
\end{cor}

\begin{proof}
By Lemma \ref{lem:criterion2leafless}, conditions (i) and (ii) imply that $\Spec(\CC[\widehat{\OO}])$
has no codimension 2 leaves. By Lemma \ref{lem:computeH2}, condition (iii) implies that $H^2(\widehat{\OO},\CC)=0$. The corollary follows at once from Proposition \ref{prop:criterionbirigidcover}.
\end{proof}

\begin{prop}\label{prop:listofsemirigid}
The following is a complete list of birationally semi-rigid orbits in simple exceptional Lie algebras:
\begin{center}
    \begin{tabular}{|c|l|} \hline
        $\fg$ & Birationally semi-rigid orbits \\ \hline
         $G_2$ & $G_2(a_1)$\\ \hline
         $F_4$ & $A_2$, $B_2$, $C_3(a_1)$, $F_4(a_3)$,  \\ \hline
         $E_6$ & $A_2$, $D_4(a_1)$, $2A_2$, $A_5$, $E_6(a_3)$ \\ \hline
         $E_7$ & $(3A_1)''$, $A_2$, $A_2+3A_1$, $(A_3+A_1)''$, $D_4(a_1)$, $A_3+2A_1$, $D_4(a_1)+A_1$,\\ 
         
         &  $A_3+A_2+A_1$, $A_5+A_1$, $D_5(a_1)+A_1$, $E_7(a_5)$, $E_7(a_4)$\\ \hline
         
         $E_8$ & $A_2$, $2A_2$, $D_4(a_1)$, $D_4(a_1)+A_2$, $D_4+A_2$, $D_6(a_2)$, $E_6(a_3)+A_1$,\\ 
         & $E_7(a_5)$, $E_8(a_7)$, $E_8(b_6)$\\ \hline
    \end{tabular}
\end{center}
\end{prop}

\begin{proof}
By Proposition \ref{prop:A1}, we can restrict our attention to orbits with nontrivial $\pi_1(\OO)$ and only $A_1$ singularities in $\Spec(\CC[\OO])$. Using the incidence diagrams in \cite[Sec 13]{fuetal2015}, we find that there are 38 orbits with these properties. For most of these orbits, Corollary \ref{cor:criterionbirigidcover} can be straightforwardly applied to show that $\OO$ is birationally semi-rigid. In some cases, a more elaborate argument is required. On the other hand, 10 of these orbits are \emph{not} semi-rigid. This is proved either by cohomology considerations, see Proposition \ref{prop:criterionbirigidcover}, or by a counting argument involving Proposition \ref{prop:propsofbind}.

We pause to introduce some notational conventions which will remain in place for the rest of the paper. For $\fg$ a simple exceptional Lie algebra, we number the simple roots as in Bourbaki (\cite{Bourbaki46}). In type $E_n$, this means that the simple roots forming the subgraph of type $A_{n-1}$ are labeled $\alpha_1,\alpha_3,...,\alpha_n$ from left to right and the remaining simple root is labeled $\alpha_2$. In type $F_4$ the simple roots are labeled $\alpha_1,...,\alpha_4$ from left to right ($\alpha_1$, $\alpha_2$ are the long roots, and $\alpha_3$, $\alpha_4$ are short). In type $G_2$, $\alpha_1$ is the short root. 

If $\fg$ has rank $n$ and $\{r_1,...,r_p\} \subseteq \{1,...,n\}$, there is a unique standard Levi subalgebra with simple roots $\alpha_{r_1},...,\alpha_{r_p}$. We denote this Levi subalgebra by $\fl(X;r_1,...,r_p)$, where $X$ is the Lie type of the Levi ($X$ is included in the notation only for the reader's convenience---it is completely determined by the numbers $r_1,...,r_p$). Two standard Levis $\fl(X;r_1,...,r_p)$ and $\fl(Y;s_1,...,s_q)$ are conjugate under $\Ad(\fg)$ if and only if $p=q$ and there is a Weyl group element $w$ such that $\{\alpha_{r_1},...,\alpha_{r_p}\} = w\{\alpha_{s_1},...,\alpha_{s_p}\}$. Sometimes, the Levi subalgebra $\fl(X;r_1,...,r_p)$ is completely determined by $X$. In such cases, we will often omit $r_1,...,r_p$ from the notation, writing simply $\fl(X)$. 

The calculations below involve a number of elementary `micro-computations', which are carried out in each case in exactly the same fashion. To avoid repeating references and explanations, we will catalogue them below:
\begin{itemize}
    \item Given a nilpotent orbit $\OO$, determine the finite set $\mathcal{P}_{\mathrm{rig}}(\OO)$, see (\ref{eq:Prig}). This is deducible in every case from the tables in \cite[Sec 4]{deGraafElashvili}.
    \item Given a nilpotent orbit $\OO$, determine the codimension 2 orbits $\OO_k \subset \overline{\OO}$. This is evident from the incidence diagrams in \cite{Spaltenstein}. In some cases, we will also need to determine the singularity of $\OO_k$ and its normalization, $\Sigma_k$. This can be deduced from the incidence diagrams in \cite[Sec 13]{fuetal2015}.
    \item Determine the fundamental group $\pi_1(\OO)$ of a nilpotent orbit $\OO$. See \cite[Sec 6.1]{CM} for classical types and \cite[Sec 8.4]{CM} for exceptional types.
    \item Determine the reductive part $\mathfrak{r}$ of the centralizer of a nilpotent element $e \in \OO$. See \cite[Sec 6.1]{CM} for classical types and \cite[Sec 13.1]{Carter1993} for exceptional types.
\end{itemize}

\vspace{3mm}
\noindent \underline{$G_2(a_1) \subset G_2$}. We have
    $$\mathcal{P}_{\mathrm{rig}}(\OO) = \{(L(A_1;1),\{0\}), (L(A_1;2), \{0\})\}$$
    and therefore $m(\OO)=1$. There is one codimension 2 orbit in $\overline{\OO}$, namely $\OO_1= \widetilde{A}_1$. The fundamental groups are as follows
    $$\pi_1(\OO) = S_3 \qquad \pi_1(\OO_1) = 1$$
    Note that
$$|\pi_1(\OO)||\pi_1(\OO_1)|^{-1}=6>1=m(\OO).$$
and $\mathfrak{r} = \{0\}$. So $\widehat{\OO}$ is birationally rigid by Corollary \ref{cor:criterionbirigidcover}.

\vspace{3mm}
\noindent \underline{$A_2 \subset F_4$}. We have
    $$\mathcal{P}_{\mathrm{rig}}(\OO) = \{(L(C_3;2,3,4),\{0\}\},$$
    and therefore $m(\OO)=1$. There is one codimension 2 orbit in $\overline{\OO}$, namely $\OO_1 = A_1+\widetilde{A}_1$. The fundamental groups are as follows 
    $$\pi_1(\OO) = S_2 \qquad \pi_1(\OO_1) = 1$$
Note that
$$|\pi_1(\OO)||\pi_1(\OO_1)|^{-1}=2>1=m(\OO),$$
and $\mathfrak{r}=A_2$. So $\widehat{\OO}$ is birationally rigid by Proposition \ref{cor:criterionbirigidcover}.

\vspace{3mm}
\noindent \underline{$B_2 \subset F_4$}. We have
    $$\mathcal{P}_{\mathrm{rig}}(\OO) = \{(L(C_3;2,3,4),\OO_{(2,1^4)})\},$$
    and therefore $m(\OO)=1$. There is one codimension 2 orbit in $\overline{\OO}$, namely $\OO_1 = A_2+\widetilde{A}_1$. The fundamental groups are as follows 
    $$\pi_1(\OO) = S_2 \qquad \pi_1(\OO_1) = 1$$
Note that
$$|\pi_1(\OO)||\pi_1(\OO_1)|^{-1}=2>1=m(\OO),$$
and $\mathfrak{r}=2A_1$. So $\widehat{\OO}$ is birationally rigid by Proposition \ref{cor:criterionbirigidcover}.

\vspace{3mm}
\noindent \underline{$C_3(a_1) \subset F_4$}. We have
    $$\mathcal{P}_{\mathrm{rig}}(\OO) = \{(L(B_3;1,2,3),\OO_{(2^2,1^3)})\},$$
    and therefore $m(\OO)=1$. There are two codimension 2 orbits in $\overline{\OO}$, namely $\OO_1 = B_2$ and $\OO_2=\widetilde{A}_2+A_1$. The singularity of $\OO_2\subset \overline{\OO}$ is of type $m$ and is therefore resolved under the normalization map $X=\Spec(\CC[\OO])\to \overline{\OO}$. Let $\widecheck{\OO}_1\subset X$ be the preimage of $\OO_1$. The slice $\Sigma_1$ to $\OO_1$ is of type $2A_1$, and therefore the preimage of the $\Sigma_1$ is the disjoint union of two copies of $\Sigma_1$. If $\widecheck{\OO}_1$ is not connected, then there are at least two symplectic leaves of codimension $2$ over $\OO_1$ with open $G$ orbits being the irreducible components of $\widecheck{\OO}_1$. That contradicts to \cref{lem:weaklynormal}. It follows that $\widecheck{\OO}_1$ is a $2$-fold connected cover of $\OO_1$. The fundamental groups are as follows 
    $$\pi_1(\OO) = S_2 \qquad \pi_1(\OO_1) = S_2 \qquad \pi_1(\widecheck{\OO}_1)=1$$
Note that
$$|\pi_1(\OO)||\pi_1(\widecheck{\OO}_1)|^{-1}=2>1=m(\OO),$$
and $\mathfrak{r}=A_1$. Analogously to Proposition \ref{cor:criterionbirigidcover}, we see that $\widehat{\OO}$ is birationally rigid.
    
\vspace{3mm}
\noindent \underline{$F_4(a_3) \subset F_4$}. We have
    $$\mathcal{P}_{\mathrm{rig}}(\OO) = \{(L(A_1+A_2;1,3,4), \{0\}), (L(A_2+A_1;1,2,4), \{0\}), (L(B_2), \{0\})\},$$
    and therefore $m(\OO)=2$. There is one codimension 2 orbit in $\overline{\OO}$, namely $\OO_1 = C_3(a_1)$. The fundamental groups are as follows
    $$\pi_1(\OO) = S_4 \qquad \pi_1(\OO_1) = S_2$$
    Note that
    $$|\pi_1(\OO)||\pi_1(\OO_1)|^{-1}=12>2 = m(\OO),$$
    and $\mathfrak{r}=0$. So $\widehat{\OO}$ is birationally rigid by Proposition \ref{cor:criterionbirigidcover}.
    
\vspace{3mm}
\noindent \underline{$F_4(a_2) \subset F_4$}. Note that $\pi_1(\OO)=S_2$. Hence, $\OO$ admits 2 non-isomorphic covers (including the trivial one). We have
    $$\mathcal{P}_{\mathrm{rig}}(\OO) = \{(L(2A_1;1,4),\{0\}),(L(B_2;2,3),\OO_{(2^2,1)})\}.$$
    By (iii) of Proposition \ref{prop:propsofbind}, $\mathrm{Bind}^G_{L(2A_1;1,4)} \{0\}$ and $\mathrm{Bind}^G_{L(B_2;2,3)} \OO_{(2^2,1)}$ are  non-isomorphic covers of $\OO$. In particular, all covers of $\OO$ are birationally induced.

\vspace{3mm}
\noindent \underline{$A_2 \subset E_6$}. We have
    $$\mathcal{P}_{\mathrm{rig}}(\OO) = \{(L(A_5),\{0\})\},$$
    and therefore $m(\OO)=1$. There is one codimension 2 orbit in $\overline{\OO}$, namely $\OO_1 = 3A_1$. The fundamental groups are as follows
    $$\pi_1(\OO) = S_2 \qquad \pi_1(\OO_1) = 1.$$
    Note that
    $$|\pi_1(\OO)||\pi_1(\OO_1)|^{-1}=2>1= m(\OO),$$
    and $\mathfrak{r}=2A_2$. So  $\widehat{\OO}$ is birationally rigid by Corollary \ref{cor:criterionbirigidcover}.
    
\vspace{3mm}
\noindent \underline{$D_4(a_1) \subset E_6$}.  We have
    $$\mathcal{P}_{\mathrm{rig}}(\OO) = \{(L(2A_2+A_1),\{0\}), (L(A_3+A_1;1,2,4,5),\{0\}), (L(D_4),\OO_{(2^2, 1^4)})\},$$
    and therefore $m(\OO)=2$. Consider the set $\mathcal{P}_1(\OO)$ of pairs $(M,\OO_M)$ consisting of a Levi subgroup $M \subset G$ of semisimple co-rank 1 and a nilpotent $M$-orbit $\OO_M$ such that $\OO=\Ind^G_M \OO_M$, considered up to $G$-conjugacy. Note that $(M,\OO_M) \in \mathcal{P}_1(\OO)$ if and only if there is a pair $(L,\OO_L) \in \mathcal{P}_{\mathrm{rig}}(\OO)$ such that $L$ is (conjugate to) a subgroup of $M$ and $\OO_M = \Ind^M_L \OO_L$. Thus we have
    \begin{table}[H]
        \begin{tabular}{|c|c|c|} \hline
           $M$  &  $\OO_M$ & $\pi_1^M(\OO_M)$ \\ \hline
           $L(2A_2+A_1)$ & $\{0\}$ & $1$ \\ \hline
           $L(D_5)$ & $\OO_{(3^2,1^4)}$ & $\ZZ_2$ \\ \hline
           $L(A_4+A_1)$ & $\OO_{(2,1^3)} \times \{0\}$ & $1$ \\ \hline
           $L(A_5)$ & $\OO_{(2^2,1^2)}$ & $1$ \\ \hline
        \end{tabular}
        \caption{$\mathcal{P}_1(\OO)$}
    \end{table}
    If $\widetilde{\OO}$ is a birationally induced cover of $\OO$, then $\widetilde{\OO}=\mathrm{Bind}^G_M \widetilde{\OO}_M$ for some pair  $(M,\OO_M) \in \mathcal{P}_1(\OO)$ and $M$-equivariant cover $\widetilde{\OO}_M$ of $\OO_M$. By the table above, there are 5 such $(M,\widetilde{\OO}_M)$, up to conjugation by $G$. However, the pairs $(L(D_5),\OO_{(3^2,1^4)}), (L(A_4+A_1), \OO_{(2,1^3)} \times \{0\}), (L(A_5), \OO_{(2^2,1^2)}) \in \mathcal{P}_1(\OO)$ are induced from a common element of $\mathcal{P}_{\mathrm{rig}}(\OO)$, namely $(L(A_3+A_1;1,2,4,5),\{0\})$. Hence three of the five $(M,\widetilde{\OO}_M)$ give rise to isomorphic covers of $\OO$. It follows that there are at most 3 non-isomorphic birationally induced covers of $\OO$. Since $\pi_1(\OO) \simeq S_3$, there are a total of 4 non-isomorphic covers of $\OO$. So at least one such is birationally rigid.

\vspace{3mm}
\noindent \underline{$(3A_1)'' \subset E_7$}. We have
    $$\mathcal{P}_{\mathrm{rig}}(\OO) = \{(L(E_6),\{0\})\}$$
    and therefore $m(\OO)=1$. There is one codimension 2 orbit in $\overline{\OO}$, namely $\OO_1 =2A_1$. The fundamental groups are as follows
    $$\pi_1(\OO)=\ZZ_2 \qquad \pi_1(\OO_1) =1.$$
    Note that
    $$|\pi_1(\OO)||\pi_1(\OO_1)|^{-1} = 2 > 1 = m(\OO),$$
   and $\mathfrak{r} = F_4$. So $\widehat{\OO}$ is birationally rigid by Corollary \ref{cor:criterionbirigidcover}.

\vspace{3mm}
\noindent \underline{$A_2 \subset E_7$}. We have
    $$\mathcal{P}_{\mathrm{rig}}(\OO) = \{(L(D_6),\{0\})\},$$
    and therefore $m(\OO)=1$. There is one codimension 2 orbit in $\overline{\OO}$, namely $\OO_1 = (3A_1)'$. The fundamental groups are as follows 
    $$\pi_1(\OO) = S_2 \qquad \pi_1(\OO_1) = 1.$$
    Note that
$$|\pi_1(\OO)||\pi_1(\OO_1)|^{-1}=2>1=m(\OO),$$
and $\mathfrak{r} = A_5$. So $\widehat{\OO}$ is birationally rigid by Corollary \ref{cor:criterionbirigidcover}.

\vspace{3mm}
\noindent \underline{$A_2+3A_1 \subset E_7$}. We have
    $$\mathcal{P}_{\mathrm{rig}}(\OO) = \{(L(A_6),\{0\})\},$$
    and therefore $m(\OO)=1$. There is one codimension 2 orbit in $\overline{\OO}$, namely $\OO_1 = A_2+2A_1$. The fundamental groups are as follows
    $$\pi_1(\OO) = \ZZ_2 \qquad \pi_1(\OO_1) = 1.$$ 
    Note that
$$|\pi_1(\OO)||\pi_1(\OO_1)|^{-1}=2>1=m(\OO),$$
and $\mathfrak{r} = G_2$. So $\widehat{\OO}$ is birationally rigid by Corollary \ref{cor:criterionbirigidcover}. 

\vspace{3mm}
\noindent \underline{$(A_3+A_1)'' \subset E_7$}. We have
    $$\mathcal{P}_{\mathrm{rig}}(\OO) = \{(L(D_5), \{0\})\},$$
    and therefore $m(\OO)=2$. There are two codimension 2 orbits in $\overline{\OO}$, namely $\OO_1 = A_3$, and $\OO_2=2A_2$. The fundamental groups are as follows
    $$\pi_1(\OO) = \ZZ_2 \qquad \pi_1(\OO_1) =1 \qquad \pi_1(\OO_2) =1.$$
    Let $\widehat{\OO}$ be the universal cover of $\OO$. 
    
    We will show that the singularity $\Sigma_1$ of $X$ is resolved under the map $\widehat{X}\to X$. Note that $|\pi_1(\OO)||\pi_1(\OO_1)|^{-1}=2=m(\OO)$. So if $\Sigma_1$ is not resolved, both $\OO$ and $\widehat{\OO}$ are birationally induced from a co-rank 2 Levi. Since $\mathcal{P}_{\mathrm{rig}}(\OO)$ contains a single Levi (of co-rank 2), there is only one cover of $\OO$ which can be induced from a corank 2 Levi, namely $\mathrm{Bind}_L^G(\{0\})$. We conclude that $\Sigma_1$ is resolved under the map $\widehat{X}\to X$. Analogously, $\Sigma_2$ is resolved.
    We have $\mathfrak{r}=B_3$, and therefore $\widehat{\OO}$ is birationally rigid by Corollary \ref{cor:criterionbirigidcover}.

\vspace{3mm}
\noindent \underline{$D_4(a_1) \subset E_7$}. We have
    $$\mathcal{P}_{\mathrm{rig}}(\OO) = \{(L(A_1+A_5),\{0\}), (L(D_6),\OO_{(2^4,1^4)})\},$$
    and therefore $m(\OO)=1$. There is one codimension 2 orbit in $\overline{\OO}$, namely $\OO_1 = (A_3+A_1)'$. The fundamental groups are as follows
    $$\pi_1(\OO) = S_3 \qquad \pi_1(\OO_1) = 1.$$
    Note that
$$|\pi_1(\OO)||\pi_1(\OO_1)|^{-1}=6>1=m(\OO),$$
and $\mathfrak{r}=3A_1$. So $\widehat{\OO}$ is birationally rigid by Corollary \ref{cor:criterionbirigidcover}.

\vspace{3mm} 
\noindent \underline{$A_3+2A_1 \subset E_7$}. We have
    $$\mathcal{P}_{\mathrm{rig}}(\OO) = \{(L(E_6),3A_1)\},$$
    and therefore $m(\OO)=1$. There is one codimension 2 orbit in $\overline{\OO}$, namely $\OO_1 = (A_3+A_1)'$. The fundamental groups are as follows
    $$\pi_1(\OO) = \ZZ_2 \qquad \pi_1(\OO_1) =1.$$
    Note that
$$|\pi_1(\OO)||\pi_1(\OO_1)|^{-1}=2>1=m(\OO),$$
and $\mathfrak{r}=2A_1$. So $\widehat{\OO}$ is birationally rigid by Corollary \ref{cor:criterionbirigidcover}.

\vspace{3mm}
\noindent \underline{$D_4(a_1)+A_1 \subset E_7$}. We have
    $$\mathcal{P}_{\mathrm{rig}}(\OO) = \{(L(A_5; 1,3,4,5,6),\{0\})\},$$
    and therefore $m(\OO)=2$. There are two codimension 2 orbits in $\overline{\OO}$, namely $\OO_1 = D_4(a_1)$ and $\OO_2=A_3+2A_1$. Consider the map $X=\Spec(\CC[\OO])\to \overline{\OO}$. Let $\widecheck{\OO}_1\subset X$ be the preimage of $\OO_1$. The slice $\Sigma_1$ to $\OO_1$ is of type $3A_1$, and therefore the preimage of $\Sigma_1$ is a disjoint union of three copies of $A_1$singularities. The singularity $\Sigma_2$ of the leaf $\fL_2$ corresponding to $\OO_2$ is of type $A_1$. If $\widecheck{\OO}_1$ is not connected, then there are at least two symplectic leaves of codimension $2$ over $\OO_1$ with open $G$ orbits being the irreducible components of $\widecheck{\OO}_1$. That contradicts to \cref{lem:weaklynormal}. It follows that $\widecheck{\OO}_1$ is a $3$-fold connected cover of $\OO_1$. The fundamental groups are as follows 
    $$\pi_1(\OO) = S_2\times \ZZ_2 \qquad \pi_1(\OO_1) = S_3 \qquad \pi_1(\widecheck{\OO}_1)=\ZZ_2 \qquad \pi_1(\OO_2)=\ZZ_2$$
Note that
$$|\pi_1(\OO)||\pi_1(\widecheck{\OO}_1)|^{-1}=2=m(\OO), \qquad |\pi_1(\OO)||\pi_1({\OO}_2)|^{-1}=2=m(\OO).$$

Let $\widehat{\OO}\to \OO$ be the universal cover of $\OO$. If one of the singularities $\Sigma_1$, $\Sigma_2$ is not resolved under $\widehat{X}\to X$, then we have $\dim \fP^X=2$ and $\dim \fP^{\widehat{X}}\ge 2$. Thus, both $\OO$ and $\widehat{\OO}$ are birationally induced from a corank $2$ Levi. However, there is only one cover that can be induced from a corank $2$ Levi, namely $\mathrm{Bind}_L^G(\{0\})$. Therefore, both $\Sigma_1$ and $\Sigma_2$ are resolved under $\widehat{X}\to X$.
    We have $\mathfrak{r}=2A_1$, and therefore $\widehat{\OO}$ is birationally rigid by Corollary \ref{cor:criterionbirigidcover}.

\vspace{3mm}
\noindent \underline{$A_3+A_2 \subset E_7$}. Note that $\pi_1(\OO)=S_2$. Hence, $\OO$ has two non-isomorphic covers (including the trivial one). We have
    $$\mathcal{P}_{\mathrm{rig}}(\OO) = \{(L(A_1+D_5), \{0\} \times \OO_{(2^2,1^6)}), (L(D_6), \OO_{(3,2^2,1^5)})\}$$
    By (iii) of Proposition \ref{prop:propsofbind}, $\mathrm{Bind}^G_{L(A_1+D_5)} \{0\}$ and $\mathrm{Bind}^G_{L(D_6)} \OO_{(3,2^2,1^5)}$ are non-isomorphic covers of $\OO$. In particular, all covers of $\OO$ are birationally induced.
    
\vspace{3mm}
\noindent \underline{$A_3+A_2+A_1 \subset E_7$}. We have
    $$\mathcal{P}_{\mathrm{rig}}(\OO) = \{(L(A_4+A_2; 1,2,3,4,6,7),\{0\})\}.$$
    Since $L=L(A_4+A_2; 1,2,3,4,6,7)$ is of semisimple co-rank 1 and $\OO_L = \{0\}$, there is a unique birationally induced cover of $\OO$, namely $\mathrm{Bind}_L^G\{0\}$. On the other hand, since $\pi_1(\OO)=\ZZ_2$, there are two non-isomorphic covers of $\OO$. So one must be birationally rigid.

\vspace{3mm} 
\noindent \underline{$D_4+A_1 \subset E_7$}. Note that $\pi_1(\OO) = \ZZ_2$. Hence, $\OO$ has two non-isomorphic covers (including the trivial one). We have
    $$\mathcal{P}_{\mathrm{rig}}(\OO) = \{(L(D_6),\OO_{(3,2^4,1)})\}$$
    Write $(L,\OO_L) = (L(D_6),\OO_{(3,2^4,1)})$. An {\fontfamily{cmtt}\selectfont atlas} computation shows that $\pi_1^L(\OO_L) = \ZZ_2$. Let $\widetilde{\OO}_L$ denote the 2-fold $L$-equivariant cover of $\OO_L$. By (iv) of Proposition \ref{prop:propsofbind}, $\mathrm{Bind}^G_L \widetilde{\OO}_L$ is a 2-fold cover of $\mathrm{Bind}^G_L \OO_L$. In particular, both covers of $\OO$ are birationally induced. 

\vspace{3mm}
\noindent \underline{$A_5+A_1 \subset E_7$}. We have
    $$\mathcal{P}_{\mathrm{rig}}(\OO) = \{(L(E_6),2A_2+A_1)\},$$
    and therefore $m(\OO)=1$. There is one codimension 2 orbit in $\overline{\OO}$, namely $\OO_1 = A_4+A_2$. The fundamental groups are as follows
    $$\pi_1(\OO) = \ZZ_2 \qquad \pi_1(\OO_1) = 1.$$
    Note that
$$|\pi_1(\OO)||\pi_1(\OO_1)|^{-1}=2>1=m(\OO),$$
and $\mathfrak{r} = A_1$. So $\widehat{\OO}$ is birationally rigid by Corollary \ref{cor:criterionbirigidcover}.

\vspace{3mm}
\noindent \underline{$D_5(a_1) +A_1 \subset E_7$}. Note that $\pi_1(\OO) = \ZZ_2$. Hence, $\OO$ has two non-isomorphic covers (including the trivial one). We have
    $$\mathcal{P}_{\mathrm{rig}}(\OO) = \{(L(A_3+A_2),\{0\}) \}$$
    We note that the induction from $(L(A_3+A_2),\{0\})$ is birational (there are several ways to see this --- one will be given in the $D_5(a_1) +A_1 \subset E_7$ portion of Section \ref{subsec:inflcharexceptional}).  Recall the set $\mathcal{P}_1(\OO)$ defined above. Since $\OO$ is birationally induced from $(L(A_3+A_2),\{0\})$, every pair $(M,\OO_M) \in \mathcal{P}_1(\OO)$ is of the form $\OO_M = \Ind^M_L \{0\}$, where $M$ is a co-rank 1 Levi containing (a $G$-conjugate of) $L$. Thus we have
    \begin{table}[H]
    \begin{tabular}{|c|c|c|} \hline
       $M$ &  $\OO_M$ & $\pi_1^M(\OO_M)$ \\ \hline 
     $A_6$ & $\OO_{(2^3,1^3)}$ & 1\\ \hline
     $D_6$ & $\OO_{(3^3,1^3)}$ & 1 \\ \hline
     $A_4+A_2$ & $\OO_{(2,1^3)}\times \{0\}$ & 1 \\ \hline
     $A_3+A_2+A_1$ & $\{0\}\times \{0\}\times \OO_{(2)}$ & 1 \\ \hline
    \end{tabular}
    \caption{$\mathcal{P}_1(\OO)$}
\end{table}

Since $\OO$ is birationally induced from $(L,\OO_L)$, it is birationally induced from $(M,\OO_M)$ for all the pair $(M, \OO_M)$ in the list above. It follows that the universal cover of $\OO$ is birationally rigid.

\vspace{3mm}
\noindent \underline{$D_6(a_2) \subset E_7$}. Note that $\pi_1(\OO) = \ZZ_2$. Hence, $\OO$ admits two non-isomorphic covers (including the trivial one). We have
    \begin{align*}
    \OO &= \Ind^G_{L(D_5)} \OO_{(3,2^2,1^3)}\\
    &= \Ind^G_{L(D_5+A_1)} (\Ind^{L(D_5+A_1)}_{L(D_5)} \OO_{(3,2^2,1^3)})\\
    &= \Ind^G_{L(D_5+A_1)} \OO_{(3,2^2,1^3)} \times \OO_{(2)} \end{align*}
   Write $(L,\OO_L) = (L(D_5+A_1), \OO_{(3,2^2,1^3)} \times \OO_{(2)})$. An {\fontfamily{cmtt}\selectfont atlas} computation shows that $\pi_1^L(\OO_L) = \ZZ_2$. Let $\widetilde{\OO}_L$ denote the two-fold $L$-equivariant cover of $\OO_L$. By (iv) of Proposition \ref{prop:propsofbind}, $\mathrm{Bind}^G_L \widetilde{\OO}_L$ is a two-fold cover of $\mathrm{Bind}^G_L \OO_L$. In particular, both covers of $\OO$ are birationally induced. 
  
\vspace{3mm} 
\noindent \underline{$E_6(a_3) \subset E_7$}. Note that $\pi_1(\OO)=S_2$. So $\OO$ admits two non-isomorphic covers (inluding the trivial one). We have
    \begin{align*}
    \OO &= \Ind^G_{L(2A_1+A_3;1,2,4,5,7)} \{0\}\\
    &= \Ind^G_{L(A_1+D_5)} (\Ind^{L(A_1+D_5)}_{L(2A_1+A_3;1,2,4,5,7)} \{0\})\\
    &= \Ind^G_{L(A_1+D_5)} \{0\} \times \OO_{(3^2,1^4)}
    \end{align*}
    Write $(L,\OO_L) = (L(A_1+D_5), \{0\} \times \OO_{(3^2,1^4)})$. An {\fontfamily{cmtt}\selectfont atlas} computation shows that $\pi_1^L(\OO_L) = \ZZ_2$. Let $\widetilde{\OO}_L$ denote the two-fold $L$-equivariant cover of $\OO_L$. By (iv) of Proposition \ref{prop:propsofbind}, $\mathrm{Bind}^G_L \widetilde{\OO}_L$ is a two-fold cover of $\mathrm{Bind}^G_L \OO_L$. In particular, both covers of $\OO$ are birationally induced.

\vspace{3mm} 
\noindent \underline{$E_7(a_5) \subset E_7$}. We have
    $$\mathcal{P}_{\mathrm{rig}}(\OO) = \{(L(A_1+2A_2;1,2,3,5,6),\{0\}), (L(A_1+A_3;1,2,4,5), \{0\}), (L(D_4), \OO_{(2^2,1^4)})\},$$
    and therefore $m(\OO)=3$. There are two codimension 2 orbits in $\overline{\OO}$, namely $\OO_1 = E_6(a_3)$ and $\OO_2 = E=D_6(a_2)$. The fundamental groups are as follows
    $$\pi_1(\OO) = S_3 \times \ZZ_2 \qquad \pi_1(\OO_1) = \pi_1(\OO_2) = \ZZ_2.$$
    Note that
$$|\pi_1(\OO)||\pi_1(\OO_k)|^{-1}=6>3=m(\OO) \qquad k=1,2,$$
and $\mathfrak{r} = 0$. So $\widehat{\OO}$ is birationally rigid by Corollary \ref{cor:criterionbirigidcover}.

\vspace{3mm}
\noindent \underline{$E_7(a_4) \subset E_7$}. We have
$$\mathcal{P}_{\mathrm{rig}}(\OO) = \{(L(A_1+D_4),\{0\} \times \OO_{(3,2^2,1)}), (L(2A_1+A_2;2,3,5,6),\{0\})\},$$
    and therefore $m(\OO)=3$. There are three codimension 2 orbits in $\overline{\OO}$, namely $\OO_1 = A_6$, $\OO_2= D_5+A_1$, and $\OO_3= D_6(a_1)$. The fundamental groups are as follows
    $$\pi_1(\OO) = S_2 \times \ZZ_2, \qquad \pi_1(\OO_1) =1, \qquad \pi_1(\OO_2) = \pi_1(\OO_3) = \ZZ_2.$$
     Note that
$$|\pi_1(\OO)||\pi_1(\OO_1)|^{-1}=4>3=m(\OO) ,$$
    and therefore the singularity $\Sigma_1$ is resolved under the map $\widehat{X}\to X$, see Lemma \ref{lem:criterion2leafless}. Moreover, the preimage of $\Sigma_1$ is two copies of $\CC^2$, and the Galois group $S_2\times \ZZ_2$ of the cover permutes the two copies. Thus, we have a map $\pi_1(\OO)\to S_2$. Let $K$ be its kernel, and let $\widetilde{\OO}=\widehat{\OO}/K$. Then $\widetilde{X}$ has $2$ symplectic leaves over the leaf $\fL_1\subset X$ corresponding to $\OO_1\subset \overline{\OO}$. We claim that $\widetilde{X}\to X$ resolves $\Sigma_2$ and $\Sigma_3$. Otherwise, we have $\dim \fP^{\widetilde{X}}\ge 3$, and hence both $\OO$ and $\widetilde{\OO}$ are birationally induced from a co-rank 3 Levi. Since $\mathcal{P}_{\mathrm{rig}}(\OO)$ contains a unique pair with the Levi of co-rank $\geq 3$, namely $(L(2A_1+A_2;2,3,5,6),\{0\})$, it follows that both $\OO$ and $\widetilde{\OO}$ are birationally induced from $(L(2A_1+A_2;2,3,5,6),\{0\})$, which is a contradiction. Therefore, $\widehat{\OO}$ is $2$-leafless. We have $\mathfrak{r} = 0$. So $\widehat{\OO}$ is birationally rigid by Corollary \ref{cor:criterionbirigidcover}.

\vspace{3mm}
\noindent \underline{$A_2 \subset E_8$}. We have
    $$\mathcal{P}_{\mathrm{rig}}(\OO) = \{(L(E_7),\{0\})\},$$
    and therefore $m(\OO)=1$. There is one codimension 2 orbit in $\overline{\OO}$, namely $\OO_1 = 3A_1$. The fundamental groups are as follows
    $$\pi_1(\OO) = S_2 \qquad \pi_1(\OO_1) = 1.$$
    Note that
$$|\pi_1(\OO)||\pi_1(\OO_1)|^{-1}=2>1=m(\OO),$$
and $\mathfrak{r}=E_6$. So $\widehat{\OO}$ is birationally rigid by Corollary \ref{cor:criterionbirigidcover}.

\vspace{3mm}   
\noindent \underline{$2A_2 \subset E_8$}. We have
    $$\mathcal{P}_{\mathrm{rig}}(\OO) = \{(L(D_7),\{0\})\},$$
    and therefore $m(\OO)=1$. There is a single codimension 2 orbit in $\overline{\OO}$, namely $\OO_1 = A_2+3A_1$. The fundamental groups are as follows
    $$\pi_1(\OO) = S_2 \qquad \pi_1(\OO_1) = 1.$$
    Note that
$$|\pi_1(\OO)||\pi_1(\OO_1)|^{-1}=2>1=m(\OO),$$
and $\mathfrak{r} = 2G_2$. So $\widehat{\OO}$ is birationally rigid by Proposition \ref{prop:criterionbirigidcover}.

\vspace{3mm}
\noindent \underline{$D_4(a_1) \subset E_8$}. We have
    $$\mathcal{P}_{\mathrm{rig}}(\OO) = \{(L(A_1+E_6),\{0\}), (L(E_7),2A_1)\},$$
    and therefore $m(\OO)=1$. There is a single codimension 2 orbit in $\overline{\OO}$, namely $\OO_1 = A_3+A_1$. The fundamnetal groups are as follows
    $$\pi_1(\OO) =S_3 \qquad \pi_1(\OO_1) = 1.$$
    Note that
$$|\pi_1(\OO)||\pi_1(\OO_1)|^{-1}=6>1=m(\OO),$$
and $\mathfrak{r} = D_4$. So $\widehat{\OO}$ is birationally rigid by Proposition \ref{prop:criterionbirigidcover}.

\vspace{3mm}
\noindent \underline{$A_3+A_2 \subset E_8$}. Recall that $H^2(\widehat{\OO},\CC) \simeq \mathfrak{z}(\mathfrak{r})$. Since $\mathfrak{r} = B_2+T_1$, $H^2(\widehat{\OO},\CC) \neq 0$. Hence, $\widehat{\OO}$ is birationally induced by Proposition \ref{prop:criterionbirigidcover}. On the other hand, $\widehat{\OO}$ is the unique nontrivial cover of $\OO$, since $\pi_1(\OO)=S_2$. 
  
\vspace{3mm}  
\noindent \underline{$D_4(a_1)+A_2 \subset E_8$}. We have
    $$\mathcal{P}_{\mathrm{rig}}(\OO) = \{(L(A_7),\{0\})\},$$
    and therefore $m(\OO)=1$. There is a single codimension 2 orbit in $\overline{\OO}$, namely $\OO_1 = A_3+A_2+A_1$. The fundamental groups are as follows
    $$\pi_1(\OO) = S_2 \qquad \pi_1(\OO_1) = 1.$$
    Note that
$$|\pi_1(\OO)||\pi_1(\OO_1)|^{-1}=2>1=m(\OO),$$
and $\mathfrak{r} = A_2$. So $\widehat{\OO}$ is birationally rigid by Proposition \ref{prop:criterionbirigidcover}.

\vspace{3mm}
\noindent \underline{$D_4+A_2 \subset E_8$}. Note that $\pi_1(\OO) = \ZZ_2$. Hence, $\OO$ has two non-isomorphic covers (including the trivial one). We have
    $$\mathcal{P}_{\mathrm{rig}}(\OO) = \{(L(A_6;1,3,4,5,6,7),\{0\})\}.$$
    We note that the induction from $(L(A_6;1,3,4,5,6,7),\{0\})$ is birational.  Recall the set $\mathcal{P}_1(\OO)$ defined above. Since $\OO$ is birationally induced from $(L(A_6;1,3,4,5,6,7),\{0\}))$, every pair $(M,\OO_M) \in \mathcal{P}_1(\OO)$ is of the form $\OO_M = \Ind^M_L \{0\}$, where $M$ is a co-rank 1 Levi containing (a $G$-conjugate of) $L$. Thus we have
    \begin{table}[H]
    \begin{tabular}{|c|c|c|} \hline
       $M$ &  $\OO_M$ & $\pi_1^M(\OO_M)$ \\ \hline 
     $D_7$ & $\OO_{(2^6,1^2)}$ & 1\\ \hline
     $A_7$ & $\OO_{(2,1^6)}$ & 1 \\ \hline
     $A_6+A_1$ & $\{0\}\times \OO_{(2)}$ & 1 \\ \hline
     $E_7$ & $A_2+3A_1$ & 1 \\ \hline
    \end{tabular}
    \caption{$\mathcal{P}_1(\OO)$}
\end{table}
Since $\OO$ is birationally induced from $(L,\OO_L)$, it is birationally induced from $(M,\OO_M)$ for all the pair $(M, \OO_M)$ in the list above. It follows that the universal cover of $\OO$ is birationally rigid.

\vspace{3mm}
\noindent \underline{$D_6(a_2) \subset E_8$}. We have
    $$\mathcal{P}_{\mathrm{rig}}(\OO) = \{(L(D_7),(3,2^4,1))\},$$
    and therefore $m(\OO)=1$. There are $2$ codimension 2 orbits in $\overline{\OO}$, namely $\OO_1 = D_5(a_1)+A_2$ and $\OO_2=A_5+A_1$. The corresponding singularities are of types $A_1$ and $m$. Thus, there is a unique codeminsion $2$ leaf $\fL_1\subset \Spec(\CC[\OO])$. The fundamental groups are as follows
    $$\pi_1(\OO) = S_2 \qquad \pi_1(\OO_1) = 1.$$
    Note that
$$|\pi_1(\OO)||\pi_1(\OO_1)|^{-1}=2>1=m(\OO),$$
and $\mathfrak{r}=2A_1$. So $\widehat{\OO}$ is birationally rigid by Corollary \ref{cor:criterionbirigidcover}.

\vspace{3mm}
\noindent \underline{$E_6(a_3) \subset E_8$}. Note that $\pi_1(\OO)=S_2$. So $\OO$ admits two non-isomorphic covers (including the trivial one). We have
    $$\OO=\Ind^G_{L(D_5+A_1)} \{0\} = \Ind^G_{L(D_7)} (\Ind^{L(D_7)}_{L(D_5+A_1)} \{0\}) = \Ind^G_{L(D_7)} \OO_{(3^2,1^8)}.$$
    Write $(L,\OO_L) = (L(D_7),\OO_{(3^2,1^8)})$. An {\fontfamily{cmtt}\selectfont atlas} computation shows that $\pi_1^L(\OO_L) \simeq \ZZ_2$. Let $\widetilde{\OO}_L$ denote the two-fold $L$-equivariant cover of $\OO_L$. By (iv) of Proposition \ref{prop:propsofbind}, $\mathrm{Bind}^G_M \widetilde{\OO}_M$ is a 2-fold cover of $\mathrm{Bind}^G_M \OO_M$. In particular, both covers of $\OO$ are birationally induced.
    
\vspace{3mm}  
\noindent \underline{$E_6(a_3)+A_1 \subset E_8$}. We have
    $$\mathcal{P}_{\mathrm{rig}}(\OO) = \{(L(E_7),A_1+2A_2)\},$$
    and therefore $m(\OO)=1$. There are $2$ codimension 2 orbits in $\overline{\OO}$, namely $\OO_1=A_5+A_1$ and $\OO_2=D_5(a_1)+A_2$. The corresponding singularities are of types $A_1$ and $m$. Thus, there is a unique codeminsion $2$ leaf $\fL_1\subset \Spec(\CC[\OO])$. The fundamental groups are as follow
    $$\pi_1(\OO) =S_2 \qquad \pi_1(\OO_1) =1.$$
    Note that
$$|\pi_1(\OO)||\pi_1(\OO_1)|^{-1}=2>1=m(\OO),$$
and $\mathfrak{r}=A_1$. So $\widehat{\OO}$ is birationally rigid by Corollary \ref{cor:criterionbirigidcover}.

\vspace{3mm}
\noindent \underline{$E_7(a_5) \subset E_8$}. We have
$$\mathcal{P}_{\mathrm{rig}}(\OO) = \{(L(E_7), (A_1+A_3)'), (L(E_6+A_1), 3A_1 \times \{0\})\},$$
and therefore $m(\OO)=1$. There are two codimension 2 orbits in $\overline{\OO}$, namely $\OO_1=E_6(a_3)+A_1$ and $\OO_2=D_6(a_2)$. The fundamental groups are as follows
$$\pi_1(\OO) =  S_3, \qquad \pi_1(\OO_1) = S_2, \qquad \pi_1(\OO_2) = S_2.$$
Note that
$$|\pi_1(\OO)||\pi_1(\OO_1)|^{-1} = |\pi_1(\OO)||\pi_1(\OO_2)|^{-1} = 3 > 1,$$
and $\mathfrak{r} = A_1$. So $\widehat{\OO}$ is birationally rigid by Corollary \ref{cor:criterionbirigidcover}.

\vspace{3mm}
\noindent \underline{$E_8(a_7) \subset E_8$}. We have
    \begin{align*}\mathcal{P}_{\mathrm{rig}}(\OO) = \{&(L(A_3+A_4),\{0\}), (L(A_2+D_5), \{0\} \times \OO_{(2^2,1^6)}),\\
    &(L(A_1+A_5; 1,2,4,5,6,7),\{0\}), (L(D_6), (2^4,1^4))\},\end{align*}
    and therefore $m(\OO)=2$. There is a single codimension 2 orbit in $\overline{\OO}$, namely $\OO_1 = E_7(a_5)$. The fundamnetal groups are as follows
    $$\pi_1(\OO) = S_5 \qquad \pi_1(\OO_1) = S_3.$$
    Note that
$$|\pi_1(\OO)||\pi_1(\OO_1)|^{-1}=20>2=m(\OO),$$
and $\mathfrak{r}=0$. So $\widehat{\OO}$ is birationally rigid by Corollary \ref{cor:criterionbirigidcover}.

\vspace{3mm}
\noindent \underline{$E_7(a_4) \subset E_8$}. Note that $\pi_1(\OO)=S_2$. So $\OO$ admits two non-isomorphic covers (including the trivial one). We have
    $$\mathcal{P}_{\mathrm{rig}}(\OO) = \{(L(A_1+D_5;1,2,3,4,5,7), \{0\} \times \OO_{(2^2,1^6)}), (L(D_6), \OO_{(3,2^2,1^5)})\}.$$
    By (iii) of Proposition \ref{prop:propsofbind},  $\mathrm{Bind}^G_{L(A_1+D_5;1,2,3,4,5,7)} \{0\}$ and $\mathrm{Bind}^G_{L(D_6)} \OO_{(3,2^2,1^5)}$ are non-isomorphic covers of $\OO$. In particular, both covers of $\OO$ are birationally induced.
    
\vspace{3mm}   
\noindent \underline{$D_5+A_2\subset E_8$} . Recall that $H^2(\widehat{\OO},\CC) \simeq \fz(\mathfrak{r})$. Since $\mathfrak{r} = T_1$, $H^2(\widehat{\OO},\CC) \neq 0$.  So by Proposition \ref{prop:criterionbirigidcover}, $\widehat{\OO}$ is birationally induced. On the other hand, $\widehat{\OO}$ is the unique nontrivial cover of $\OO$, since $\pi_1(\OO)=S_2$. The singularity of a codimension $2$ orbit $E_7(a_4)\subset \overline{\OO}$ is not of type $m$ (namely, it is of type $A_1$), and therefore the Namikawa space $\fP^X$ for $X=\Spec(\CC[\OO])$ is non-trivial. Thus, $\OO$ is birationally induced, and so all covers of $\OO$ are birationally induced.
 
\vspace{3mm}
\noindent \underline{$D_7(a_1) \subset E_8$}. Since there is a cosimension $2$ orbit $D_6(a_2)\subset \overline{\OO}$ with the singularity not of type $m$ (namely, it is of type $2A_1$), the argument above for $D_5+A_2 \subset E_8$ holds word for word. In particular, all covers of $\OO$ are birationally induced.
 \end{proof}

\subsection{Some computational tools}

Let $\OO$ be a nilpotent orbit such that $H^2(\OO,\CC)=0$. Let $X=\Spec(\CC[\OO])$ and write $\fL_1,...,\fL_t \subset X$ for the codimension 2 leaves. For the determination of unipotent infinitesimal characters (to be carried out in Section \ref{sec:inflchars}), there are three separate computations one needs to perform:
\begin{enumerate}
    \item Compute the birationally minimal induction datum $(L,\OO_L)$ for $\OO$.
    \item Compute the $\fL_k$-adapted resolution datum $(M_k, \OO_{M_k})$ for each codimension 2 leaf.
    \item Compute the isomorphism $\eta_k: \fX(\fm) \xrightarrow{\sim} \fP_k^X$ for each codimension 2 leaf.
\end{enumerate}
We will use several techniques for each of these computations, which we will explain and catalog below. 

\subsubsection{Computation of $(L,\OO_L)$}\label{subsubsec:computeLOL}

First, we explain two techniques for computing $(L,\OO_L)$. Using an $\Ad(\fg)$-invariant identification $\fg \simeq \fg^*$, we can regard $\OO$ as a nilpotent $G$-orbit in $\fg$. Choose an element $e \in \OO$ and an $\mathfrak{sl}(2)$-triple $(e,f,h)$. The operator $\ad(h)$ defines a $\ZZ$-grading on $\fg$
$$\fg = \bigoplus_{i \in \ZZ} \fg_i, \qquad \fg_i := \{\xi \in \fg \mid \ad(h)(\xi) = i\xi\}.$$
We say that $\OO$ is \emph{even} if $\fg_i=0$ for every odd integer $i$. In any event, we can define a parabolic subalgebra
\begin{equation}\label{eq:JMlevi}\mathfrak{p}_{\OO} = \mathfrak{l}_{\OO} \oplus \mathfrak{n}_{\OO}, \qquad \fl_{\OO} := \fg_0, \qquad \fn_{\OO} := \bigoplus_{i \geq 1} \fg_i.\end{equation}
We call $\fp_{\OO}$ (resp. $\fl_{\OO}$) the \emph{Jacobson-Morozov} parabolic (resp. Levi) associated to $\OO$. Both $\fp_{\OO}$ and $\fl_{\OO}$ are well-defined up to conjugation by $G$. It is very easy to determine $\fp_{\OO}$ from the weighted diagram for $\OO$---it is the parabolic subalgebra corresponding to the simple roots labeled `0.' 
The following result is well-known. The proof is contained in \cite{Kostant1959}, see also \cite[Thm 3.3.1]{CM}.

\begin{lemma}\label{lem:even}
If $\OO$ is even, then $\OO=\mathrm{Bind}^G_{L_{\OO}} \{0\}$.
\end{lemma}

We will also use the following.

\begin{lemma}\label{lem:Prigunique}
Suppose
\begin{itemize}
    \item[(i)] The set $\mathcal{P}_{\mathrm{rig}}(\OO)$ contains a unique element $(L_0,\OO_{L_0})$.
    \item[(ii)] $\dim  \fX(\fl_0 \cap [\fg,\fg]) \leq \dim \fP^X$.
\end{itemize}
Then $\OO=\mathrm{Bind}^G_{L_0} \OO_{L_0}$. 
\end{lemma}

\begin{proof}
Choose a birationally minimal induction datum $(L, \OO_L)$ for $\OO$. Suppose for contradiction that $(L,\OO_L)$ is not $G$-conjugate to $(L_0,\OO_{L_0})$. In particular, by condition (i), $\OO_L$ is not rigid. So there is a proper Levi subgroup $M \subset L$ and a rigid $M$-orbit $\OO_M$ such that $\OO_L = \Ind^L_M \OO_M$. By the transitivity of induction, $(M,\OO_M) \in \mathcal{P}_{\mathrm{rig}}(\OO)$. Hence, $(M,\OO_M)$ is $G$-conjugate to $(L_0,\OO_{L_0})$ and $\dim \fX(\fl_0 \cap [\fg,\fg]) > \dim \fX(\fl \cap [\fg,\fg])$. On the other hand, by (ii) of Proposition \ref{prop:nilpquant}, there is a linear isomorphism $\fX(\fl \cap [\fg,\fg]) \simeq \fP^X$. So
$$\dim \fX(\fl_0 \cap [\fg,\fg]) > \dim \fX(\fl \cap [\fg,\fg]) = \dim \fP^X.$$
This contradicts condition (ii).
\end{proof}

\subsubsection{Computation of $(M_k,\OO_{M_k})$}\label{subsubsec:computeMOM}

Fix a birationally minimal induction datum $(L,\OO_L)$ for $\OO$, and choose a codimension 2 leaf $\fL_k \subset X$. In many cases we consider, $\fL_k = \fL_1$ is the \emph{only} codimension 2 leaf. In such cases, $(L,\OO_L)$ trivially satisfies conditions (i)-(iii) of Proposition \ref{prop:adaptedresolutiondatum} and therefore 
$$(M_1, \OO_{M_1}) = (L,\OO_L)$$
In the remaining cases, we use the following result, which follows immediately from \cite[Lem 7.5.10]{LMBM} and Lemma \ref{lem:weaklynormal}.

\begin{lemma}\label{lem:findMk}
Suppose $M \subset G$ is a Levi subgroup containing $L$ and let $\OO_M=\Ind^M_L\OO_L$. Then $(M,\OO_M) = (M_k,\OO_{M_k})$ if the following conditions are satisfied:
\begin{itemize}
    \item[(i)] The semisimple corank of $M$ equals the dimension of $\fP_k$.
    \item[(ii)] For every $j \neq k$, there is a nilpotent $M$-orbit $\OO_{M,j} \subset \overline{\OO}_M$ of codimension 2 such that
    $$\OO_j = \mathrm{Ind}^G_M \OO_{M,j}$$
    \item[(iii)] For every $j \neq k$, such that the singularity of the orbit $\OO_j \subset \overline{\OO}$ is not of type $m$, the singularity of $\OO_{M,j} \subset \overline{\OO}_{M,j}$ is not of type $m$, and moreover the singularities of $\fL_j\subset Spec(\CC[\OO])$ and $\fL_{M,j}\subset Spec(\CC[\OO_M])$ are equivalent. 
\end{itemize}
\end{lemma}

\begin{rmk}
We note that condition (iii) of Lemma \ref{lem:findMk} is implied by the following condition (which is often easier to check):
    \begin{itemize}
        \item[(iii')] For every $j \neq k$ the singularities of $\OO_j\subset \overline{\OO}$ and $\OO_{M,j}\subset \overline{\OO}_M$ are equivalent. 
    \end{itemize}
\end{rmk}

\subsubsection{Computation of $\eta_k$}\label{subsubsec:etak}

As in the previous subsection, fix a birationally minimal induction datum $(L,\OO_L)$ for $\OO$, and choose a codimension 2 leaf $\fL_k \subset X$. Assume
\begin{itemize}
    \item $\pi_1(\fL_k)$ acts trivially on $H^2(\mathfrak{S}_k,\CC)$.
\end{itemize}
Under this condition, $\fP_{k}$ can be identified with the vector space $\fh_{k}^*$, i.e. the dual Cartan subalgebra corresponding to the Kleinian singularity $\Sigma_k$. In particular, $\fP_{k}$ admits a natural basis consisting of fundamental weights, denoted $\{\omega_i(k) \mid 1 \leq i \leq n(k)\}$. On the other hand, $\fX(\fm_k)$ admits a basis consisting of dominant generators for the free abelian group $\fX(M_k)$, denoted $\{\tau_i(k) \mid 1 \leq i \leq n(k)\}$. In cases when we have $\fP=\fP_1$ we omit the index $k$ and write simply $\omega_i$ and $\tau_i$. We wish to compute $\eta_k$ in terms of these bases. In most cases, we will use one of the following two results, established in \cite{LMBM}.

\begin{prop}\label{prop:identification1}
Suppose
\begin{itemize} 
    \item[(a1)] $\fL_k=\fL_1$ is the unique codimension 2 leaf in $\widetilde{X}$. 
    \item[(a2)] $\Pic(\OO) \simeq \fX(\Gamma_1)$.
    \item[(a3)] $\Pic(\OO_{M_1}) = 0$.
    \item[(a4)] Up to the action of $W^{\widetilde{X}}$, there is a unique parabolic subgroup $P \subset G$ with Levi factor $L$.
\end{itemize}
Then $M_1=L$, $\eta_1=\eta$, and $\eta_1\{\tau_i\}  = \{\omega_i\}$.
\end{prop}

\begin{prop}\label{prop:identification2}
Suppose
\begin{itemize}
    \item[(b1)] $\Sigma_k$ is of type $A_1$. 
\end{itemize}
Then $\Pic(\mathbb{O}_{M_k})$ is finite, 
\begin{equation}\label{eq:defofck}c_k := 2|\Pic(\mathbb{O}_{M_k})||\Pic(\mathbb{O})|^{-1}\end{equation}
is an integer, and $\eta_k(\tau_1(k)) = c_k\omega_1(k)$.
\end{prop}

In a few cases, neither proposition is straightforwardly applicable, and we must provide a separate argument.

\section{Unipotent ideals and their infinitesimal characters}\label{sec:inflchars}

Let $G$ be a connected reductive algebraic group and let $\widetilde{\OO}$ be a $G$-equivariant nilpotent cover. Recall the canonical quantization $(\cA_0,\Phi_0)$ of $\CC[\widetilde{\OO}]$, cf. Definition \ref{def:canonical}. 

\begin{definition}
The \emph{unipotent ideal} attached to $\widetilde{\OO}$ is the two-sided ideal
$$I_0(\widetilde{\OO}) := \ker{(\Phi_0: U(\fg) \to \cA_0)} \subset U(\fg)$$
\end{definition}

We show in \cite{LMBM} that $I_0(\widetilde{\OO})$ is a completely prime primitive ideal with associated variety $\overline{\OO}$. We also give a classification of unipotent ideals, which we will now recall. Write $\widetilde{\OO}^1 \succeq \widetilde{\OO}^2$ if there is a $G$-equivariant morphism $\widetilde{\OO}^1 \to \widetilde{\OO}^2$ such that the induced morphism $\Spec(\CC[\widetilde{\OO}^1]) \to \Spec(\CC[\widetilde{\OO}^2])$ is \'{e}tale over the open subset in $\Spec(\CC[\widetilde{\OO}^2])$ obtained by removing all symplectic leaves of codimension $\geq 4$. Note that $\succeq$ defines a partial order on the set of ($G$-equivariant) nilpotent covers. Consider the equivalence relation (on the same set) which is generated by $\succeq$ and write $[\widetilde{\OO}]$ for the equivalence class of $\widetilde{\OO}$.

\begin{theorem}[Thm 6.5.5, \cite{LMBM}]\label{thm:classificationideals}
Suppose $\widetilde{\OO}^1$, $\widetilde{\OO}^2$ are $G$-equivariant nilpotent covers. Then $I_0(\widetilde{\OO}^1) = I_0(\widetilde{\OO}^2)$ if and only if $[\widetilde{\OO}^1]=[\widetilde{\OO}^2]$. 
\end{theorem}

Let $\gamma_0(\widetilde{\OO}) \in \fh^*/W$ denote the infinitesimal character of $I_0(\widetilde{\OO})$. If $(L,\widetilde{\OO}_L)$ is a birationally minimal indution datum, then we have the following result.

\begin{prop}[Prop 8.1.1, \cite{LMBM}]\label{prop:nodeltashift}
There is an equality in $\fh^*/W$
$$\gamma_0(\widetilde{\OO}) = \gamma_0(\widetilde{\OO}_L).$$
\end{prop}

Thus, the computation of $\gamma_0(\widetilde{\OO})$ can be reduced to the case of birationally rigid covers. In \cite{LMBM}, we computed $\gamma_0(\widetilde{\OO})$ for such covers for linear classical groups. In this section, we will compute $\gamma_0(\widetilde{\OO})$ for all birationally rigid covers for spin and exceptional groups, effectively completing the computation of unipotent infinitesimal characters for all nilpotent covers.

\subsection{Computing $\gamma_0(\widetilde{\OO})$}\label{subsec:algorithm}

Let $\widetilde{\OO}$ be a birationally rigid nilpotent cover. In this subsection, we will recall a general algorithm, developed in \cite{LMBM}, for computing $\gamma_0(\widetilde{\OO})$. This algorithm has three separate components.

\subsubsection{Reduction to the case of birationally rigid orbits}

Choose a birationally minimal induction datum $(L,\OO_L)$ for $\OO$. Since $\widetilde{\OO}$ is birationally rigid, $\Spec(\CC[\widetilde{\OO}])$ has no codimension 2 leaves, see Corollary \ref{cor:criterionbirigidcover}. Let $X= \Spec(\CC[\OO])$ and let $\epsilon \in \fP^X$ denote the barycenter parameter (cf. (\ref{eq:defofepsilon})). Recall, Proposition \ref{prop:nilpquant}, that there is a natural identification 
$$\eta: \fX(\fl \cap [\fg,\fg]) \xrightarrow{\sim} \fP^X$$
We define $\delta := \eta^{-1}(\epsilon)$. This element can be computed using the techniques of Section \ref{subsubsec:etak}. The following is a special case of \cite[Prop 8.1.3]{LMBM}.

\begin{prop}\label{prop:deltashift}
There is an equality in $\fh^*/W$
$$\gamma_0(\widetilde{\OO}) = \gamma_0(\OO_L) + \delta$$
\end{prop}
This proposition reduces the computation of $\gamma_0(\widetilde{\OO})$ to the case of birationally rigid orbits.

\subsubsection{Reduction to the case of rigid orbits}

Let $\OO$ be a birationally rigid orbit. Choose a Levi subgroup $L \subset G$ and a rigid orbit $\OO_L$ such that $\OO=\Ind^G_L \OO_L$. The following is a special case of \cite[Prop 8.1.3]{LMBM}.

\begin{prop}\label{prop:inflcharbirigidorbit}
There is an equality in $\fh^*/W$
$$\gamma_0(\OO) = \gamma_0(\OO_L).$$
\end{prop}

\subsubsection{Case of rigid orbits}

If $\OO$ is a rigid orbit, $\gamma_0(\OO)$ was computed in all cases in \cite{LMBM}. The argument proceeds as follows. With six exceptions (in types $G_2$,$F_4$, $E_7$, and $E_8$), $I_0(\OO)$ is the \emph{unique} primitive ideal with the following two properties
\begin{enumerate}
    \item The associated variety of $I_0(\OO)$ is $\overline{\OO}$.
    \item $I_0(\OO)$ has multiplicity $1$ along $\overline{\OO}$.
\end{enumerate}
Such ideals were classified (i.e. their infinitesimal characters were computed) by McGovern (\cite{McGovern1994}) in classical types and Premet (\cite{Premet2013}) in exceptional types. For the six `bad' orbits in exceptional types, there are several ideals satisfying (1) and (2) above--some, though not all, were computed by Premet in \cite{Premet2013}. In \cite[Appendix C]{LMBM} we use a technical geometric argument to resolve this ambiguity, determing the unipotent ideal $I_0(\widetilde{\OO})$ in all cases.

\subsection{Spin groups}\label{subsec:inflcharspin}

Let $G=\mathrm{SO}(n)$, $\widetilde{G}=\mathrm{Spin}(n)$, and let $\widetilde{\OO}$ be a birationally rigid $\widetilde{G}$-equivariant nilpotent cover. In this section, we will give a combinatorial formula for the infinitesimal character $\gamma_0(\widetilde{\OO})$.

We begin by recalling the classification of birationally semi-rigid nilpotent orbits for 
$\widetilde{G}$.

\begin{prop}[Props 7.6.7, 7.6.11, \cite{LMBM}]
\label{prop:Spinsemirigid}
Let $\OO$ be a nilpotent $\widetilde{G}$-orbit corresponding to a partition $p$ of $n$.
\begin{itemize}
    \item $\OO$ admits a birationally rigid $G$-equivariant cover if and only if the following conditions are satisfied:
\begin{itemize}
    \item[(i)] if $p_i$ is odd, then $p_i \leq p_{i+1}+2$.
    \item[(ii)] If $p_i$ is even, then $p_i \leq p_{i+1}+1$.
\end{itemize}
\item $\OO$ admits a birationally rigid $\widetilde{G}$-equivariant cover, which is \emph{not} $G$-equivariant, if and only if the following conditions are satisfied:
\begin{itemize}
    \item[(ia)] $p$ is rather odd (i.e. every odd part occurs with multiplicity 1).
    \item[(iia)] $p_i \leq p_{i+1}+1$ if $p_i$ is even, and $p_i \leq p_{i+1}+4$ if $p_i$ is odd.
   \item[(iiia)] $p_i \neq p_{i+1}+3$ for all $i$.
   \item[(iva)] $p_i=p_{i+1}=4$ for some odd $p_i$.
\end{itemize}
\end{itemize}
\end{prop}

Our main result in this subsection will require some additional notation. By a $\frac{1}{2}\ZZ$-\emph{partition} of $n \in \ZZ$ we will mean a non-increasing sequence $p=(p_1,p_2,...,p_t)$ in $\frac{1}{2}\ZZ$ such that $\sum_{i=1}^t p_i = n$ (to prevent confusion, we will call ordinary partitions `$\ZZ$-partitions' in this section to emphasize the distinction with $\frac{1}{2}\ZZ$-partitions).

\begin{definition}\label{def:combinatorics}
Let $p$ be a $\ZZ$-partition of $n$.

\begin{itemize}
\item Define a $\ZZ$-partition (with no repeated parts)
$$S_4(p) = \{i: p_i = p_{i+1} + 4\}.$$
\item If $x=(x_1,...,x_r)$ is a subpartition of $S_4(p)$, define a $\ZZ$-partition $p \# x$ by deleting the columns in $p$ numbered $p_{x_1}, p_{x_1}-1$, $p_{x_2}$, $p_{x_2}-1$, ..., $p_{x_r}$, $p_{x_r}-1$. 
\item Define $\ZZ$-partitions $x(p)$, $y(p)$, and $z(p)$ by extracting all parts of $p$ with multiplicity 1,2, and 4, respectively.
\item Define a $\ZZ$-partition $f(p)$ as follows: for every odd $i$ with $p_i \geq p_{i+1} +2$, move one box down from $p_i$ to $p_{i+1}$.
\item If all parts in $p$ are of multiplicity 2, define a $\ZZ$-partition $g(p)$ by replacing every pair $(p_i,p_i)$ with $(p_i+1,p_i-1)$.
\item If all parts in $p$ are of multiplicity 2, define a $\frac{1}{2}\ZZ$-partition $h'(p)$ by replacing every pair $(p_i,p_i)$ with $(p_i+1/2,p_i-1/2)$.
\item If all parts in $p$ are of multiplicity 4, define $\frac{1}{2}\ZZ$-partition $h(p)$ by replacing every quadruple $(p_i,p_i,p_i,p_i)$ with $(p_i+1,p_i+1/2,p_i-1/2,p_i-1)$.
\item If all parts in $p$ are of even multiplicity, define $p^{1/2}$ by halving all multiplicities.
\end{itemize}

\end{definition}

\begin{example}
If $p=(9,5,4^2,3^4,1)$, then
$$S_4(p) = (1), \quad p\#S_4(p) = (7,5,4^2,3^4,1), \quad x(p) = (9,5,1), \quad y(p) = (4^2), \quad z(p)= (3^4).$$
If $p=(5^2,2^2,1^2)$, then $g(p)=(6,4,3,2,1)$. If $p=(5^4,1^4)$, then $h(p) = (6,11/2,9/2,4,2,3/2,1/2)$.
\end{example}

\begin{definition}
\begin{itemize}
\item If $q$ is a $\ZZ$-partition of $n$, define a $\floor{\frac{n}{2}}$-tuple $\rho(q) \in (\frac{1}{2}\ZZ)^{\floor{\frac{n}{2}}}$ by appending the sequence
$$(\frac{q_i-1}{2}, \frac{q_i-3}{2}, ..., \frac{3-q_i}{2}, \frac{1-q_i}{2})$$
for each $i \geq 1$.
\item If $q$ is a $\frac{1}{2}\ZZ$-partition, define a $\floor{\frac{n}{2}}$-tuple $\rho^+(q) \in (\frac{1}{4}\ZZ)^{\floor{\frac{n}{2}}}$ by appending the positive elements of the sequence
$$(\frac{q_i-1}{2}, \frac{q_i-3}{2}, ..., \frac{3-q_i}{2}, \frac{1-q_i}{2})$$
for each $i \geq 1$, and then $0$'s as needed so that $|\rho^+(p)| = \floor{\frac{n}{2}}$ (if $q_i=1/2$, append nothing).
\end{itemize}
\end{definition}

\begin{example}
If $q=(3^2,2)$, then
$$\rho(q) = (1,0,-1,1,0,-1,1/2,-1/2).$$
If $q = (5/2,5/2,2,3/2,1/2)$, then
$$\rho^+(q) = (3/4,3/4,1/2,1/4,0).$$
\end{example}

\begin{prop}\label{prop:inflcharspin}
Let $\OO$ be a nilpotent $\widetilde{G}$-orbit corresponding to a partition $p$ of $n$. Assume $\OO$ admits a birationally rigid $\widetilde{G}$-equivariant nilpotent cover $\widetilde{\OO} \to \OO$. Let
$$x:=x(p^t), \quad y:=y(p^t), \quad z:=z(p^t).$$
By Proposition \ref{prop:Spinsemirigid}, $p^t=x\cup y \cup z$. In standard coordinates on $\fh^*$
$$\gamma_0(\widetilde{\OO}) = \rho^+(f(x) \cup g(y) \cup h(z))$$
If $n$ is even, this weight contains (at least one) entry equal to $0$ (i.e. the Weyl group acts on this weight by arbitrary permutations and sign changes). 
\end{prop}

\begin{proof}
Let $\widehat{\OO} \to \widetilde{\OO}$ denote the universal ($\widetilde{G}$-equivariant) cover. By Corollary \ref{cor:criterionbirigidcover}, $\Spec(\CC[\widetilde{\OO}])$ has no codimension 2 leaves. So by Theorem \ref{thm:classificationideals}, $I_0(\widehat{\OO})=I_0(\widetilde{\OO})$, and therefore $\gamma_0(\widehat{\OO}) = \gamma_0(\widetilde{\OO})$. If $\widehat{\OO}$ is $G$-equivariant, then by Proposition \ref{prop:Spinsemirigid}, $z=\emptyset$, i.e. $p^t = x \cup y$. Hence
$$\rho^+(f(x) \cup g(y) \cup h(z)) = \rho^+(f(x) \cup g(y)).$$
This coincides with $\gamma_0(\widehat{\OO})$ by \cite[Prop 8.2.8]{LMBM}. Thus, we can assume that $\widehat{\OO}$ is not $G$-equivariant. In particular, by Proposition \ref{prop:Spinsemirigid}, $p$ is rather odd. So $\pi_1^{\widetilde{G}}(\OO)$ is a central $\ZZ_2$-extension of $\pi_1^G(\OO)$, see \cite[Cor 6.1.6]{CM}, and $\widehat{\OO}$ is a 2-fold cover of the universal $G$-equivariant cover $\widetilde{\OO}$ of $\OO$. 

Form the partition $p\# S_4(p)$ of $n-2|S_4(p)|$ as in Definition \ref{def:combinatorics} and let $\widetilde{\OO}_{p\# S_4(p)}$ be the universal $\mathrm{SO}(n-2|S_4(p)|)$ cover of $\OO_{p\# S_4(p)}$. We claim first of all that the following pair is a birationally minimal induction datum for $\widetilde{\OO}$
$$L = \prod_{k \in S_4(p)} \mathrm{GL}(k) \times \mathrm{SO}(n-2|S_4(p)|), \qquad \widetilde{\mathbb{O}}_L =  \{0\} \times ... \{0\} \times \widetilde{\mathbb{O}}_{p \# S_4(p)}.$$
Indeed, this is a special case of \cite[Thm 4.17]{Mitya2020}. Let
$$x' = x((p\#S_4(p))^t), \qquad y' = y((p\#S_4(p))^t),$$
so that $(p\#S_4(p))^t = x' \cup y'$. Applying \cite[Prop 8.2.8]{LMBM} to $\widetilde{\OO}_L$ we obtain
$$\gamma_0(\widetilde{\OO}_L) = (\rho(S_4(p)), \gamma_0(\widetilde{\mathbb{O}}_{p \# S_4(p)})) = (\rho(S_4(p)), \rho^+(f(x') \cup g(y'))).$$
The shift $\delta = \eta^{-1}(\epsilon) \in \fX(\fl)$ was computed in \cite[Cor 7.7.7]{LMBM}. In standard coordinates it is the element
$$\delta = (\underbrace{1/4,...,1/4}_{|S_4(p)|},0,...,0) \in \fX(\fl)$$
So by Proposition \ref{prop:deltashift} we have
$$\gamma_0(\widehat{\OO}) = (\rho(S_4(p)), \rho^+(f(x') \cup g(y'))) + (\underbrace{1/4,...,1/4}_{|S_4(p)|},0,...,0).$$
Note that up to permutations and sign changes
$$\rho(S_4(p)) + (1/4,...,1/4) = \rho^+(h'(z^{1/2})),$$
So (up to permutations and sign changes)
$$\gamma_0(\widehat{\OO}) = \rho^+(f(x') \cup g(y') \cup h'(S_4(p) \cup S_4(p))).$$
By the construction of $S_4(p)$ and $p\# S_4(p)$, we have 
$$x' = x, \quad \text{and} \quad y' = y \cup z^{1/2}.$$
So (up to permutation)
\begin{equation}\label{eq:eq1}\gamma_0(\widehat{\OO}) = \rho^+(f(x) \cup g(y) \cup g(z^{1/2}) \cup h'(z^{1/2})).\end{equation}
By definition, $h(z) =  g(z^{1/2}) \cup h'(z^{1/2})$. So (\ref{eq:eq1}) becomes (up to permutation)
$$\gamma_0(\widehat{\OO}) = \rho^+(f(x)\cup g(y) \cup h(z)).$$
\end{proof}

\begin{example}
Let $n=6$ (so that $\widetilde{G} = \mathrm{Spin}(6) \simeq \mathrm{SL}(4)$), and let $\OO$ be the principal nilpotent orbit (so that $p=(5,1)$). By Proposition \ref{prop:Spinsemirigid}, $\OO$ admits a birationally rigid $\widetilde{G}$-equivariant cover $\widetilde{\OO}$ (which is not $G$-equivariant). In the notation of Proposition \ref{prop:inflcharspin}, we have
$$p^t = (2,1^4), \quad x=(2), \quad y=\emptyset, \quad z=(1^4).$$
Hence
$$f(x) = (1,1), \quad g(y) = \emptyset, \quad h(z) = (2,3/2,1/2).$$
So
$$\gamma_0(\widetilde{\OO}) = \rho^+(f(x) \cup g(y) \cup h(z))) = (1/2,1/4,0) = \rho/4.$$
\end{example}

\begin{example}
Let $n=15$ (so that $\widetilde{G}=\mathrm{Spin}(15)$) and let $\OO$ be the nilpotent orbit corresponding to the partition $p=(9,5,1)$. By Proposition \ref{prop:Spinsemirigid}, $\OO$ admits a birationally rigid $\widetilde{G}$-equivariant cover $\widetilde{\OO}$ (which is not $G$-equivariant). In the notation of Proposition \ref{prop:inflcharspin}, we have
$$p^t=(3,2^4,1^4), \quad x=(3), \quad y = \emptyset, \quad z = (2^4,1^4).$$
Hence
$$f(x) = (2,1), \quad g(y) =\emptyset, \quad h(z) = (3,5/2,2,3/2,3/2,1,1/2).$$
So
$$\gamma_0(\widetilde{\OO}) = \rho^+(3,5/2,2,2,3/2,3/2,1,1,1/2)= (1,3/4,1/2,1/2,1/4,1/4,0).$$
\end{example}

\subsection{Exceptional groups}\label{subsec:inflcharexceptional}

In this section, we will produce a complete list of unipotent infinitesimal characters attached to birationally rigid covers $\widetilde{\OO}$ for simple exceptional groups.

There are essentially two cases to consider:

\begin{enumerate}
    \item $\widetilde{\OO}$ is (the trivial cover of) a birationally rigid orbit.
    \item $\widetilde{\OO}$ is a nontrivial cover of a birationally induced orbit. 
\end{enumerate}
Of course, it is also possible that $\widetilde{\OO}$ is a \emph{nontrivial} cover of a birationally rigid orbit, but such covers bring nothing new into the mix. Indeed, $I_0(\widetilde{\OO}) = I_0(\OO)$ by Theorem \ref{thm:classificationideals}, so we can easily reduce to (1). 

We pause to review our standing notational conventions for Levis and weights in exceptional types. As explained in the proof of Proposition \ref{prop:listofsemirigid}, simple roots are denoted $\alpha_1,...,\alpha_n$ and are numbered in accordance with the Bourbaki conventions. We write $L(X;r_1,...,r_p)$ for the standard Levi subgroup with simple roots $\alpha_{r_1},...,\alpha_{r_p}$ and Lie type $X$ (omitting $r_1,...,r_p$ in some cases when $X$ determines the Levi). All weights will be denoted in fundamental weight coordinates. For example, $\rho$ is denoted, in every case, by the tuple $(1,1,...,1)$.

\subsubsection{Birationally rigid orbits}\label{subsec:exceptionalorbitinflchar}

In this section, we will produce a complete list of unipotent infinitesimal characters attached to birationally rigid orbits. A list of such orbits is provided in Proposition \ref{prop:listofbirigid}. For the rigid orbits, $\gamma_0(\OO)$ was computed in \cite{LMBM}. This leaves only three (non-rigid) birationally rigid orbits, namely $A_2+A_1$ and $A_4+A_1$ in $E_7$ and $A_4+A_1$ in $E_8$. For these orbits we compute $\gamma_0(\OO)$ using Proposition \ref{prop:inflcharbirigidorbit}.

The calculations below involve a number of easy `micro-computations,' which we will catalogue here for the reader's convenience (and to avoid repeating references and explanations in the calculations below):

\begin{enumerate}
    \item Given a nilpotent orbit $\OO$, determine the finite set $\mathcal{P}_{\mathrm{rig}}(\OO)$ (see (\ref{eq:Prig})). This is easily deducible from the tables in \cite[Sec 4]{deGraafElashvili}.
    \item Determine the unipotent infinitesimal character $\gamma_0(\OO)$ attached to a birationally rigid orbit in a classical Lie algebra. This was carried out in \cite[Sec 8]{LMBM}, see \cite[Prop 8.2.3]{LMBM} for explicit formulas in terms of partitions.
    \item Given a Levi subgroup $L \subset G$, express the fundamental weights for $L$ in terms of fundamental weights for $G$. This is a somewhat tedious computation involving the Cartan matrices for $L$ and $G$, which can be expedited using {\fontfamily{cmtt}\selectfont atlas}.
    \item Given an arbitrary weight $\lambda \in \fh^*$, compute the (unique) dominant $W$-conjugate $\lambda^+ \in \fh^*$. This is quite difficult to do by hand, the {\fontfamily{cmtt}\selectfont atlas} command `make\_dominant' is helpful.
\end{enumerate}

We now begin the calculations:

\begin{itemize}
    \item $A_2+A_1 \subset E_7$. By \cite[Sec 4]{deGraafElashvili}, $\OO$ is induced from the rigid orbit $\OO_L=A_1$ of the Levi $L=L(E_6)$. The infinitesimal character $\gamma_0(\OO_L)$ was computed in \cite{LMBM}. It is $\rho(\fl)-\varpi$, where $\varpi$ is the fundamental weight for $\fl$ corresponding to the central node in $E_6$. In fundamental weight coordinates for $\fg$
    $$\gamma_0(\OO_L) = (1,1,1,0,1,1,-6).$$
    Applying Proposition \ref{prop:inflcharbirigidorbit} and conjugating by $W$ we get
    $$\gamma_0(\OO) = (1,0,0,1,0,1,0).$$
    
    \item $A_4+A_1 \subset E_7$. By \cite[Sec 4]{deGraafElashvili}, $\OO$ is induced from the  orbit $\OO_L=\{0\}$ of the Levi $L=L(A_4+A_1;1,2,3,4,6)$. By Proposition \ref{prop:inflcharbirigidorbit}
    $$\gamma_0(\OO) = \gamma_0(\OO_L) = \rho(\fl) = \frac{1}{2}(2,2,2,2,-7,2,-1)$$
    Conjugating by $W$ we get
    $$\gamma_0(\OO) = \frac{1}{2}(1,0,0,1,0,1,0).$$
    
    \item $A_4+A_1 \subset E_8$. By \cite[Sec 4]{deGraafElashvili}, $\OO$ is induced from the rigid orbit $\OO_L=A_1 \times \{0\}$ of the Levi $L=L(E_6+A_1;1,2,3,4,5,6,8)$. The infinitesimal character $\gamma_0(\OO_L)$ was computed in \cite{LMBM}. It is $\rho(\fl)-\varpi$, where $\varpi$ is the fundamental weight for $\fl$ corresponding to the central node in the $E_6$ factor. In fundamental weights coordinates for $\fg$
    $$\gamma_0(\OO_L) = \frac{1}{2}(2,2,2,0,2,2,-13,2).$$
    Applying Proposition \ref{prop:inflcharbirigidorbit} and conjugating by $W$ we get
    $$\gamma_0(\OO) = \frac{1}{2}(1,0,0,1,0,1,0,2).$$
\end{itemize}

\begin{table}[H]
    \begin{tabular}{|c|c|c|} \hline
       $\OO$ &  rigid? & $\gamma_0(\OO)$ \\ \hline 
     $\{0\}$ & yes & \cellcolor{blue!20}$(1,1)$\\ \hline
     $A_1$ & yes & $\frac{1}{3}(3,1)$ \\ \hline
     $\widetilde{A}_1$ & yes & $\frac{1}{2}(1,1)$ \\ \hline
    \end{tabular}
    \caption{Unipotent infinitesimal characters attached to birationally rigid orbits: type $G_2$. Special unipotent characters are highlighted in blue.}
    \label{table:orbitG2}
\end{table}

\begin{table}[H]
    \begin{tabular}{|c|c|c|} \hline
       $\OO$ &  rigid? & $\gamma_0(\OO)$ \\ \hline 
        $\{0\}$ & yes & \cellcolor{blue!20}$(1,1,1,1)$ \\ \hline
        $A_1$ & yes & $\frac{1}{2}(1,1,2,2)$ \\ \hline
        $\widetilde{A}_1$ & yes & \cellcolor{blue!20}$(1,0,1,1)$ \\ \hline
        $A_1+\widetilde{A}_1$ & yes & \cellcolor{blue!20}$(1,0,1,0)$ \\ \hline
        $A_2+\widetilde{A}_1$ & yes & $\frac{1}{4}(1,1,2,2)$  \\ \hline
        $\widetilde{A}_2+A_1$ & yes & $\frac{1}{3}(1,1,1,1)$ \\ \hline
    \end{tabular}
    \caption{Unipotent infinitesimal characters attached to birationally rigid orbits: type $F_4$. Special unipotent characters are highlighted in blue.}
    \label{table:orbitF4}
\end{table}

\begin{table}[H]\label{table:E6orbits}
    \begin{tabular}{|c|c|c|} \hline
       $\OO$ &  rigid? & $\gamma_0(\OO)$ \\ \hline 
        $\{0\}$ & yes & \cellcolor{blue!20}$(1,1,1,1,1,1)$\\ \hline
        $A_1$ & yes & \cellcolor{blue!20}$(1,1,1,0,1,1)$  \\ \hline
        $3A_1$ & yes & $\frac{1}{2}(1,1,1,1,1,1)$  \\ \hline
        $2A_2+A_1$ & yes & $\frac{1}{3}(1,1,1,1,1,1)$ \\ \hline
    \end{tabular}
    \caption{Unipotent infinitesimal characters attached to birationally rigid orbits: type $E_6$. Special unipotent characters are highlighted in blue.}
    \label{table:orbitE6}
\end{table}

\begin{table}[H]\label{table:E7orbits}
    \begin{tabular}{|c|c|c|} \hline
       $\OO$ &  rigid? & $\gamma_0(\OO)$  \\ \hline 
        $\{0\}$ & yes & \cellcolor{blue!20}$(1,1,1,1,1,1,1)$  \\ \hline
        $A_1$ & yes &  \cellcolor{blue!20}$(1,1,1,0,1,1,1)$ \\ \hline
        $2A_1$ & yes & \cellcolor{blue!20}$(1,1,1,0,1,0,1)$  \\ \hline
        $(3A_1)'$ & yes & $\frac{1}{2}(1,1,1,1,1,1,2)$ \\ \hline
        $4A_1$ & yes & $\frac{1}{2}(1,1,1,1,1,1,1)$ \\ \hline
        $A_2+A_1$ & no & \cellcolor{blue!20}$(1,0,0,1,0,1,0)$ \\ \hline
        $A_2+2A_1$ & yes & \cellcolor{blue!20}$(1,0,0,1,0,0,1)$ \\ \hline
        $2A_2+A_1$ & yes & $\frac{1}{3}(1,1,1,1,1,1,1)$ \\ \hline
        $(A_3+A_1)'$ & yes & $\frac{1}{2}(1,1,0,1,0,1,1)$ \\ \hline
        $A_4+A_1$ & no & \cellcolor{blue!20}$\frac{1}{2}(1,0,0,1,0,1,0)$ \\ \hline
    \end{tabular}
    \caption{Unipotent infinitesimal characters attached to birationally rigid orbits: type $E_7$. Special unipotent characters are highlighted in blue.}
    \label{table:orbitE7}
\end{table}

\begin{table}[H]
    \begin{tabular}{|c|c|c|} \hline
       $\OO$ &  rigid? & $\gamma_0(\OO)$ \\ \hline 
        $\{0\}$ & yes & \cellcolor{blue!20}$(1,1,1,1,1,1,1,1)$ \\ \hline
        
        $A_1$ & yes & \cellcolor{blue!20}$(1,1,1,0,1,1,1,1)$ \\ \hline
        
        $2A_1$ & yes & \cellcolor{blue!20}$(1,1,1,0,1,0,1,1)$ \\ \hline
        
        $3A_1$ & yes & $\frac{1}{2}(1,1,1,1,1,1,2,2)$  \\ \hline
        $4A_1$ & yes & $\frac{1}{2}(1,1,1,1,1,1,1,1)$ \\ \hline
        
        $A_2+A_1$ & yes & \cellcolor{blue!20}$(1,0,0,1,0,1,0,1)$ \\ \hline
        
        $A_2+2A_1$ & yes & \cellcolor{blue!20}$(1,0,0,1,0,0,1,1)$ \\ \hline
        
        $A_2+3A_1$ & yes & $\frac{1}{2}(1,1,1,0,1,1,1,1)$ \\ \hline
        $2A_2+A_1$ & yes & $\frac{1}{3}(1,1,1,1,1,1,1,3)$ \\ \hline
        $A_3+A_1$ & yes & $\frac{1}{2}(1,1,0,1,0,1,1,2)$  \\ \hline
        $2A_2+2A_1$ & yes & $\frac{1}{3}(1,1,1,1,1,1,1,1)$ \\ \hline
        $A_3+2A_1$ & yes & $\frac{1}{2}(1,1,1,0,1,0,1,1)$ \\ \hline
        $D_4(a_1)+A_1$ & yes & \cellcolor{blue!20}$(0,0,0,1,0,0,1,0)$ \\ \hline
        
        $A_3+A_2+A_1$ & yes & $\frac{1}{2}(1,0,0,1,0,1,1,1)$ \\ \hline
        $A_4+A_1$ & no & \cellcolor{blue!20}$\frac{1}{2}(1,0,0,1,0,1,0,2)$ \\ \hline
        $2A_3$ & yes & $\frac{1}{4}(1,1,1,1,1,1,1,1)$ \\ \hline
        $A_4+A_3$ & yes & $\frac{1}{5}(1,1,1,1,1,1,1,1)$ \\ \hline
        $A_5+A_1$ & yes & $\frac{1}{6}(2,2,,1,1,1,1,1,1)$ \\ \hline
        $D_5(a_1)+A_2$ & yes & $\frac{1}{4}(1,1,1,0,1,1,1,1)$ \\ \hline
    \end{tabular}
    \caption{Unipotent infinitesimal characters attached to birationally rigid orbits: type $E_8$. Special unipotent characters are highlighted in blue.}
    \label{table:orbitE8}
\end{table}

\subsubsection{Birationally rigid covers}\label{subsec:exceptionalcoverinflchar}

In this section, we will produce a complete list of unipotent ideals attached to birationally rigid covers $\widetilde{\OO}$ of birationally induced orbits. In addition to (1)-(4) of Section \ref{subsec:exceptionalorbitinflchar}, we will repeatedly carry out the following computations:
 
\begin{enumerate}[resume]
    \item Given a nilpotent orbit $\OO$, determine the codimension 2 orbits $\OO_k \subset \overline{\OO}$, their singularities, and the normalizations thereof. These are evident from the incidence diagrams in \cite[Sec 13]{fuetal2015}.
    \item Given a nilpotent orbit $\OO$, determine the fundamental group $\pi_1(\OO)$. See \cite[Sec 6.1]{CM} for classical types and \cite[Sec 8.4]{CM} for exceptional.
    \item Determine the reductive part $\mathfrak{r}$ of the centralizer of $e \in \OO$. See \cite[Sec 6.1]{CM} for classical types and \cite[Sec 13.1]{CM} for exceptional.
    \item Given a nilpotent orbit $\OO$ in a Levi subalgebra $\fl \subset \fg$, compute the induced orbit $\Ind^{\fg}_{\fl}\OO$. If $\fg$ is exceptional, we use \cite[Sec 4]{deGraafElashvili}. If $\fg$ is classical, we use the well-known formulas involving partitions, see \cite[Sec 7.3]{CM}.
    \item Given a singularity $\Sigma_k \subset \Spec(\CC[\OO])$ of type $A_1$, determine the integer $c_k$ (cf. (\ref{eq:defofck})). By definition
    $$c_k = 2|\Pic(\OO_{M_k})||\Pic(\OO)|^{-1} = 2|\pi_1(\OO_{M_k})_{\mathrm{ab}}||\pi_1(\OO)_{\mathrm{ab}}|^{-1},$$
    where $H_{\mathrm{ab}}$ denotes the abelianization of $H$. 
    \item For a given Levi subgroup $M \subset G$ determine a set of generators $\tau_i$ for the free abelian group $\fX(M)$. If $M$ is standard, then we can take $\tau_i$ to be the fundamental weights for $G$ corresponding to the simple roots not contained in $M$. More generally, if $\beta_1,...,\beta_m$ are simple roots for $M$, then $\fX(M)$ is identified with the lattice
    $$\{\lambda \in \Lambda \mid \langle \lambda, \beta_i^{\vee}\rangle =0, \ 1 \leq i \leq m\}.$$
    The co-roots $\beta_i^{\vee}$ can be computed by hand or using the {\fontfamily{cmtt}\selectfont atlas} software.
\end{enumerate}

We now begin the calculations:

\vspace{3mm}

\paragraph{Type $G_2$}.

\begin{center}
$$\dynkin[label,label macro/.code={\alpha_{\drlap{#1}}},edge
length=.75cm] G2$$
\end{center}

\begin{itemize}
    \item $G_2(a_1)$. Note that $\OO$ is even and $L_{\OO} = L(A_1;2)$. Hence by Lemma \ref{lem:even}
    $$(L,\OO_L) = (L(A_1;2), \{0\})$$
    There is one codimension 2 leaf $\fL_1 \subset \Spec(\CC[\OO])$ and the corresponding singularity is of type $A_1$. Thus
    
    \begin{center}
        \begin{tabular}{|c|c|c|c|c|} \hline
            $k$ & $\Sigma_k$ & $M_k$ & $\OO_{M_k}$ & $c_k$\\ \hline
            $1$ & $A_1$ & $L(A_1;2)$ & $\{0\}$ & $1$ \\ \hline 
        \end{tabular}
    \end{center}
    Note that $\tau_1=(1,0)$. So by Proposition \ref{prop:identification2}, we have  $\delta_1=\frac{1}{2}\tau_1 = \frac{1}{2}(1,0)$. Now by Proposition \ref{prop:deltashift}
$$\gamma_0(\widehat{\OO}) = \rho(\fl) + \delta_1 = \frac{1}{2}(-3,2) + \frac{1}{2}(1,0) = (-1,1)$$
Conjugating by $W$, we get
$$\gamma_0(\widehat{\OO}) = (1,0)$$
\end{itemize}

\paragraph{Type $F_4$}.

\begin{center}
$$\dynkin[label,label macro/.code={\alpha_{\drlap{#1}}},edge
length=.75cm] F4$$
\end{center}

\begin{itemize}
    \item $A_2$. Note that $\OO$ is even and $L_{\OO} = L(C_3;2,3,4)$. Hence by Lemma \ref{lem:even}
    $$(L,\OO_L) = (L(C_3;2,3,4), \{0\}).$$
    There is one codimension 2 leaf $\fL_1 \subset \Spec(\CC[\OO])$ and the corresponding singularity is of type $A_1$. Thus
        \begin{center}
        \begin{tabular}{|c|c|c|c|c|} \hline
            $k$ & $\Sigma_k$ & $M_k$ & $\OO_{M_k}$ & $c_k$\\ \hline
            $1$ & $A_1$ & $L(C_3;2,3,4)$ & $\{0\}$ & $1$\\ \hline 
        \end{tabular}
    \end{center}
    Note that $\tau_1=(1,0,0,0)$. So by Proposition \ref{prop:identification2}, we have $\delta_1=\frac{1}{2}\tau_1=\frac{1}{2}(1,0,0,0)$. Now by Proposition \ref{prop:deltashift}
    $$\gamma_0(\widehat{\OO}) = \rho(\fl) + \delta_1 = (-3,1,1,1) + \frac{1}{2}(1,0,0,0) = \frac{1}{2}(-5,2,2,2).$$
    Conjugating by $W$ we get
    $$\gamma_0(\widehat{\OO}) = \frac{1}{2}(1,1,0,2).$$

    \item $B_2$. There is a single codimension 2 leaf $\fL_1\subset \Spec(\CC[\OO])$ and the corresponding singularity is of type $A_1$. We have $\fr=A_1$, and therefore $\dim \fP^X=1$. We have 
    $$\mathcal{P}_{\mathrm{rig}}(\OO) = \{(L(C_3;2,3,4),(2,1^4))\},$$
    and therefore by Lemma \ref{lem:Prigunique}
$$(L,\OO_L) = (L(C_3;2,3,4),(2,1^4)).$$
    It follows that $M_1=L$, and
    \begin{center}
        \begin{tabular}{|c|c|c|c|c|} \hline
            $k$ & $\Sigma_k$ & $M_k$ & $\OO_{M_k}$ & $c_k$\\ \hline
            $1$ & $A_1$ & $L(C_3;2,3,4)$ & $(2,1^4)$ & $2$ \\ \hline 
        \end{tabular}
    \end{center}
Note that $\tau_1=(1,0,0,0)$. So by Proposition \ref{prop:identification2}, we have $\delta_1=\frac{1}{4}\tau_1=\frac{1}{4}(1,0,0,0)$. In standard coordinates on $\mathrm{Sp}(6)$, $\gamma_0(\OO_L) = \frac{1}{2}(5,3,1)$, see \cite[Prop 8.2.3]{LMBM}. So in our coordinates
$$\gamma_0(\OO_L) =  \frac{1}{4}(-9,2,4,4)$$
Now by Proposition \ref{prop:deltashift}
 $$\gamma_0(\widehat{\OO}) = \gamma_0(\OO_L) +\delta_1 = \frac{1}{4}(-9,2,4,4) + \frac{1}{4}(1,0,0,0) = \frac{1}{2}(-4,1,2,2)$$
Conjugating by $W$ we get
$$\gamma_0(\widehat{\OO}) = \frac{1}{2}(0,1,0,2)$$

    \item $C_3(a_1)$. There are two codimension orbits in the closure of $\OO$, and the corresponding singularities are of types $m$ and $2A_1$. Hence, there is a single codimension 2 leaf $\fL_1\subset \Spec(\CC[\OO])$ and the corresponding singularity is of type $A_1$. Note that $\fr=A_1$, and therefore $\dim \fP^X=1$. We have
    $$\mathcal{P}_{\mathrm{rig}}(\OO) = \{(L(B_3;1,2,3), (2^2,1^2))\},$$
    and so by Lemma \ref{lem:Prigunique}
    $$(L,\OO_L) = (L(B_3;1,2,3), (2^2,1^2)).$$
    Thus, $M_1=L$, and
    \begin{center}
        \begin{tabular}{|c|c|c|c|c|} \hline
            $k$ & $\Sigma_k$ & $M_k$ & $\OO_{M_k}$ & $c_k$\\ \hline
            $1$ & $A_1$ & $L(B_3;1,2,3)$ & $(2^2,1^2)$ & $1$\\ \hline 
        \end{tabular}
    \end{center}
Note that $\tau_1=(0,0,0,1)$. So by Proposition \ref{prop:identification2}, we have $\delta_1=\frac{1}{2}\tau_1=\frac{1}{2}(0,0,0,1)$. In standard coordinates on $\mathrm{SO}(7)$, $\gamma_0(\OO_L) = \frac{1}{2}(3,2,1)$, see \cite[Prop 8.2.3]{LMBM}. So in our coordinates
$$\gamma_0(\OO_L)=\frac{1}{2}(1,1,2,-6).$$
Now by Proposition \ref{prop:deltashift}
$$\gamma_0(\widehat{\OO}) = \gamma_0(\OO_L) + \delta_1 = \frac{1}{2}(1,1,2,-6) + \frac{1}{2}(0,0,0,1) = \frac{1}{2}(1,1,2,-5).$$
Conjugating by $W$ we get
$$\gamma_0(\widehat{\OO}) = \frac{1}{2}(1,0,1,1).$$

    \item $F_4(a_3)$. Note that $\OO$ is even and $L_{\OO}= L(A_1+A_2;1,3,4)$. Hence by Lemma \ref{lem:even}
    $$(L,\OO_L) = (L(A_1+A_2;1,3,4),\{0\})$$
    There is a single codimension 2 leaf $\fL_1 \subset \Spec(\CC[\OO])$ and the corresponding singularity is of type $A_1$. Thus
    \begin{center}
        \begin{tabular}{|c|c|c|c|c|} \hline
            $k$ & $\Sigma_k$ & $M_k$ & $\OO_{M_k}$ & $c_k$\\ \hline
            $1$ & $A_1$ & $L(A_1+A_2;1,3,4)$ & $\{0\}$  & $1$\\ \hline 
        \end{tabular}
    \end{center}
Note that $\tau_1=(0,1,0,0)$. So by Proposition \ref{prop:identification2}, we have $\delta_1=\frac{1}{2}\tau_1=\frac{1}{2}0,1,0,0)$. Now by Proposition \ref{prop:deltashift}
    $$\gamma_0(\widehat{\OO}) = \rho(\fl) + \delta_1 = \frac{1}{2}(2,-3,2,2) + \frac{1}{2}(0,1,0,0) = (1,-1,1,1)$$
    Conjugating by $W$ we get
    $$\gamma_0(\widehat{\OO})= (0,0,1,0)$$

\end{itemize}

\paragraph{Type $E_6$}.

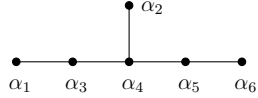
\begin{center}
$$\dynkin[label,label macro/.code={\alpha_{\drlap{#1}}},edge
length=.75cm] E6$$
\end{center}

\begin{itemize}
    \item $A_2$. Note that $\OO$ is even and $L_{\OO} = L(A_5)$. Hence by Lemma \ref{lem:even}
    $$(L, \OO_L) = (L(A_5),\{0\})$$
    There is a single codimension 2 leaf $\fL_1 \subset \Spec(\CC[\OO])$ and the corresponding singularity is of type $A_1$. Thus
    \begin{center}
        \begin{tabular}{|c|c|c|c|c|} \hline
            $k$ & $\Sigma_k$ & $M_k$ & $\OO_{M_k}$ & $c_k$ \\ \hline
            $1$ & $A_1$ & $L(A_5)$ & $\{0\}$ & $1$\\ \hline 
        \end{tabular}
    \end{center}
Note that $\tau_1=(0,1,0,0,0,0)$. So by Proposition \ref{prop:identification2}, we have $\delta_1=\frac{1}{2}\tau_1=\frac{1}{2}(0,1,0,0,0,0)$. Now by Proposition \ref{prop:deltashift}
    $$\gamma_0(\widehat{\OO}) = \rho(\fl) + \delta_1 = \frac{1}{2}(2,-9,2,2,2,2) + \frac{1}{2}(0,1,0,0,0,0) = (1,-4,1,1,1,1)$$
    Conjugating by $W$ we get
    $$\gamma_0(\widehat{\OO}) = (1,0,0,1,0,1)$$

    \item $2A_2$. Note that $\OO$ is even and $L_{\OO} =L(D_4)$. Hence by Lemma \ref{lem:even}
    $$(L,\OO_L) = (L(D_4),\{0\}).$$
    There is a single codimension 2 leaf $\fL_1 \subset \Spec(\CC[\OO])$ and the corresponding singularity is of type $A_2$. Thus
    \begin{center}
        \begin{tabular}{|c|c|c|c|c|} \hline
            $k$ & $\Sigma_k$ & $M_k$ & $\OO_{M_k}$ & $c_k$\\ \hline
            $1$ & $A_2$ & $L(D_4)$ & $\{0\}$ & - \\ \hline 
        \end{tabular}
    \end{center}
We have
    $$\tau_1(1) = (1,0,0,0,0,0) \qquad \tau_2(1) = (0,0,0,0,0,1)$$
    So by Proposition \ref{prop:identification1}, $\delta = \frac{1}{3}(1,0,0,0,0,1)$. Now by Proposition \ref{prop:deltashift}
    $$\gamma_0(\widehat{\OO}) = \rho(\fl) + \delta_1 = (-3,1,1,1,1,-3)+\frac{1}{3}(1,0,0,0,0,1) = \frac{1}{3}(-8,1,1,1,1,-8)$$
    Conjugating by $W$ we get
    $$\gamma_0(\widehat{\OO}) = \frac{1}{3}(1,3,1,1,1,1).$$
    
        \item $D_4(a_1)$. Note that $\OO$ is even and $L_{\OO} = L(2A_2+A_1)$. Hence by Lemma \ref{lem:even}
    $$(L,\OO_L) = (L(2A_2+A_1), \{0\}).$$
    There is a single codimension 2 leaf $\fL_1 \subset \Spec(\CC[\OO])$ and the corresponding singularity is of type $A_1$. Thus
    \begin{center}
        \begin{tabular}{|c|c|c|c|c|} \hline
            $k$ & $\Sigma_k$ & $M_k$ & $\OO_{M_k}$  & $c_k$\\ \hline
            $1$ & $A_1$ & $L(2A_2+A_1)$ & $\{0\}$ & $1$ \\ \hline 
        \end{tabular}
    \end{center}
Note that $\tau_1=(0,0,0,1,0,0)$. So by Proposition \ref{prop:identification2}, we have $\delta_1=\frac{1}{2}\tau_1=\frac{1}{2}(0,0,0,1,0,0)$. Now by Proposition \ref{prop:deltashift}
    $$\gamma_0(\widehat{\OO}) = \rho(\fl) + \delta_1 = \frac{1}{2}(2,2,2,-5,2,2) + \frac{1}{2}(0,0,0,1,0,0) = (1,1,1,-2,1,1).$$
    Conjugating by $W$ we get
    $$\gamma_0(\widehat{\OO}) = (0,0,0,1,0,0).$$

    \item $A_5$. We have
    $$\mathcal{P}_{\mathrm{rig}}(\OO) = \{(L(D_4),(3,2^2,1))\}$$
   There is a single codimension 2 leaf $\fL_1\subset \Spec(\CC[\OO])$ which maps to the orbit $\OO_1=A_4+A_1$, and the corresponding singularity is of type $A_2$. Since $\pi_1(\fL_1)=\pi_1(\OO_1)=1$, the monodromy action is trivial. We have $\fr=A_1$, and therefore $\dim \fP^X=2$. So by Lemma \ref{lem:Prigunique}
    $$(L,\OO_L) = (L(D_4), (3,2^2,1)).$$
    Therefore $M_1=L$, and we have
    \begin{center}
        \begin{tabular}{|c|c|c|c|c|} \hline
            $k$ & $\Sigma_k$ & $M_k$ & $\OO_{M_k}$ & $c_k$ \\ \hline
            $1$ & $A_2$ & $L(D_4)$ & $(3,2^2,1)$ & -\\ \hline 
        \end{tabular}
    \end{center}
Since $M_k$ is standard, we have
    $$\tau_1 = (1,0,0,0,0,0) \qquad \tau_2 = (0,0,0,0,0,1)$$
    We claim that the map $\eta$ is given by $\eta(\tau_i)=2\omega_i$ for $i=1,2$. By \cite[Proposition 7.1.2]{LMBM}, there is a short exact sequence
    $$0\to \Pic(Y)\to \Cl(Y)\to \Pic(\OO_L)\to 0$$
    We have $\Pic(\OO_L)\simeq \pi_1(\OO_L)_{\mathrm{ab}}\simeq \ZZ_2^2$, and $\Pic(Y)\simeq \Pic(G/P)\simeq \ZZ^2$, and $\Pic^a(Y)$ is spanned by $\tau_1$ and $\tau_2$ under this identification. Similarly to \cite[Lem 7.7.5]{LMBM}, there is a short exact sequence corresponding to the two irreducible components of the exceptional divisor
    $$0\to \ZZ^2\to \Cl(Y)\to \Pic(\OO)\to 0.$$
    From an {\fontfamily{cmtt}\selectfont atlas} computation, we see that $\Pic(\OO)\simeq \pi_1(\OO)_{\mathrm{ab}}\simeq \ZZ_3$. Write $T$ for the sublattice of $\Cl(Y)$ spanned by the components of the exceptional divisor. It follows that $\Cl(\ZZ^2)\simeq \ZZ^2$, and moreover the sublattice $3\Pic(Y)$ in $\Cl(Y)$ is contained in $T$. The quotient of the lattice $\Cl(Y)$ by $3\Pic(\OO)$ is $\ZZ_3^2\times \ZZ_2^2$, and hence there is a short exact sequence
    $$0\to 3\Pic(Y)\to T\to \ZZ_3\times \ZZ_2^2\to 0.$$
    
    Similarly, we have a sublattice $T_{\Sigma}\subset \Cl(\fS)\simeq \Pic(\fS)$ spanned by the irreducible components of the exceptional divisor of $\fS\to \Sigma$. We have a short exact sequence
    $$0\to T_{\Sigma}\to \Cl(\fS)\to \ZZ_3\to 0.$$
    Therefore, $3\Pic(\fS)$ is a sublattice of $T_{\Sigma}$, and the cokernel of the embedding $3\Pic(\fS)\to T_{\Sigma}$ is isomorphic to $\ZZ_3$.
    
    Let $f: \Pic(Y)\to \Pic(\fS)$ be the restriction map, and let $g: T\to T_{\Sigma}$ be the map which takes the irreducible component of the exceptional divisor of $Y\to X$ to its intersection with $\fS$, an irreducible component of the exceptional divisor of $\fS\to \Sigma$. Since $Y$ and $\fS$ are $\QQ$-terminal, there is a commutative diagram.
    
    \begin{center}
\begin{tikzcd}[
  ar symbol/.style = {draw=none,"#1" description,sloped},
  isomorphic/.style = {ar symbol={\simeq}},
  equals/.style = {ar symbol={=}},
  ]
  0\ar[r]& 3\Pic(\OO) \ar[d, "3f"] \ar[r]& T \ar[d, "g"] \ar[r] & \ZZ_3\times \ZZ_2^2 \ar[r] \ar[d, "h"]& 0\\
   0\ar[r]& 3\Pic(\fS) \ar[r]& T_{\Sigma}\ar[r] & \ZZ_3 \ar[r] & 0
\end{tikzcd}
\end{center}
    Note that $\Coker(f)\simeq \Coker(3f)\simeq \Ker(h)\simeq \ZZ_2^2$, and therefore $\eta(\tau_i)$ and $\eta(\tau_2)$ span a lattice in $\Pic(\fS)=\ZZ \langle\omega_1, \omega_2 \rangle$ with quotient isomorphic to $\ZZ_2^2$.
    
    Since there is a unique pair $(L, \OO_L)$ in $\mathcal{P}_{\mathrm{rig}}(\OO)$, $X$ admits a unique $\QQ$-factorial terminalization, up to isomorphism, see \cite[Lemma 7.2.4]{LMBM}. For an algebraic variety $Z$ with a projective morphism $Z \to S$, write $\Pic^a(Z) \subset \Pic(Z)$ for the semigroup of relatively ample line bundles. By \cite[Proposition 2.17]{BPW} $\eta$ induces an isomorphism between $\RR_{>0}\Pic^a(Y)$ and the fundamental domain for the $W$-action on $\fP^X=\fP^{\fS}$, i.e. $\RR_{>0}\Pic^a(\fS)$. We note that $\Pic^a(Y)$ and $\Pic^a(\fS)$ are spanned by $\tau_1, \tau_2$ and $\omega_1, \omega_2$ respectively. So in these bases both $\eta$ and $\eta^{-1}$ are given by 2-by-2 matrices with nonnegative coefficients. Hence $\eta$ is given by a diagonal matrix, up to permutations of bases. Since the cokernel of the map $f: \Pic(Y)\to \Pic(\fS)$ is equal to $\ZZ_2^2$, it follows that $\eta(\tau_1)=2\omega_1$ and $\eta(\tau_2)=2\omega_2$. 
    
    Now we have $\delta = \frac{1}{6}(\tau_1 + \tau_2) = \frac{1}{6}(1,0,0,0,0,1)$. By \cite[Prop 8.3]{LMBM}, $\gamma_0(\OO_L) = \frac{1}{2}\rho(\fl)$. Thus by Proposition \ref{prop:deltashift}
    $$\gamma_0(\widehat{\OO}) = \frac{1}{2}\rho(\fl) + \delta_1 = \frac{1}{2}(-3,1,1,1,1,-3) + \frac{1}{6}(1,0,0,0,0,1) = \frac{1}{6}(-8,3,3,3,3,-8)$$
    Conjugating by $W$ we get
    $$\gamma_0(\widehat{\OO}) = \frac{1}{6}(1,3,1,1,1,1).$$
    
    \item $E_6(a_3)$. Note that $\OO$ is even and $L_{\OO}=L(3A_1;2,3,5)$. Hence by Lemma \ref{lem:even}
    $$(L,\OO_L) = (L(3A_1;2,3,5), \{0\}).$$
    There are two codimension 2 leaves $\fL_1,\fL_2 \subset \Spec(\CC[\OO])$, corresponding to the orbits $\OO_1=A_5$ and $\OO_2 = D_5(a_1)$. The corresponding singularities are of types $A_1$ and $A_2$, respectively. We will compute the 
    pairs $(M_k,\OO_{M_k})$ using Lemma \ref{lem:findMk}. First, consider the pair
    $$(M_1, \OO_{M_1}) = (L(A_5; 2,3,5,\theta_1,\theta_2), (3^2))$$
    where $\theta_1$ and $\theta_2$ are the roots
    $$\theta_1=-\alpha_2-\alpha_3-\alpha_4-\alpha_5-\alpha_6 \qquad \theta_2=\alpha_1+\alpha_2+\alpha_3+2\alpha_4+\alpha_5+\alpha_6.$$
    An {\fontfamily{cmtt}\selectfont atlas} computation shows that this indeed defines of Levi (of the stated type). By construction, $L \subset M_1$, and $\OO_{M_1} = \Ind^{M_1}_L \{0\}$. Furthermore, $\dim \fX(\fm_2) = 1 = \dim\fP_2$. Note that $\overline{\OO}_{M_1}$ contains a single codimension 2 orbit, namely $\OO_{M_1,2} = \OO_{(3,2,1)}$, and the corresponding singularity is of type $A_2$. Finally, note that
    \begin{align*}
    \Ind^G_{M_1} \OO_{M_1,2} &= \Ind^G_{M_1}(\Ind^{M_1}_{L(A_2+A_1;2,5,\theta_1)} \{0\})\\
    &=\Ind^G_{L(A_2+A_1;2,5,\theta_1)} \{0\}
    \end{align*}
    Up to conjugation by $G$, there is a unique Levi in $G$ of type $A_2+A_1$. So the final induction coincides with $D_5(a_1) = \OO_2$, see \cite[Sec 4]{deGraafElashvili}. It now follows from Lemma \ref{lem:findMk} that the pair $(M_1,\OO_{M_1})$ is adapted to $\fL_1$.
    
    Next, define
    $$(M_2,\OO_{M_2}) = (L(D_4),\OO_{(3^2,1^2)}).$$
    By construction, $L \subset M_2$, and $\OO_{M_2} = \mathrm{Ind}^{M_2}_L \{0\}$. Furthermore, $\dim \fX(\fm_2) = 2 = \dim \fP_2$. Note that $\overline{\OO}_{M_2}$ contains a single codimension 2 orbit, namely $\OO_{M_2,1} = \OO_{(3,2^2,1)}$, and the corresponding singularity is of type $A_1$. Finally $\Ind^G_{M_2} \OO_{M_2,1} = A_5=\OO_1$, see \cite[Sec 4]{deGraafElashvili}. So by Lemma \ref{lem:findMk}, the pair $(M_2,\OO_{M_2})$ is adapted to $\fL_2$. 
    \begin{center}
        \begin{tabular}{|c|c|c|c|c|} \hline
            $k$ & $\Sigma_k$ & $M_k$ & $\OO_{M_k}$ & $c_k$\\ \hline
            $1$ & $A_1$ & $L(A_5;2,3,5,\theta_1,\theta_2)$ & $(3^2)$ & $1$\\ \hline 
            $2$ & $A_2$ & $L(D_4)$ & $(3^2,1^2)$ & - \\ \hline
        \end{tabular}
    \end{center}
By definition $\tau_1(1)$ is (either) generator of the free abelian group
    \begin{align*}
    \fX(M_1) &= \{\lambda \in \Lambda \mid \langle \lambda, \alpha^{\vee}\rangle =0, \ \alpha \in \Delta(\fm,\fh)\}\\
    &= \{(x,0,0,y,0,z) \in \ZZ^6 \mid x+y = y+z =0\}
    \end{align*}
    So we may take $\tau_1(1)=(1,0,0,-1,0,0,1)$ and hence $\delta_1=\frac{1}{2}(1,0,0,-1,0,1)$.
    
    On the other hand, $M_2$ is standard. Hence
    $$\tau_2(1) = (1,0,0,0,0,0) \qquad \tau_2(2) = (0,0,0,0,0,1).$$
    We claim that $\eta_2(\tau_i(2))=\omega_i(2)$. We note that it suffices to prove the analogous assertion for the orbit $\OO_M\subset \fm^*$, where $M=M_1$. Namely, we have $M=L(A_5;2,3,5,\theta_1,\theta_2)$, and $\OO_M=\OO_{(3^2)}$. We note that this orbit is Richardson, and $(3A_1, \{0\})$ is the only pair in $P_{\mathrm{rig}}(\OO_M)$. Thus, by \cite[Lemma 7.2.4]{LMBM} $Y_M\simeq T^*(Q/P_M)$ is the unique up to an isomorphism $\QQ$-terminalization of $X_M$. Arguing as in the case of orbit $A_5\subset E_6$ we see that $\eta_2$ is diagonal with positive integer coefficients in the bases $\tau_1(2), \tau_2(2)$ and $\omega_1(2), \omega_2(2)$ respectively. Consider the sublattices $T$ and $T_\Sigma$ of $\Cl(Y_M)$ and $\Cl(\fS)$ spanned by the irreducible components of the exceptional divisors of $Y_M\to X_M$ and $\fS\to \Sigma$ respectively. Let $f: \Pic(Y_M)\to \Pic(\fS)$ and $g: T\xrightarrow{\sim} T_\Sigma$ be the maps constructed as for the orbit $A_5$. Note that both $Y_M$ and $\fS$ are smooth, and thus $\Pic(Y_M)\simeq \Cl(Y_M)$ and $\Pic(\fS)\simeq \Cl(\fS)$. Since $\pi_1(\OO_M)\simeq \ZZ_3$, we have both $T\subset \Cl(Y_M)$ and $T_\Sigma\subset \Cl(\fS)$ are of index $3$, and therefore $f$ is an isomorphism. Thus, both diagonal entries of $\eta$ must be $1$, and thus $\eta_2(\tau_i(2))=\omega_i(2)$.
    
    It follows that $\delta_2 = \frac{1}{3}(\tau_2(2)+\tau_2(2)) = \frac{1}{3}(1,0,0,0,0,1)$. Now by Proposition \ref{prop:deltashift}
    \begin{align*}
        \gamma_0(\widehat{\OO}) &= \rho(\fl) + \delta_1+ \delta_2\\
        &= \frac{1}{2}(-1,2,2,-3,2,-1) + \frac{1}{2}(1,0,0,-1,0,1) + \frac{1}{3}(1,0,0,0,0,1)\\
        &= \frac{1}{3}(1,3,3,-6,3,1)
    \end{align*}
Conjugating by $W$ we get
$$\gamma_0(\widehat{\OO}) = \frac{1}{3}(0,1,1,0,1,0).$$
\end{itemize}

\paragraph{Type $E_7$}.

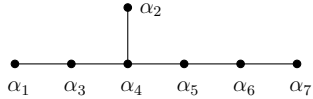
\begin{center}
$$\dynkin[label,label macro/.code={\alpha_{\drlap{#1}}},edge
length=.75cm] E7$$
\end{center}

\begin{itemize}
    \item $(3A_1)''$. Note that $\OO$ is even and $L_{\OO} = L(E_6)$. Hence by Lemma \ref{lem:even}
    $$(L,\OO_L) = (L(E_6), \{0\})$$
    There is a unique codimension 2 leaf $\fL_1 \subset \Spec(\CC[\OO])$ and the corresponding singularity is of type $A_1$. Thus
    \begin{center}
        \begin{tabular}{|c|c|c|c|c|} \hline
            $k$ & $\Sigma_k$ & $M_k$ & $\OO_{M_k}$ & $c_k$\\ \hline
            $1$ & $A_1$ & $L(E_6)$ & $\{0\}$ & $1$\\ \hline 
        \end{tabular}
    \end{center}
Note that $\tau_1=(0,0,0,0,0,0,1)$. So by Proposition \ref{prop:identification2}, we have $\delta_1 = \frac{1}{2}\tau_1 = \frac{1}{2}(0,0,0,0,0,0,1)$. Now by Proposition \ref{prop:deltashift} 
    $$\gamma_0(\widehat{\OO}) = \rho(\fl) + \delta = (1,1,1,1,1,1,-8) + \frac{1}{2}(0,0,0,0,0,0,1) = \frac{1}{2}(2,2,2,2,2,2,-15)$$
    Conjugating by $W$ we get
    $$\gamma_0(\widehat{\OO}) = \frac{1}{2}(2,1,2,1,1,1,1)$$
    
    \item $A_2$.  Note that $\OO$ is even and $L_{\OO} = L(D_6)$. Hence by Lemma \ref{lem:even}
$$(L,\OO_L) = (L(D_6), \{0\}).$$
There is a single codimension 2 leaf $\fL_1 \subset \Spec(\CC[\OO])$ and the corresponding singularity is of type $A_1$. Thus
    \begin{center}
        \begin{tabular}{|c|c|c|c|c|} \hline
            $k$ & $\Sigma_k$ & $M_k$ & $\OO_{M_k}$ & $c_k$\\ \hline
            $1$ & $A_1$ & $L(D_6)$ & $\{0\}$ & $1$ \\ \hline 
        \end{tabular}
    \end{center}
Note that $\tau_1=(1,0,0,0,0,0,0)$. So by Proposition \ref{prop:identification2}, we have $\delta_1=\frac{1}{2}\tau_1=\frac{1}{2}(1,0,0,0,0,0,0)$. Now by Proposition \ref{prop:deltashift}
$$\gamma_0(\widehat{\OO}) = \rho(\fl) + \delta = \frac{1}{2}(-15,2,2,2,2,2,2) + \frac{1}{2}(1,0,0,0,0,0,0,0) = (-7,1,1,1,1,1,1).$$
Conjugating by $W$ we get
$$\gamma_0(\widehat{\OO}) = (1,0,0,1,0,1,1).$$

    \item $A_2+3A_1$. Note that $\OO$ is even and $L_{\OO} = L(A_6)$. Hence by Lemma \ref{lem:even}
$$(L,\OO_L) = (L(A_6), \{0\}).$$
There is a unique codimension 2 leaf $\fL_1 \subset \Spec(\CC[\OO])$ and the corresponding singularity is of type $A_1$. Thus
    \begin{center}
        \begin{tabular}{|c|c|c|c|c|} \hline
            $k$ & $\Sigma_k$ & $M_k$ & $\OO_{M_k}$ & $c_k$\\ \hline
            $1$ & $A_1$ & $L(A_6)$ & $\{0\}$ & $1$ \\ \hline 
        \end{tabular}
    \end{center}
Note that $\tau_1=(0,1,0,0,0,0,0)$. So by Proposition \ref{prop:identification2}, we have $\delta_1=\frac{1}{2}\tau_1=\frac{1}{2}(0,1,0,0,0,0,0)$. Now by Proposition \ref{prop:deltashift}
$$\gamma_0(\widehat{\OO}) = \rho(\fl) + \delta_1 = (1,-6,1,1,1,1,1) + \frac{1}{2}(0,1,0,0,0,0,0) = \frac{1}{2}(2,-11,2,2,2,2,2)$$
Conjugating by $W$ we get
$$\gamma_0(\widehat{\OO}) = \frac{1}{2}(1,1,1,0,1,1,1).$$

    \item $(A_3+A_1)''$. Note that $\OO$ is even and $L_{\OO} = L(D_5;2,3,4,5,6)$. Hence by Lemma \ref{lem:even}
    $$(L,\OO_L) = (L(D_5;2,3,4,5,6), \{0\}).$$
    There are two codimension 2 leaves $\fL_1,\fL_2 \subset \Spec(\CC[\OO])$, corresponding to the orbits $\OO_1=2A_2$ and $\OO_2=A_3$. Both singularities are of type $A_1$. We will compute the pairs $(M_k,\OO_{M_k})$ using Lemma \ref{lem:findMk}. First, consider the pair
    $$(M_1,\OO_{M_1}) = (L(D_6), (3,1^9)).$$
    By construction, $L \subset M_1$ and $\OO_{M_1} = \mathrm{Ind}^{M_1}_L \{0\}$. Furthermore, $\dim \fX(\fm_1) = 1 = \dim \fP_1$. $\overline{\OO}_{M_1}$ contains a single codimension 2 orbit, namely $\OO_{M_1,2} = (2^2,1^8)$, and the corresponding singularity is of type $A_1$. Finally, note that $\Ind^G_{M_1} \OO_{M_1,2} = A_3 = \OO_2$. It now follows from Lemma \ref{lem:findMk} that $(M_1,\OO_{M_1})$ is adapted to $\fL_1$. 
    
    Next, consider the pair
    $$(M_2, \OO_{M_2}) = (L(D_5+A_1;2,3,4,5,6,\theta), \{0\} \times (2))$$
    where $\theta$ is the highest root for $G$, i.e.
    $$\theta= 2\alpha_1 + 2\alpha_2 + 3\alpha_3 + 4\alpha_4 + 3\alpha_5 + 2\alpha_6 + \alpha_7$$
    An {\fontfamily{cmtt}\selectfont atlas} computation shows that this indeed defines a Levi (of the stated type). By construction, $L \subset M_2$ and of course $\overline{\OO}_{M_2} = \mathrm{Ind}^{M_2}_L \{0\}$. Furthermore, $\dim \fX(\fm_2) = 1 = \dim \fP_2$. $\overline{\OO}_{M_2}$ contains a single codimension 2 orbit, namely $\OO_{M_2,1} =\{0\}$, and the corresponding singularity is of type $A_1$. Finally, note that by \cite[Sec 4]{deGraafElashvili} $\Ind^G_{M_2} \OO_{M_2,1} =2A_2=\OO_1$. So by Lemma \ref{lem:findMk}, $(M_2,\OO_{M_2})$ is adapted to $\fL_2$. 
    \begin{center}
        \begin{tabular}{|c|c|c|c|c|} \hline
            $k$ & $\Sigma_k$ & $M_k$ & $\OO_{M_k}$ & $c_k$ \\ \hline
            $1$ & $A_1$ & $L(D_6)$ & $(3,1^9)$ & $2$ \\ \hline 
            $2$ & $A_1$ & $L(D_5+A_1;2,3,4,5,6,\theta)$ & $\{0\} \times (2)$ & $2$\\ \hline
        \end{tabular}
    \end{center}
    Since $M_1$ is standard, $\tau_1(1)=(1,0,0,0,0,0,0)$ and hence $\delta_1=\frac{1}{4}\tau_1(1)=\frac{1}{4}(1,0,0,0,0,0,0)$ by Proposition \ref{prop:identification2}. On the other hand, $\tau_1(2)$ is (either) generator of the free abelian group
    \begin{align*}\fX(M_2) &= \{\lambda \in \Lambda \mid \langle \lambda, \alpha^{\vee}\rangle =0, \ \alpha \in \Delta(\fm_2,\fh)\}\\
    &=\{(x,0,0,0,0,0,y) \in \ZZ^7 \mid 2x+y=0\}.\end{align*}
    So we may take $\tau_1(2) = (1,0,0,0,0,0,-2)$ and $\delta_2=\frac{1}{4}\tau_2 = \frac{1}{4}(1,0,0,0,0,0,-2)$. Now by Proposition \ref{prop:deltashift} 
    \begin{align*}
        \gamma_0(\widehat{\OO}) &= \rho(\fl) + \delta_1+ \delta_2\\
        &= (-5,1,1,1,1,1,-4) + \frac{1}{4}(1,0,0,0,0,0,0) + \frac{1}{4}(1,0,0,0,0,0,-2)\\
        &= \frac{1}{2}(-9,2,2,2,2,2,-9)
    \end{align*}
    Conjugating by $W$ we get
    $$\gamma_0(\widehat{\OO}) = \frac{1}{2}(2,0,1,1,0,1,1)$$

    \item $D_4(a_1)$. Note that $\OO$ is even and $L_{\OO} = L(A_1+A_5)$. Hence by Lemma \ref{lem:even}
$$(L,\OO_L) = (L(A_1+A_5),\{0\})$$
There is a single codimension 2 leaf $\fL_1\subset \Spec(\CC[\OO])$ and the corresponding singularity is of type $A_1$.
    \begin{center}
        \begin{tabular}{|c|c|c|c|c|} \hline
            $k$ & $\Sigma_k$ & $M_k$ & $\OO_{M_k}$ & $c_k$ \\ \hline
            $1$ & $A_1$ & $L(A_1+A_5)$ & $\{0\}$ & $1$ \\ \hline 
        \end{tabular}
    \end{center}
Note that $\tau_1= (0,0,1,0,0,0,0)$. So by Proposition \ref{prop:identification2}, we have $\delta_1=\frac{1}{2}\tau_1 = \frac{1}{2}(0,0,1,0,0,0,0)$. Now by Proposition \ref{prop:deltashift} 
$$\gamma_0(\widehat{\OO}) = \rho(\fl) + \delta_1 = \frac{1}{2}(2,2,-9,2,2,2,2) + \frac{1}{2}(0,0,1,0,0,0,0) = (1,1,-4,1,1,1,1).$$
Conjugating by $W$ we get
$$\gamma_0(\widehat{\OO}) = (0,0,0,1,0,0,1).$$

    \item $A_3+2A_1$. We have
    $$\mathcal{P}_{\mathrm{rig}}(\OO) = \{(L(E_6),3A_1)\}.$$
    
    There is a single codimension 2 leaf $\fL_1\subset \Spec(\CC[\OO])$ and the corresponding singularity is of type $A_1$. We have $\fr=2A_1$, and thus $\dim \fP^X=1$. Now by Lemma \ref{lem:Prigunique}, we have
$$(L,\OO_L) = (L(E_6),3A_1)$$
Thus, $M_1=L$ and
    \begin{center}
        \begin{tabular}{|c|c|c|c|c|} \hline
            $k$ & $\Sigma_k$ & $M_k$ & $\OO_{M_k}$ & $c_k$ \\ \hline
            $1$ & $A_1$ & $L(E_6)$ & $3A_1$ & $1$ \\ \hline 
        \end{tabular}
    \end{center}
Note that $\tau_1=(0,0,0,0,0,0,1)$. So by Proposition \ref{prop:identification2}, we have $\delta_1=\frac{1}{2}\tau_1=\frac{1}{2}(0,0,0,0,0,0,1)$. Note that $\gamma_0(\OO_L) = \frac{1}{2}\rho(\fl)$, see Table \ref{table:E6orbits}. Now by Proposition \ref{prop:deltashift}
$$\gamma_0(\widehat{\OO}) = \frac{1}{2}\rho(\fl) + \delta_1 = \frac{1}{2}(1,1,1,1,1,1,-8) + \frac{1}{2}(0,0,0,0,0,0,1) = \frac{1}{2}(1,1,1,1,1,1,-7)$$
Conjugating by $W$ we get
$$\gamma_0(\widehat{\OO}) = \frac{1}{2}(1,1,1,0,1,0,1).$$

    \item $D_4(a_1)+A_1$. By \cite[Sec 4]{deGraafElashvili}
    $$\mathcal{P}_{\mathrm{rig}}(\OO) = \{(L(A_5;1,3,4,5,6), \{0\})\}.$$
     There are two codimension orbits in $\overline{\OO}$, and the corresponding singularities are of types $A_1$ and $3A_1$. Therefore, there are two codimension 2 leaves $\fL_1$, $\fL_2\subset \Spec(\CC[\OO])$ and the corresponding singularities are both of type $A_1$. We have $\fr=2A_1$, and thus $\dim \fP^X=2$. So by Lemma \ref{lem:Prigunique}, we have
    $$(L,\OO_L) = (L(A_5;1,3,4,5,6), \{0\}).$$
    We now compute the pairs $(M_k,\OO_{M_k})$ using Lemma \ref{lem:findMk}. Let
    $$(M_1,\OO_{M_1}) = (L(E_6), A_2)$$
    By construction, $L \subset M_1$, and as observed in the calculation for $A_2 \subset E_6$, $\OO_{M_1} = \mathrm{Ind}^{M_1}_L \{0\}$. Furthermore, $\dim \fX(\fm_1) = 1 = \dim \fP_1$. Note that $\overline{\OO}_{M_1}$ contains a single orbit of codimension 2, namely $\OO_{M_1,2} = 3A_1$, and the corresponding singularity is of type $A_1$. Finally, by \cite[Sec 4]{deGraafElashvili} $\Ind^G_{M_1} \OO_{M_1,2} = A_3+2A_1 = \OO_2$. So by Lemma \ref{lem:findMk}, $(M_1,\OO_{M_1})$ is adapted to $\fL_1$.
    
    The computation of $M_2$ is slightly more involved. Consider the pair
    $$(M_2,\OO_{M_2}) = (L(A_5+A_1;1,3,4,5,6,\theta), \{0\} \times (2))$$
    where $\theta$ is the root
    $$\theta = \alpha_1 + 2\alpha_2+2\alpha_3+3\alpha_4+2\alpha_5+\alpha_6$$
    An {\fontfamily{cmtt}\selectfont atlas} computation shows that this indeed defines a Levi subgroup (of the stated type). By construction $L \subset M_1$ and clearly $\OO_{M_2} = \mathrm{Ind}^{M_2}_L \{0\}$. $\overline{\OO}_{M_2}$ contains a unique codimension 2 orbit, namely $\OO_{M_2,1}=\{0\}$, and the corresponding singularity is of type $A_1$. Finally,  $\Ind^G_{M_2} \OO_{M_2,1} = D_4(a_1) = \OO_1$. So by Lemma \ref{lem:findMk} $(M_2,\OO_{M_2})$ is adapted to $\fL_2$.
    \begin{center}
        \begin{tabular}{|c|c|c|c|c|} \hline
            $k$ & $\Sigma_k$ & $M_k$ & $\OO_{M_k}$ & $c_k$ \\ \hline
            $1$ & $A_1$ & $L(E_6)$ & $A_2$ & $1$ \\ \hline 
            $2$ & $A_1$ & $L(A_5+A_1;1,3,4,5,6,\theta)$ & $\{0\} \times (2)$ & $1$\\ \hline
        \end{tabular}
    \end{center}
 Since $M_1$ is standard, $\tau_1(1)=(0,0,0,0,0,0,1)$, and hence by Proposition \ref{prop:identification2}, we have $\delta_1=\frac{1}{2}\tau_1(1)=\frac{1}{2}(0,0,0,0,0,0,1)$. On the other hand, $\tau_1(2)$ is (either) generator of the free abelian group
    \begin{align*}
        \fX(M_2) &= \{\lambda \in \Lambda \mid \langle \lambda, \alpha^{\vee}\rangle =0, \ \forall \alpha \in \Delta(\fm_2,\fh)\}\\
        &= \{(0,x,0,0,0,0,y) \in \ZZ^7 \mid x+y=0\}
    \end{align*}
    So we may take $\tau_2=(0,1,0,0,0,0,-1)$ and hence by Proposition \ref{prop:identification2},$\delta_2=\frac{1}{2}\tau_1(2) = \frac{1}{2}(0,1,0,0,0,0,-1)$. Now by Proposition \ref{prop:deltashift}
    \begin{align*}\gamma_0(\widehat{\OO}) &= \rho(\fl) + \delta_1+\delta_2\\
    &= \frac{1}{2}(2,-9,2,2,2,2,-5) + \frac{1}{2}(0,0,0,0,0,0,1)+\frac{1}{2}(0,1,0,0,0,0,-1)\\
    &= \frac{1}{2}(2,-8,2,2,2,2,-5).\end{align*}
    Conjugating by $W$ we get
    $$\gamma_0(\widehat{\OO}) = \frac{1}{2}(0,0,1,1,0,1,1).$$

    \item $A_3+A_2+A_1$. Note that $\OO$ is even and $L_{\OO} = L(A_4+A_2)$. Hence by Lemma \ref{lem:even}
    $$(L,\OO_L) = (L(A_4+A_2), \{0\}).$$
    There is a single codimension 2 leaf $\fL_1\subset \Spec(\CC[\OO])$ and the corresponding singularity is of type $A_1$.
    \begin{center}
        \begin{tabular}{|c|c|c|c|c|} \hline
            $k$ & $\Sigma_k$ & $M_k$ & $\OO_{M_k}$ & $c_k$ \\ \hline
            $1$ & $A_1$ & $L(A_4+A_2)$ & $\{0\}$ & $1$ \\ \hline 
        \end{tabular}
    \end{center}
Note that $\tau_1=(0,0,0,0,1,0,0)$. So by Proposition \ref{prop:identification2}, we have $\delta_1=\frac{1}{2}\tau_1 = \frac{1}{2}(0,0,0,0,1,0,0)$. Now by Proposition \ref{prop:deltashift}
    \begin{align*}
        \gamma_0(\widehat{\OO}) &= \rho(\fl) + \delta_1\\
        &= (1,1,1,1,-4,1,1) + \frac{1}{2}(0,0,0,0,1,0,0)\\
        &= \frac{1}{2}(2,2,2,2,-7,2,2)
    \end{align*}
Conjugating by $W$ we get
$$\gamma_0(\widehat{\OO}) = \frac{1}{2}(1,0,0,1,0,1,1).$$

    \item $A_5+A_1$. We have
    $$\mathcal{P}_{\mathrm{rig}}(\OO) = \{(L(E_6),2A_2+A_1)\},$$
So by Lemma \ref{lem:Prigunique}
$$(L,\OO_L)=(L(E_6),2A_2+A_1),$$
There is a single codimension 2 leaf $\fL_1 \subset \Spec(\CC[\OO])$ and the corresponding singularity is of type $A_1$.
    \begin{center}
        \begin{tabular}{|c|c|c|c|c|} \hline
            $k$ & $\Sigma_k$ & $M_k$ & $\OO_{M_k}$ & $c_k$ \\ \hline
            $1$ & $A_1$ & $L(E_6)$ & $2A_2+A_1$ & $3$ \\ \hline 
        \end{tabular}
    \end{center}
Note that $\tau_1=(0,0,0,0,0,0,1)$. So by Proposition \ref{prop:identification2}, we have $\delta_1=\frac{1}{6}\tau_1 = \frac{1}{6}(0,0,0,0,0,0,1)$. Now by Proposition \ref{prop:deltashift}
$$\gamma_0(\widehat{\OO}) = \frac{1}{3}\rho(\fl) + \delta_1 = \frac{1}{3}(1,1,1,1,1,1,-8) + \frac{1}{6}(0,0,0,0,0,0,1) = \frac{1}{6}(2,2,2,2,2,2,-15)$$
Conjugating by $W$ we get
$$\gamma_0(\widehat{\OO}) = \frac{1}{6}(2,1,2,1,1,1,1).$$
    
\item $D_5(a_1) +A_1$. Note that $\OO$ is even and $L_{\OO} = L(A_3+A_2;2,3,4,6,7)$. Hence by Lemma \ref{lem:even}
$$(L,\OO_L) = (L(A_3+A_2;2,3,4,6,7), \{0\}).$$
There are two codimension 2 leaves $\fL_1,\fL_2 \subset \Spec(\CC[\OO])$, corresponding to the orbits $\OO_1=A_4+A_2$ and $\OO_2=D_5(a_1)$. The corresponding singularities are both of type $A_1$. Next we will compute the pairs $(M_k,\OO_{M_k})$ using Lemma \ref{lem:findMk}. 
First consider the pair
$$(M_1,\OO_{M_1}) = (L(D_6), (3^3,1^3)).$$
By construction, $L \subset M_1$ and $\OO_{M_1}=\mathrm{Ind}^{M_1}_L \{0\}$. Furthermore, $\dim \fX(\fm_1) = 1 =\dim \fP_1$. $\overline{\OO}_{M_1}$ contains a single codimension 2 orbit, namely $\OO_{M_1,2}= (3^2,2^2,1^2)$, and the singularity is of type $A_1$. Finally
$$\Ind^G_{M_1} \OO_{M_1,2} = \Ind^G_{M_1} (\Ind^{M_1}_{L(A_4;4,5,6,7)} \{0\}) = \Ind^G_{L(A_4;4,5,6,7)} \{0\}$$
Up to conjugation by $G$, there is a unique Levi subgroup of Lie type $A_4$. So the final induction above is the same as $\Ind^G_{L(A_4;1,2,3,4)}\{0\}$, which is $D_5(a_1)=\OO_2$. It follows from Lemma \ref{lem:findMk} that $(M_1,\OO_{M_1})$ is adapted to $\fL_1$. 

Next, consider the pair
$$(M_2,\OO_{M_2}) = (L(A_3+A_2+A_1;2,3,4,6,7,\theta), \{0\} \times \{0\} \times (2)),$$
where $\theta$ is the highest root for $G$, i.e.
$$\theta = 2\alpha_1+2\alpha_2+3\alpha_3+4\alpha_4+3\alpha_5+2\alpha_6+\alpha_7.$$
An {\fontfamily{cmtt}\selectfont atlas} computation shows that this indeed defines a Levi subgroup (of the correct Lie type). By construction, $L \subset M_2$ and clearly $\OO_{M_2} = \mathrm{Ind}^{M_2}_L \{0\}$. Furthermore, $\dim \fX(\fm_2) = 1 = \dim\fP_2$. There is a single codimension 2 orbit in $\overline{\OO}_{M_2}$, namely $\OO_{M_2,1} = \{0\}$, and the singularity is of type $A_1$. Finally $\Ind^G_{M_2} \OO_{M_2,1} = A_4+A_2=\OO_1$. So by Lemma \ref{lem:findMk}, $M_2$ is adapted to $\fL_2$. 
    \begin{center}
        \begin{tabular}{|c|c|c|c|c|} \hline
            $k$ & $\Sigma_k$ & $M_k$ & $\OO_{M_k}$ & $c_k$ \\ \hline
            $1$ & $A_1$ & $L(D_6)$ & $(3^3,1^3)$ & $2$ \\ \hline 
            $2$ & $A_1$ & $L(A_3+A_2+A_1;2,3,4,6,7,\theta)$ & $\{0\} \times \{0\} \times (2)$ & $2$\\ \hline
        \end{tabular}
    \end{center}
Since $M_1$ is standard, $\tau_1(1)=(1,0,0,0,0,0,0)$ and so by Proposition \ref{prop:identification2}, we have $\delta_1=\frac{1}{4}\tau_1(1)=\frac{1}{4}(1,0,0,0,0,0,0)$. On the other hand, $\tau_1(2)$ is (either) generator of the free abelian group
\begin{align*}
    \fX(M_2) &= \{\lambda \in \Lambda \mid \langle \lambda, \alpha^{\vee}\rangle =0, \ \forall \alpha \in \Delta(\fm_2,\fh)\}\\
    &= \{(x,0,0,0,y,0,0) \in \ZZ^7 \mid 2x+3y=0\}
\end{align*}
Hence $\tau_1(2) = (3,0,0,0,-2,0,0)$ and  $\delta_2=\frac{1}{4}\tau_1(2) = \frac{1}{4}(3,0,0,0,-2,0,0)$. Now by Proposition \ref{prop:deltashift}
\begin{align*}
    \gamma_0(\widehat{\OO}) &= \rho(\fl) + \delta_1 + \delta_2\\
    &= \frac{1}{2}(-3,2,2,2,-6,2,2) + \frac{1}{4}(1,0,0,0,0,0,0) + \frac{1}{4}(3,0,0,0,-2,0,0)\\
    &= \frac{1}{2}(-1,2,2,2,-7,2,2)
\end{align*}
Conjugating by $W$ we get
$$\gamma_0(\widehat{\OO}) = \frac{1}{2}(1,0,0,1,0,0,1).$$

    \item $E_7(a_5)$. Note that $\OO$ is even and $L_{\OO} = L(A_1+2A_2; 1,2,3,5,6)$. Hence by Lemma \ref{lem:even}
$$(L,\OO_L) = (L(A_1+2A_2;1,2,3,5,6), \{0\}).$$
There are two codimension 2 leaves $\fL_1,\fL_2 \subset X$, corresponding to the orbits $\OO_1=E_6(a_3)$ and $\OO_2=D_6(a_2)$. Both singularities are of type $A_1$. We will compute the pairs $(M_k,\OO_{M_k})$ using Lemma \ref{lem:findMk}. First define
$$(M_1, \OO_{M_1}) = (L(E_6),D_4(a_1))$$
By construction, $L \subset M_1$, and
$\OO_{M_1} = \mathrm{Ind}^{M_1}_L \{0\}$. Furthermore, $\dim \fX(\fm_1) = 1 = \dim \fP_1$. Note that $\overline{\OO}_{M_1}$ contains a single codimension 2 orbit, namely $\OO_{M_1,2} = A_3+A_1$, and the corresponding singularity is of type $A_1$. Finally, by \cite[Sec 4]{deGraafElashvili}
\begin{align*}
\Ind^G_{M_1} \OO_{M_1,2} &= \Ind^G_{M_1} (\Ind^{M_1}_{L(D_5;1,2,3,4,5)} (3,2^2,1^3))\\
&= \Ind^G_{L(D_5;1,2,3,4,5)} (3,2^2,1^3)\\
&= D_6(a_2)\\
&= \OO_2
\end{align*}
So by Lemma \ref{lem:findMk}, $(M_1,\OO_{M_1})$ is adapted to $\fL_1$. 

Next define
$$(M_2,\OO_{M_2}) = (L(A_1+A_5; 1, 2, 3, 5, 6, \theta), \{0\} \times (2^3)).$$
where $\theta$ is the positive root
$$\theta=\alpha_2+\alpha_3+2\alpha_4+\alpha_5+\alpha_6+\alpha_7$$
An {\fontfamily{cmtt}\selectfont atlas} computation shows that this indeed defines a Levi (of the stated type). By construction, $L \subset M_2$, and $\OO_{M_1} = \mathrm{Ind}^{M_1}_L \{0\}$. Furthermore, $\dim \fX(\fm_2) = 1 = \dim \fP_2$. Note that $\overline{\OO}_{M_2}$ contains a single codimension 2 orbit, namely $\OO_{M_2,1} = \{0\} \times (2^2,1^2)$, and the corresponding singularity is of type $A_1$. Finally, by \cite[Sec 4]{deGraafElashvili}
\begin{align*}
\Ind^G_{M_2} \OO_{M_2,1} &= \Ind^G_{M_2} (\Ind^{M_2}_{L(2A_1+A_3;1,2,4,5,7)} \{0\})\\
&= \Ind^G_{L(2A_1+3A_1;1,2,4,5,7)} \{0\}\\
&= E_6(a_3)\\
&= \OO_1
\end{align*}
So by Lemma \ref{lem:findMk}, $(M_2,\OO_{M_2})$ is adapted to $\fL_2$. 
    \begin{center}
        \begin{tabular}{|c|c|c|c|c|} \hline
            $k$ & $\Sigma_k$ & $M_k$ & $\OO_{M_k}$ & $c_k$ \\ \hline
            
            $1$ & $A_1$ & $L(E_6)$ & $D_4(a_1)$ & $1$ \\ \hline 
            $2$ & $A_1$ & $L(A_1+A_5; 1, 2, 3, 5, 6, \theta)$ & $\{0\} \times (2^3)$ & $1$\\ \hline
        \end{tabular}
    \end{center}
Since $M_1$ is standard, $\tau_1(1)=(0,0,0,0,0,0,1)$ and hence by Proposition \ref{prop:identification2} $\delta_1=\frac{1}{2}\tau_1(1) = \frac{1}{2}(0,0,0,0,0,0,1)$. On the other hand, $\tau_1(2)$ is (either) generator of the free abelian group 
\begin{align*}\fX(M_2) &= \{\lambda \in \Lambda \mid \langle \lambda, \alpha^{\vee}\rangle =0, \ \forall \alpha \in \Delta(\mathfrak{m},\mathfrak{h})\}\\
&=\{(0,0,0,x,0,0,y) \in \ZZ^7 \mid 2x+y=0\}
\end{align*}
So we may take $\tau_1(2) = (0,0,0,1,0,0,-2)$ and $\delta_2=\frac{1}{2}\tau_1(2) = \frac{1}{2}(0,0,0,1,0,0,-2)$. Now by Proposition \ref{prop:deltashift}
\begin{align*}
\gamma_0(\widehat{\OO}) &= \rho(\fl) + \delta_1+ \delta_2\\
&= \frac{1}{2}(2,2,2,-5,2,2,-2) + \frac{1}{2}(0,0,0,0,0,0,1) + \frac{1}{2}(0,0,0,1,0,0,-2)\\
&= \frac{1}{2}(2,2,2,-4,2,2,-3)
\end{align*}
Conjugating by $W$ we get
$$\gamma_0(\widehat{\OO}) = \frac{1}{2}(0,0,1,0,1,0,0).$$

\item $E_7(a_4)$. Note that $\OO$ is even and $L_{\OO} = L(2A_1+A_2;2,3,5,6)$. Hence by Lemma \ref{lem:even}
    $$(L,\OO_L) = (L(2A_1+A_2;2,3,5,6), \{0\})$$
    There are three codimension 2 leaves $\fL_1,\fL_2,\fL_3 \subset \Spec(\CC[\OO])$, corresponding to the orbits $\OO_1=A_6$, $\OO_2=D_5+A_1$, and $\OO_3=D_6(a_1)$. All three singularities are of type $A_1$. We will compute the pairs $(M_k,\OO_{M_k})$ using Lemma \ref{lem:findMk}.
    
    For $\fL_1$, consider the standard pair
    $$(M_1,\OO_{M_1}) = (L(D_6),(5,3^2,1))$$
    By construction, $L \subset M_1$ and $\OO_{M_1} = \mathrm{Ind}^{M_1}_L \{0\}$. Furthermore, $\dim \fX(\fm_1) = 1 = \dim \fP_1$. $\overline{\OO}_{M_1}$ contains two codimension 2 orbits, namely $\OO_{M_1,2} = (4^2,3,1)$ and $\OO_{M_1,3} = (5,3,2^2)$, and the corresponding singularities are of type $A_1$. Note that
    \begin{align*}
        \Ind^G_{M_1} \OO_{M_1,2} &= \Ind^G_{M_1} (\Ind^{M_1}_{L(2A_2;2,4,6,7)} \{0\})\\
        &= \Ind^G_{L(2A_2;2,4,6,7)} \{0\}
    \end{align*}
    Up to conjugation by $G$, there is a unique Levi subgroup of type $2A_2$. So the final induction is the same as $\Ind^G_{L(2A_2;1,3,5,6)}\{0\}$, which is $D_5+A_1 = \OO_2$ by \cite[Sec 4]{deGraafElashvili}. Similarly
    \begin{align*}
        \Ind^G_{M_1} \OO_{M_1,3} &= \Ind^G_{M_1} (\Ind^{M_1}_{L(A_3;5,6,7)} \{0\})\\
        &= \Ind^G_{L(A_3;5,6,7)} \{0\}\\
        &=\Ind^G_{L(A_3;1,3,4)}\{0\}\\
        &=D_6(a_1)\\
        &=\OO_3
    \end{align*}
    It follows from Lemma \ref{lem:findMk} that $(M_1,\OO_{M_1})$ is adapted to $\fL_1$.
    
    For $\fL_2$, consider the non-standard pair
    $$(M_2, \OO_{M_2}) = (L(D_5+A_1;2,3,5,6,\theta_1,\theta_2),(3^3,1) \times (2))$$
    where $\theta_1,\theta_2$ are the roots
    $$\theta_1 = -\alpha_2-\alpha_3-\alpha_4-\alpha_5-\alpha_6 \qquad \theta_2 = 2\alpha_1+2\alpha_2+3\alpha_3+4\alpha_4+3\alpha_5+2\alpha_6+\alpha_7$$
     An {\fontfamily{cmtt}\selectfont atlas} computation shows that this indeed defines a Levi subgroup (of the correct type). By construction, $L \subset M_2$. Computing the Cartan matrix for $M_2$, we see that the simple roots for $L$ embed as the simple roots for $M_2$ corresponding to the non-central nodes of the $D_5$ Dynkin diagram. In particular, $\mathrm{Ind}^{M_2}_L \{0\} = \OO_{(3^3,1) \times (2)} = \OO_{M_2}$. Furthermore, $\dim \fX(\fm_2) = 1 = \dim \fP_2$. $\overline{\OO}_{M_2}$ contains two codimension 2 orbits, namely $\OO_{M_2,1} = (3^3,1) \times \{0\} $ and $\OO_{M_2,3} = (3^2,2^2) \times (2)$, and the corresponding singularities are of type $A_1$. We have
    \begin{align*}
        \Ind^G_{M_2} \OO_{M_2,1} &= \Ind^G_{M_2} (\Ind^{M_1}_{L(3A_1+A_2;1,2,3,5,7)} \{0\})\\
        &= \Ind^G_{L(3A_1+A_2;1,2,3,5,7)} \{0\}\\
        &= A_6\\
        &= \OO_1
    \end{align*}
    Similarly
    \begin{align*}
        \Ind^G_{M_2} \OO_{M_2,3} &= \Ind^G_{M_2} (\Ind^{M_1}_{L(A_3;1,3,4)} \{0\})\\
        &= \Ind^G_{L(A_3;1,3,4)} \{0\}\\
        &= D_6(a_1)\\
        &= \OO_3
    \end{align*}
    So by Lemma \ref{lem:findMk} $(M_2,\OO_{M_2})$ is adapted to $\fL_2$.

    For $\fL_3$, consider the pair
    $$(M_3, \OO_{M_3}) = (L(A_1+A_2+A_3;2,3,5,6,\theta_2,\theta_3), \OO_{(2)} \times \{0\} \times \OO_{(2^2)})$$
    where $\theta_3$ is the negative root
    $$\theta_3 = -\alpha_2-\alpha_3-\alpha_4-\alpha_5-\alpha_6-\alpha_7.$$
    An {\fontfamily{cmtt}\selectfont atlas} computation shows that this indeed defines a Levi (of the stated type). By construction $L \subset M_3$. Computing the Cartan matrix for $M_3$, we see that the $A_2$ factor of $L$ embeds into the $A_2$ factor of $M_3$ and the $2A_1$ factor of $L$ embeds into $A_3$ factor of $M_3$. In particular, $\mathrm{Ind}^{M_3}_L \{0\} = \OO_{(2) \times \{0\} \times \OO_{(2^2)}} = \OO_{M_3}$. Furthermore, $\dim \fX(\fm_3) = 1 = \dim \fP_3$. $\overline{\OO}_{M_3}$ contains two codimension 2 orbits, namely $\OO_{M_3,1} = \{0\} \times \{0\} \times \OO_{(2^2)}$ and $\OO_{M_3,2} = \OO_{(2)} \times \{0\} \times \OO_{(2,1^2)}$.Both singularities are of type $A_1$. We have
    \begin{align*}
    \Ind^G_{M_3} \OO_{M_3,1} &= \Ind^G_{M_3} (\Ind^{M_3 v}_{L(3A_1+A_2;1,2,3,5,7)} \{0\})\\
    &= \Ind^G_{L(3A_1+A_2;1,2,3,5,7)} \{0\}\\
    &= A_6\\
    &= \OO_1
    \end{align*}
    Similarly
    \begin{align*}
    \Ind^G_{M_3} \OO_{M_3,2} &= \Ind^G_{M_3} (\Ind^{M_3}_{L(2A_2; 1,3,5,6)} \{0\})\\
    &= \Ind^G_{L(2A_2;1,3,5,6)} \{0\}\\
    &= D_5+A_1\\
    &= \OO_2
    \end{align*}
    So by Lemma \ref{lem:findMk}, $(M_3,\OO_{M_3})$ is adapted to $\fL_3$.
    \begin{center}
        \begin{tabular}{|c|c|c|c|c|} \hline
            $k$ & $\Sigma_k$ & $M_k$ & $\OO_{M_k}$ & $c_k$ \\ \hline
            
            $1$ & $A_1$ & $L(D_6)$ & $\OO_{(5,3^2,1)}$ & $2$ \\ \hline 
            
            $2$ & $A_1$ & $L(D_5+A_1;2,3,5,6,\theta_1,\theta_2)$ & $\OO_{(3^3,1)} \times \OO_{(2)}$ & $2$\\ \hline
            
            $3$ & $A_1$ & $L(A_1+A_2+A_3;2,3,5,6,\theta_2,\theta_3)$ & $\OO_{(2)} \times \{0\} \times \OO_{(2^2)}$ & $2$\\ \hline
        \end{tabular}
    \end{center}
Since $M_1$ is standard, $\tau_1(1) = (1,0,0,0,0,0,0)$ and so by Proposition \ref{prop:identification2}, $\delta_1=\frac{1}{4}\tau_1(1)=\frac{1}{4}(1,0,0,0,0,0,0)$. By definition, $\tau_1(2)$ is (either) generator of the free abelian group
    \begin{align*}
        \fX(M_2) &= \{\lambda \in \Lambda \mid \langle \lambda, \alpha^{\vee}\rangle = 0, \ \forall \alpha \in \Delta(\fm_2,\fh)\}\\
        &= \{(x,0,0,0,0,0,y) \in \ZZ^7 \mid 2x+y=0\}
    \end{align*}
    So we may take $\tau_1(2) = (1,0,0,0,0,0,-2)$ and $\delta_2 = \frac{1}{4}\tau_1(2) = \frac{1}{4}(1,0,0,0,0,0,-2)$. By a similar computation, we get $\tau_1(3) = (3,0,0,-2,0,0,2)$ and $\delta_3 = \frac{1}{4}(3,0,0,-2,0,0,2)$. Now by Proposition \ref{prop:deltashift}
    \begin{align*}
        \gamma_0(\widehat{\OO}) &= \rho(\fl) + \delta_1+\delta_2+\delta_3\\
        &= \frac{1}{2}(-1,2,2,-4,2,2,-2) + \frac{1}{4}(1,0,0,0,0,0,0) + \frac{1}{4}(1,0,0,0,0,0,-2) + \frac{1}{4}(3,0,0,-2,0,0,2)\\
        &= \frac{1}{4}(3, 4, 4, -10, 4, 4,-4 )
    \end{align*}
\end{itemize}
Conjugating by $W$ we get
$$\gamma_0(\widehat{\OO}) = \frac{1}{4}(1,1,1,0,0,1,1).$$

\paragraph{Type $E_8$}

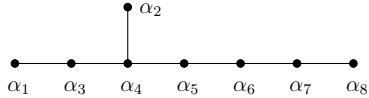
\begin{center}
$$\dynkin[label,label macro/.code={\alpha_{\drlap{#1}}},edge
length=.75cm] E8$$
\end{center}

\begin{itemize}
    \item $A_2$. Note that $\OO$ is even and $L_{\OO}=L(E_7)$. Hence by Lemma \ref{lem:even}
$$(L,\OO_L) = (L(E_7), \{0\}).$$
There is a single codimension 2 leaf $\fL_1 \subset \Spec(\CC[\OO])$ and the corresponding singularity is of type $A_1$. Thus
    \begin{center}
        \begin{tabular}{|c|c|c|c|c|} \hline
            $k$ & $\Sigma_k$ & $M_k$ & $\OO_{M_k}$ & $c_k$ \\ \hline
            
            $1$ & $A_1$ & $L(E_7)$ & $\{0\}$ & $1$ \\ \hline 
        \end{tabular}
    \end{center}
Note that $\tau_1 = (0,0,0,0,0,0,0,1)$. So by Proposition \ref{prop:identification2}, we have $\delta_1=\frac{1}{2}\tau_1=\frac{1}{2}(0,0,0,0,0,0,0,1)$. Now by Proposition \ref{prop:deltashift}
\begin{align*}
\gamma_0(\widehat{\OO}) &= \rho(\fl) + \delta_1\\
&= \frac{1}{2}(2,2,2,2,2,2,2,-27) + \frac{1}{2}(0,0,0,0,0,0,0,1) = (1,1,1,1,1,1,1,-13).
\end{align*}
Conjugating by $W$ we get
$$\gamma_0(\widehat{\OO}) = (1,0,0,1,0,1,1,1).$$

\item $2A_2$. Note that $\OO$ is even and $L_{\OO} = L(D_7)$. Hence by Lemma \ref{lem:even}
$$(L,\OO_L) = (L(D_7),\{0\})$$
There is a single codimension 2 leaf $\fL_1 \subset \Spec(\CC[\OO])$ and the corresponding singularity is of type $A_1$. 
    \begin{center}
        \begin{tabular}{|c|c|c|c|c|} \hline
            $k$ & $\Sigma_k$ & $M_k$ & $\OO_{M_k}$ & $c_k$ \\ \hline
            
            $1$ & $A_1$ & $L(D_7)$ & $\{0\}$ & $1$ \\ \hline 
        \end{tabular}
    \end{center}
Now $\tau_1=(1,0,0,0,0,0,0,0)$. So by Proposition \ref{prop:identification2}, we have $\delta_1=\frac{1}{2}\tau_1=\frac{1}{2}(1,0,0,0,0,0,0,0)$. Now by Proposition \ref{prop:deltashift}
\begin{align*}
\gamma_0(\widehat{\OO}) &= \rho(\fl) + \delta_1\\
&= \frac{1}{2}(-21,2,2,2,2,2,2,2) + \frac{1}{2}(1,0,0,0,0,0,0)\\
&=(-10,1,1,1,1,1,1,1).
\end{align*}
Conjugating by $W$ we get
$$\gamma_0(\widehat{\OO}) = (1,0,0,1,0,0,1,0).$$

    \item $D_4(a_1)$. Note that $\OO$ is even and $L_{\OO} = L(A_1+E_6)$. Hence by Lemma \ref{lem:even}
$$(L,\OO_L) = (L(A_1+E_6), \{0\}).$$
There is a single codimension 2 leaf $\fL_1 \subset \Spec(\CC[\OO])$ and the correponding singularity is of type $A_1$. 
    \begin{center}
        \begin{tabular}{|c|c|c|c|c|} \hline
            $k$ & $\Sigma_k$ & $M_k$ & $\OO_{M_k}$ & $c_k$ \\ \hline
            
            $1$ & $A_1$ & $L(A_1+E_6)$ & $\{0\}$ & $1$ \\ \hline 
        \end{tabular}
    \end{center}
Thus $\tau_1=(0,0,0,0,0,0,1,0)$, and by Proposition \ref{prop:identification2}, $\delta_1=\frac{1}{2}(0,0,0,0,0,0,1,0)$. Now by Proposition \ref{prop:deltashift}
\begin{align*}
\gamma_0(\widehat{\OO}) &= \rho(\fl) + \delta_1\\
&= \frac{1}{2}(2,2,2,2,2,2,-17,2) + \frac{1}{2}(0,0,0,0,0,0,1,0) \\
&= (1,1,1,1,1,1,-8,1).
\end{align*}
Conjugating by $W$ we get
$$\gamma_0(\widehat{\OO}) = (0,0,0,1,0,0,1,1).$$

\item $D_4(a_1)+A_2$. Note that $\OO$ is even and $L_{\OO} = L(A_7)$. Hence by Lemma \ref{lem:even}
$$(L,\OO_L) = (L(A_7), \{0\}).$$
There is a single codimension 2 leaf $\fL_1 \subset \Spec(\CC[\OO])$ and the corresponding singularity is of type $A_1$. 
    \begin{center}
        \begin{tabular}{|c|c|c|c|c|} \hline
            $k$ & $\Sigma_k$ & $M_k$ & $\OO_{M_k}$ & $c_k$ \\ \hline
            
            $1$ & $A_1$ & $L(A_7)$ & $\{0\}$ & $1$ \\ \hline 
        \end{tabular}
    \end{center}
Note that $\tau_1=(0,1,0,0,0,0,0,0)$, so by Proposition \ref{prop:identification2}, $\delta_1=\frac{1}{2}\tau_1=\frac{1}{2}(0,1,0,0,0,0,0,0)$. Now by Proposition \ref{prop:deltashift}
\begin{align*}
\gamma_0(\widehat{\OO}) &= \rho(\fl) + \delta\\
&= \frac{1}{2}(2,-15,2,2,2,2,2,2) + \frac{1}{2}(0,1,0,0,0,0,0,0)\\
&= (1,-7,1,1,1,1,1,1)
\end{align*}
Conjugating by $W$ we get
$$\gamma_0(\widehat{\OO}) = (0,0,0,1,0,0,0,1).$$

\item $D_4+A_2$. Note that $\OO$ is even and $L_{\OO} = L(A_6;1,3,4,5,6,7)$. Hence by Lemma \ref{lem:even}
$$(L,\OO_L) = (L(A_6;1,3,4,5,6,7), \{0\}).$$
There are two codimension 2 leaves $\fL_1,\fL_2 \subset \Spec(\CC[\OO])$, corresponding to the orbits $\OO_1=A_4+A_2+A_1$ and $\OO_2=D_5(a_1)+A_1$. Both singularities are of type $A_1$. We will compute $(M_k,\OO_{M_k})$ using Lemma \ref{lem:findMk}. 

First, define
$$(M_1,\OO_{M_1}) = (L(E_7), A_2+3A_1)$$
By construction $L \subset M_1$ and by the calculation for $A_2+3A_1 \subset E_6$, we have $\OO_{M_1} = \mathrm{Ind}^{M_1}_L \{0\}$. Furthermore, $\dim \fX(\fm_1)=1 =\dim \fP_1$. Note that $\overline{\OO}_{M_1}$ contains a single codimension 2 orbit, namely $\OO_{M_1,2} = A_2+2A_1$, and the corresponding singularity is of type $A_1$. Finally, $\Ind^G_{M_1} \OO_{M_1,2} = D_5(a_1) + A_1 = \OO_2$. So by Lemma \ref{lem:findMk}, $(M_1,\OO_{M_1})$ is adapted to $\fL_1$. 

Next, define
$$(M_2,\OO_{M_2}) = (L(A_1+A_6;1,3,4,5,6,7,\theta), (2) \times \{0\})$$
where $\theta$ is the highest root for $G$, i.e.
$$\theta = 2\alpha_1+3\alpha_2+4\alpha_3+6\alpha_4+5\alpha_5 + 4\alpha_6 + 3\alpha_7+2\alpha_8$$
An {\fontfamily{cmtt}\selectfont atlas} computation shows that this indeed defines a Levi (of the stated type). By construction $L \subset M_2$ and $\OO_{M_2} = \mathrm{Ind}^{M_2}_L \{0\}$. Of course, $\dim \fX(\fm_2) = 1 = \dim \fP_2$. Note that $\overline{\OO}_{M_2}$ contains a single codimension 2 orbit, namely $\OO_{M_2,1} = \{0\} \times \{0\}$, and the corresponding singularity is of type $A_1$. Finally
$$\Ind^G_{M_2} \OO_{M_2,1} = A_4+A_2+A_1 = \OO_1.$$
So by Lemma \ref{lem:findMk},$(M_2,\OO_{M_2})$ is adapted to $\fL_2$.
    \begin{center}
        \begin{tabular}{|c|c|c|c|c|} \hline
            $k$ & $\Sigma_k$ & $M_k$ & $\OO_{M_k}$ & $c_k$ \\ \hline
            
            $1$ & $A_1$ & $L(E_7)$ & $A_2+3A_1$ & $2$ \\ \hline 
            
            $2$ & $A_1$ & $L(A_1+A_6;1,3,4,5,6,7,\theta)$ & $(2) \times \{0\}$ & $2$ \\ \hline 
        \end{tabular}
    \end{center}
Since $M_1$ is standard, $\tau_1(1)=(0,0,0,0,0,0,0,1)$, so by Proposition \ref{prop:identification2} $\delta_1=\frac{1}{4}\tau_1(1)=\frac{1}{4}(0,0,0,0,0,0,0,1)$. On the other hand, $\tau_1(2)$ is (either) generator of the free abelian group
\begin{align*}
\fX(M_2) &= \{\lambda \in \Lambda \mid \langle \lambda, \alpha^{\vee}\rangle =0,  \ \forall \alpha \in \Delta(\fm_2,\fh)\}\\
&=\{(0,x,0,0,0,0,0,y) \in \ZZ^8 \mid 3x+2y=0\}.
\end{align*}
So we may take $\tau_1(2) = (0,-2,0,0,0,0,0,3)$ and $\delta_2=\frac{1}{4}\tau_1(2)=\frac{1}{4}(0,-2,0,0,0,0,0,3)$. Now by Proposition \ref{prop:deltashift}
\begin{align*}
    \gamma_0(\widehat{\OO}) &= \rho(\fl) + \delta_1+\delta_2 \\
    &= (1,-6,1,1,1,1,1,-3) + \frac{1}{4}(0,0,0,0,0,0,0,1) + \frac{1}{4}(0,-2,0,0,0,0,0,3) \\ 
    &= \frac{1}{2}(2,-13,2,2,2,2,2,-4)
\end{align*}
Conjugating by $W$ we get
$$\gamma_0(\widehat{\OO}) = \frac{1}{2}(1,0,0,1,0,0,1,1)$$

    \item $D_6(a_2)$. We have
    $$\mathcal{P}_{\mathrm{rig}}(\OO) = \{(L(D_7),(3,2^4,1))\}.$$
 There are two codimension orbits in the closure of $\OO$, and the corresponding singularities are of types $m$ and $A_1$. Therefore, there is a single codimension 2 leaf $\fL_1\subset \Spec(\CC[\OO])$ and the corresponding singularity is of type $A_1$. We have $\fr=2A_1$, and thus $\dim \fP^X=1$.  So by Lemma \ref{lem:Prigunique}
$$(L,\OO_L) = (L(D_7), (3,2^4,1)).$$
Thus, $M_1=L$ and
    \begin{center}
        \begin{tabular}{|c|c|c|c|c|} \hline
            $k$ & $\Sigma_k$ & $M_k$ & $\OO_{M_k}$ & $c_k$ \\ \hline
            
            $1$ & $A_1$ & $L(D_7)$ & $(3,2^4,1)$ & $2$ \\ \hline 
        \end{tabular}
    \end{center}
Note that $\tau_1=(1,0,0,0,0,0,0,0)$, so by Proposition \ref{prop:identification2} $\delta_1=\frac{1}{4}\tau_1=\frac{1}{4}(1,0,0,0,0,0,0,0)$. By \cite[Prop 8.2.3]{LMBM} $\gamma_0(\OO_L) = \frac{1}{2}\rho(\fl)$. Thus by Proposition \ref{prop:deltashift}
\begin{align*}
    \gamma_0(\widehat{\OO}) &= \frac{1}{2}\rho(\fl) + \delta\\
    &= \frac{1}{4}(-21,2,2,2,2,2,2,2) + \frac{1}{4}(1,0,0,0,0,0,0,0)\\
    &= \frac{1}{2}(-10,1,1,1,1,1,1,1)
\end{align*}
Conjugating by $W$ we get
$$\gamma_0(\widehat{\OO}) = \frac{1}{2}(1,0,0,1,0,0,1,0).$$
    
    \item $E_6(a_3)+A_1$. We have
    $$\mathcal{P}_{\mathrm{rig}}(\OO) = \{(L(E_7),A_1+2A_2)\}.$$
     There are two codimension orbits in the closure of $\OO$, and the corresponding singularities are $m$ and $A_1$. Hence, there is a single codimension 2 leaf $\fL_1\subset \Spec(\CC[\OO])$ and the corresponding singularity is of type $A_1$. We have $\fr=A_1$, and thus $\dim \fP^X=1$. So by Lemma \ref{lem:Prigunique}
$$(L,\OO_L) = (L(E_7), A_1+2A_2).$$
Thus, $M_1=L$ and we have
    \begin{center}
        \begin{tabular}{|c|c|c|c|c|} \hline
            $k$ & $\Sigma_k$ & $M_k$ & $\OO_{M_k}$ & $c_k$ \\ \hline
            
            $1$ & $A_1$ & $L(E_7)$ & $A_1+2A_2$ & $1$ \\ \hline 
        \end{tabular}
    \end{center}
Note that $\tau_1 = (0,0,0,0,0,0,0,1)$. So by Proposition \ref{prop:identification2}, $\delta_1=\frac{1}{2}\tau_1=\frac{1}{2}(0,0,0,0,0,0,0,1)$. By \ref{table:E7orbits} $\gamma_0(\OO_L) = \frac{1}{3}\rho(\fl)$. So by Proposition \ref{prop:deltashift}
\begin{align*}
    \gamma_0(\widehat{\OO}) &= \frac{1}{3}\rho(\fl) + \delta\\
    &= \frac{1}{6}(2,2,2,2,2,2,2,-27) + \frac{1}{2}(0,0,0,0,0,0,0,1)\\
    &= \frac{1}{3}(1,1,1,1,1,1,1,-12)
\end{align*}
Conjugating by $W$ we get
$$\gamma_0(\widehat{\OO}) = \frac{1}{3}(0,1,1,0,1,0,1,1).$$

    \item $E_7(a_5)$. We have
    $$\mathcal{P}_{\mathrm{rig}}(\OO) = \{(L(E_7),(A_1+A_3)'), (L(E_6+A_1), 3A_1 \times \{0\})\}.$$
    There are two codimension orbits in the closure of $\OO$, and the corresponding singularities are of types $m$ and $2A_1$. Hence, there is a single codimension 2 leaf $\fL_1\subset \Spec(\CC[\OO])$ and the corresponding singularity is of type $A_1$. We have $\fr=A_1$, and thus $\dim \fP^X=1$. At this point, it is not clear which pair in $\mathcal{P}_{\mathrm{rig}}(\OO)$ induces $\OO$ birationally. However, it turns out not to matter. In both cases, we get the same answer for $\gamma_0(\widehat{\OO})$.

Assume first that
$$(L,\OO_L) = (L(E_7),(A_1+A_3)').$$
There is a single codimension 2 leaf $\fL_1 \subset \Spec(\CC[\OO])$ and the corresponding singularity is of type $A_1$.
    \begin{center}
        \begin{tabular}{|c|c|c|c|c|} \hline
            $k$ & $\Sigma_k$ & $M_k$ & $\OO_{M_k}$ & $c_k$ \\ \hline
            
            $1$ & $A_1$ & $L(E_7)$ & $(A_1+A_3)'$ & $1$ \\ \hline 
        \end{tabular}
    \end{center}
Thus $\tau_1=(0,0,0,0,0,0,0,1)$ and by Proposition \ref{prop:identification2} $\delta_1=\frac{1}{2}\tau_1=\frac{1}{2}(0,0,0,0,0,0,0,1)$. $\gamma_0(\OO_L)$ was computed in the previous subsection in terms of fundamental weights for $L$. It is not difficult to rewrite this in terms of fundamental weights for $G$
$$\gamma_0(\OO_L) = \frac{1}{2}(1,1,0,1,0,1,1,-9)$$
Now by Proposition \ref{prop:deltashift}
\begin{align*}
    \gamma_0(\widehat{\OO}) &= \gamma_0(\OO_L) + \delta\\
    &= \frac{1}{2}(1,1,0,1,0,1,1,-9) + \frac{1}{2}(0,0,0,0,0,0,0,1)\\
    &=\frac{1}{2}(1,1,0,1,0,1,1,-8)
\end{align*}
Conjugating by $W$ we get
$$\gamma_0(\widehat{\OO}) = \frac{1}{2}(0,0,1,0,1,0,0,1).$$
Next, suppose
$$(L,\OO_L) = (L(E_6+A_1),3A_1 \times \{0\}).$$
Then
    \begin{center}
        \begin{tabular}{|c|c|c|c|c|} \hline
            $k$ & $\Sigma_k$ & $M_k$ & $\OO_{M_k}$ & $c_k$ \\ \hline
            
            $1$ & $A_1$ & $L(E_6+A_1)$ & $3A_1 \times \{0\}$ & $1$ \\ \hline 
        \end{tabular}
    \end{center}
Thus $\tau_1=(0,0,0,0,0,0,1,0)$ and by Proposition \ref{prop:identification2} $\delta_1=\frac{1}{2}\tau_1=\frac{1}{2}(0,0,0,0,0,0,1,0)$. $\gamma_0(\OO_L)$ was computed in Section \ref{subsec:exceptionalorbitinflchar} in terms of fundamental weights for $L$. Rewriting again in terms of fundamental weights for $G$
$$\gamma_0(\OO_L) = \frac{1}{2}(1,1,1,1,1,1,-9,2).$$
Now by Proposition \ref{prop:deltashift}
\begin{align*}
\gamma_0(\widehat{\OO}) &= \gamma_0(\OO_L) + \delta\\
&=\frac{1}{2}(1,1,1,1,1,1,-9,2) + \frac{1}{2}(0,0,0,0,0,0,1,0)\\
&= \frac{1}{2}(1,1,1,1,1,1,-8,2)
\end{align*}
Conjugating by $W$ we get once again
$$\gamma_0(\widehat{\OO}) = \frac{1}{2}(0,0,1,0,1,0,0,1).$$

\item $E_8(a_7)$. Note that $\OO$ is even and $L_{\OO} = L(A_4+A_3)$. Hence by Lemma \ref{lem:even}
$$(L,\OO_L) = (L(A_4+A_3), \{0\}).$$
There is a single codimension 2 leaf $\fL_1 \subset \Spec(\CC[\OO])$ and the corresponding singularity is of type $A_1$. 
    \begin{center}
        \begin{tabular}{|c|c|c|c|c|} \hline
            $k$ & $\Sigma_k$ & $M_k$ & $\OO_{M_k}$ & $c_k$ \\ \hline
            
            $1$ & $A_1$ & $L(A_4+A_3)$ & $\{0\}$ & $1$ \\ \hline 
        \end{tabular}
    \end{center}
Thus $\tau_1=(0,0,0,0,1,0,0,0)$, so by Proposition \ref{prop:identification2} $\delta_1=\frac{1}{2}\tau_1=\frac{1}{2}(0,0,0,0,1,0,0,0)$. Now by Proposition \ref{prop:deltashift}
\begin{align*}
    \gamma_0(\widehat{\OO}) &= \rho(\fl) + \delta\\
    &= \frac{1}{2}(2,2,2,2,-9,2,2,2) +  \frac{1}{2}(0,0,0,0,1,0,0,0)\\
    &= (1,1,1,1,-4,1,1,1)
\end{align*}
Conjugating by $W$ we get
$$\gamma_0(\widehat{\OO}) = (0,0,0,0,1,0,0,0).$$

    \item $E_8(b_6)$. $\widehat{\OO}$ is birationally rigid by Proposition \ref{prop:A1}. The infinitesimal character $\gamma_0(\widehat{\OO})$ was computed in \cite[Ex 8.5.2]{LMBM}
    $$\gamma_0(\widehat{\OO}) = \frac{1}{3}(1,1,0,0,0,1,0,1).$$
\end{itemize}

\begin{table}[H]
    \begin{tabular}{|c|c|} \hline
        $\OO$ & $\gamma_0(\widehat{\OO})$   \\ \hline
        $G_2(a_1)$ & \cellcolor{blue!20}$(1,0)$  \\ \hline
    \end{tabular}
    \caption{Unipotent infinitesimal characters attached to birationally rigid covers: type $G_2$. Special unipotent characters are highlighted in blue.}
    \label{table:coversG2}
\end{table}

\begin{table}[H]
    \begin{tabular}{|c|c|} \hline
        $\OO$ & $\gamma_0(\widehat{\OO})$ \\ \hline
        $A_2$ &  $\frac{1}{2}(1,1,0,2)$\\ \hline
        $B_2$ & $\frac{1}{2}(0,1,0,2)$  \\ \hline
        $C_3(a_1)$ & $\frac{1}{2}(1,0,1,1)$\\ \hline
        $F_4(a_3)$ & \cellcolor{blue!20}$(0,0,1,0)$ \\ \hline

    \end{tabular}
    \caption{Unipotent infinitesimal characters attached to birationally rigid covers: type $F_4$. Special unipotent characters are highlighted in blue.}
    \label{table:coversF4}
\end{table}

\begin{table}[H]
    \begin{tabular}{|c|c|} \hline
        $\OO$ & $\gamma_0(\widehat{\OO})$  \\ \hline
        $A_2$ & \cellcolor{blue!20}$(1,0,0,1,0,1)$ \\ \hline
        $2A_2$ & $\frac{1}{3}(1,3,1,1,1,1)$ \\ \hline
        $D_4(a_1)$ & \cellcolor{blue!20}$(0,0,0,1,0,0)$ \\ \hline
        $A_5$ & $\frac{1}{6}(1,3,1,1,1,1)$  \\ \hline
        $E_6(a_3)$ & $\frac{1}{3}(0,1,1,0,1,0)$   \\ \hline
    \end{tabular}
    \caption{Unipotent infinitesimal characters attached to birationally rigid covers: type $E_6$. Special unipotent characters are highlighted in blue.}
    \label{table:coversE6}
\end{table}

\begin{table}[H]
    \begin{tabular}{|c|c|} \hline
        $\OO$ & $\gamma_0(\widehat{\OO})$  \\ \hline
        $(3A_1)''$ & $\frac{1}{2}(2,1,2,1,1,1,1)$ \\ \hline
        $A_2$ & \cellcolor{blue!20}$(1,0,0,1,0,1,1)$  \\ \hline
        $A_2+3A_1$ & $\frac{1}{2}(1,1,1,0,1,1,1)$  \\ \hline
        $(A_3+A_1)''$ & $\frac{1}{2}(2,0,1,1,0,1,1)$  \\ \hline
        $D_4(a_1)$ & \cellcolor{blue!20}$(0,0,0,1,0,0,1)$ \\ \hline
        $A_3+2A_1$ & $\frac{1}{2}(1,1,1,0,1,0,1)$ \\ \hline
        $D_4(a_1)+A_1$ & $\frac{1}{2}(0,0,1,1,0,1,1)$ \\ \hline
        $A_3+A_2+A_1$ & $\frac{1}{2}(1,0,0,1,0,1,1)$ \\ \hline
        $A_5+A_1$ & $\frac{1}{6}(2,1,2,1,1,1,1)$ \\ \hline
        $D_5(a_1)+A_1$ & \cellcolor{blue!20}$\frac{1}{2}(1,0,0,1,0,0,1)$ \\ \hline
        $E_7(a_5)$ & $\frac{1}{2}(0,0,1,0,1,0,0)$ \\ \hline
        $E_7(a_4)$ & $\frac{1}{4}(1,1,1,0,0,1,1)$ \\ \hline

    \end{tabular}
    \caption{Unipotent infinitesimal characters attached to birationally rigid covers: type $E_7$. Special unipotent characters are highlighted in blue.}
    \label{table:coversE7}
\end{table}

\begin{table}[H]
    \begin{tabular}{|c|c|} \hline
        $\OO$ & $\gamma_0(\widehat{\OO})$  \\ \hline
        $A_2$ & \cellcolor{blue!20}$(1,0,0,1,0,1,1,1)$ \\ \hline
        $2A_2$ & \cellcolor{blue!20}$(1,0,0,1,0,0,1,0)$ \\ \hline
        $D_4(a_1)$ & \cellcolor{blue!20}$(0,0,0,1,0,0,1,1)$  \\ \hline
        $D_4(a_1)+A_2$ & \cellcolor{blue!20}$(0,0,0,1,0,0,0,1)$  \\ \hline
        $D_4+A_2$ & $\frac{1}{2}(1,0,0,1,0,0,1,1)$ \\ \hline
        $D_6(a_2)$ & $\frac{1}{2}(1,0,0,1,0,0,1,0)$ \\ \hline
        $E_6(a_3)+A_1$ & $\frac{1}{3}(0,1,1,0,1,0,1,1)$ \\ \hline
        $E_7(a_5)$ & $\frac{1}{2}(0,0,1,0,1,0,0,1)$  \\ \hline
        $E_8(a_7)$ & \cellcolor{blue!20}$(0,0,0,0,1,0,0,0)$\\ \hline
        $E_8(b_6)$ & $\frac{1}{3}(1,1,0,0,0,1,0,1)$ \\ \hline
    \end{tabular}
    \caption{Unipotent infinitesimal characters attached to birationally rigid covers: type $E_8$. Special unipotent characters are highlighted in blue.}
    \label{table:coversE8}
\end{table}

\section{Maximality of unipotent ideals}

\begin{theorem}\label{thm:maximality}
Let $G$ be a complex reductive algebraic group and let $\widetilde{\OO}$ be a $G$-equivariant nilpotent cover. Then the unipotent ideal $I_0(\widetilde{\OO}) \subset U(\fg)$ is maximal.
\end{theorem}

\begin{proof}
Replacing $G$ with a covering group if necessary, we can assume that $G$ is simply connected. Hence, $G$ is of the form
$$G \simeq T \times G_1 \times ... \times G_n,$$
where $T$ is a torus and $G_1,...,G_n$ are simply connected simple groups. It follow that $\widetilde{\OO}$ is of the form
$$\widetilde{\OO} \simeq \widetilde{\OO}_1 \times ... \times \widetilde{\OO}_n,$$
where $\widetilde{\OO}_i$ are $G_i$-equivariant covers of nilpotent co-adjoint $G_i$-orbits. Now
$$\cA_0^{\widetilde{X}} \simeq \bigotimes_{i=1}^n \cA_0^{\widetilde{X}_1} \qquad \Phi_0^{\widetilde{X}} = \bigotimes_{i=1}^n \Phi_0^{\widetilde{X}_i},$$
where we identify $U(\fg)$ with $U(\mathfrak{t}) \otimes \bigotimes U(\fg_i)$ and map trivially from $\mathfrak{t}$. Consequently
$$I_0(\widetilde{\OO}) = \sum_{i=1}^n U(\mathfrak{t}) \otimes U(\fg_1) \otimes ... \otimes U(\fg_{i-1}) \otimes I_0(\widetilde{\OO}_i) \otimes U(\fg_{i+1}) \otimes ... \otimes U(\fg_n).$$
This ideal is maximal if and only if each $I_0(\widetilde{\OO}_i)$ is maximal. In this way, we can reduce to the case in which $G$ is simple and simply connected, i.e. $G$ is isomorphic to $\mathrm{SL}(n)$, $\mathrm{Spin}(n)$, $\mathrm{Sp}(2n)$ (for some $n$) or a simply connected simple group of exceptional type. For $G=\mathrm{SL}(n)$ or $\mathrm{Sp}(2n)$, the maximality assertion follows from \cite[Thm 8.5.1]{LMBM}. For $G=\mathrm{Spin}(n)$, we refer to Proposition \ref{prop:maximalitySpin}. For the exceptional cases, we refer to Proposition \ref{prop:maximalityexceptional}.
\end{proof}

In \cite[Prop 6.3.3]{LMBM}, we show that the maximality of $I_0(\widetilde{\OO})$ is equivalent to the simplicity of $\cA_0^{\widetilde{X}}$. This, combined with Theorem \ref{thm:maximality}, proves the following result.

\begin{cor}
Let $\widetilde{\OO}$ be a (finite, connected) cover of a complex nilpotent orbit in a complex reductive Lie algebra. Then the canonical quantization of $\CC[\widetilde{\OO}]$ is a simple algebra.
\end{cor}

\section{Real groups}\label{sec:realgroups}

In this section, we take $G$ to be a real reductive Lie group. For concreteness, we will use Knapp's definition of `real reductive Lie group', see \cite[Chp VII, Sec 2]{KnappBeyond}. In particular, we assume that $G$ has a Cartan decomposition, finitely many connected components, and that the identity component of $G$ is a finite cover of an algebraic group. Choose a maximal compact subgroup $K \subset G$ and let $\fg$ denote the complexified Lie algebra of $G$.  Under our assumptions on $G$, we have the usual bijection between irreducible admissible representations of $G$ (up to infinitesimal equivalence) and irreducible admissible $(\fg,K)$-modules (up to isomorphism). In this setting, we propose the following deifnition.

\begin{definition}\label{def:unipotentreal}
Let $\OO \subset \fg^*$ be a nilpotent orbit which satisfies the condition
\begin{equation}\label{eq:codimcondition}
\Spec(\CC[\OO]) \text{ contains no codimension 2 leaves}.\end{equation}
Then a \emph{unipotent representation} of $G$ attached to $\OO$ is an irreducible $(\fg,K)$-module $X$ such that $\Ann(X) = I_0(\OO)$. Write $\mathrm{Unip}_{\OO}(G)$ for the set of equivalence classes of such representations.
\end{definition}

\begin{rmk}
Note: every birationally rigid orbit (in particular, every rigid orbit) satisfies condition (\ref{eq:codimcondition}). This is immediate from Proposition \ref{prop:criterionbirigidcover}.
\end{rmk}

We note that Definition \ref{def:unipotentreal} is a generalization of the notion of a \emph{special unipotent} representation, due to Adams, Barbasch, and Vogan (\cite{AdamsBarbaschVogan}). Suppose $G$ is algebraic, and let $\fg^{\vee}$ denote the Langlands dual Lie algebra. There is an order-reversing map, called Barbasch-Vogan duality, from nilpotent orbits in $(\fg^{\vee})^*$ to nilpotent orbits in $\fg^*$, see \cite[Appendix]{BarbaschVogan1985}. Denote this map by $\mathsf{D}_{\fg}$. If $\OO^{\vee} \subset (\fg^{\vee})^*$ is a nilpotent orbit corresponding to an $\mathfrak{sl}(2)$-triple $(e^{\vee},f^{\vee},h^{\vee})$, then the element $\frac{1}{2}h^{\vee} \in \fh^{\vee} \simeq \fh^*$ is well-defined modulo the natural action of $W$, and therefore determines an infinitesimal character for $U(\fg)$. Write $I_{\mathrm{max}}(\frac{1}{2}h^{\vee}) \subset U(\fg)$ for the (unique) maximal primitive ideal with this infinitesimal character. By \cite[Prop A2]{BarbaschVogan1985}, the associated variety of $I_{\mathrm{max}}(\frac{1}{2}h^{\vee})$ is $\overline{\mathsf{D}_{\fg}(\OO^{\vee})}$.

According to \cite[Sec 27]{AdamsBarbaschVogan}, the \emph{weak Arthur packet} associated to $\OO^{\vee}$ is the finite set of irreducible $(\fg,K)$-modules
$$\unip_{\OO^{\vee}}^{\mathrm{ABV}}(G) := \{X \in \mathrm{Irr}(\fg,K) \mid \Ann(X) = I_{\mathrm{max}}(\frac{1}{2}h^{\vee})\}.$$
The elements of $\unip_{\OO^{\vee}}^{\mathrm{ABV}}(G)$ are called \emph{special unipotent representations}.

The following proposition explains the relationship between $\unip_{\OO^{\vee}}^{\mathrm{ABV}}(G)$ and $\unip_{\OO}(G)$.

\begin{prop}\label{prop:specialrigid}
Suppose $\OO \subset \fg^*$ is a special nilpotent orbit which satisfies (\ref{eq:codimcondition}). Then the following are true:
\begin{itemize}
    \item[(i)] There is a unique nilpotent orbit $\OO^{\vee} \subset (\fg^{\vee})^*$ such that $\mathsf{D}_{\fg}(\OO^{\vee}) = \OO$.
    \item[(ii)] $\OO^{\vee} = \mathsf{D}_{\fg^{\vee}}(\OO)$.
    \item[(iii)] $I_0(\OO) = I_{\mathrm{max}}(\frac{1}{2}h^{\vee})$. 
\end{itemize}
In particular
$$\mathrm{Unip}_{\OO}(G) = \mathrm{Unip}_{\OO^{\vee}}^{\mathrm{ABV}}(G)$$
\end{prop}

\begin{proof}
Suppose $\mathsf{D}_{\fg}\OO^{\vee} = \OO$. By \cite[Prop 9.2.1]{LMBM}, there is a cover $\widetilde{\OO}$ of $\OO$ such that
$$I_{\mathrm{max}}(\frac{1}{2}h^{\vee}) = I_0(\widetilde{\OO})$$
Since $\OO$ satisfies (\ref{eq:codimcondition}), $[\OO] = [\widetilde{\OO}]$, see the discussion preceding Theorem \ref{thm:classificationideals}. Hence by Theorem \ref{thm:classificationideals} $I_0(\OO) = I_0(\widetilde{\OO})$. This proves (iii). By the Dynkin classification of nilpotent orbits, the map $\OO^{\vee} \mapsto \frac{1}{2}h^{\vee}$ is injective. So $\OO^{\vee}$ is uniquely determined by $I_0(\OO)$. This proves (i). For (ii), we note that $\mathsf{D}_{\fg} \circ \mathsf{D}_{\fg^{\vee}}$ restricts to the identity map on special nilpotent orbits, see (d) of \cite[Prop A2]{BarbaschVogan1985}. This completes the proof.

\end{proof}

\subsection{Unitarity}\label{subsec:unitarity}

\begin{theorem}\label{thm:unitarity}
Suppose $G$ is a real form of a simple exceptional group and $\OO \subset \fg^*$ is a rigid nilpotent orbit. Then all unipotent representations attached to $\OO$ (cf. Definition \ref{def:unipotentreal}) are unitary.
\end{theorem}

\begin{proof}
By Proposition \ref{prop:specialrigid}, if $\OO$ is special, then all unipotent representations attached to $\OO$ are special unipotent. The unitarity of such representations was recently established in \cite{AdamsMillerVogan}. So we restrict our attention to non-special rigid orbits. For any orbit $\OO$, $\unip_{\OO}(G)$ is the set of irreducible $(\fg,K)$-modules $X$ such that
\begin{enumerate}
    \item The infinitesimal character of $X$ is $\gamma_0(\OO)$.
    \item The Gelfand-Kirillov dimension of $X$ is $\frac{1}{2}\dim(\OO)$.
\end{enumerate}
Recall that $\gamma_0(\OO)$ was computed, for all rigid $\OO$, in Section \ref{subsec:exceptionalorbitinflchar}. The {\fontfamily{cmtt}\selectfont atlas} command `all\_parameters\_gamma' lists the Langlands parameters of all irreducible $(\fg,K)$-modules of a given infinitesimal character, and the command `GK\_dim' computes the Gelfand-Kirillov dimension of the representation corresponding to a given parameter. Applying these commands in conjunction for all real forms of simple exceptional groups, we determine that there are a total of 12 unipotent representations attached to rigid nilpotent orbits.  They are listed in Tables \ref{table:nonspecialG2}-\ref{table:nonspecialE8}. Once the representations have been located, the command `is\_unitary' can be used to check unitarity. 
\end{proof}

\begin{table}[H]
    \begin{tabular}{|c|c|c|c|} \hline
       $\OO$ & $\gamma_0(\OO)$ & $G$ & $\#\mathrm{Unip}_{\OO}(G)$ \\ \hline
       
       $A_1$ & $\frac{1}{3}(3,1)$ & $G_2$ (split) & 1\\ \hline
       
       $\widetilde{A}_1$ & $\frac{1}{2}(1,1)$ & $G_2$ (split) & 0\\ \hline
    \end{tabular}
    \caption{Non-special unipotents: type $G_2$}
    \label{table:nonspecialG2}
\end{table}

\begin{table}[H]
    \begin{tabular}{|c|c|c|c|} \hline
       $\OO$ & $\gamma_0(\OO)$ & $G$ & $\#\mathrm{Unip}_{\OO}(G)$ \\ \hline
       
       $A_1$ & $\frac{1}{2}(1,1,2,2)$ & $F_4$ ($B_4$) & 0\\ \hline
       
       $A_1$ & $\frac{1}{2}(1,1,2,2)$  & $F_4$ (split) & 0\\ \hline
       
       $A_2+\widetilde{A}_1$ & $\frac{1}{4}(1,1,2,2)$ & $F_4$ ($B_4$) & 0 \\ \hline
       
       $A_2+\widetilde{A}_1$ & $\frac{1}{4}(1,1,2,2)$ & $F_4$ (split) & 0 \\ \hline
       
       $\widetilde{A}_2+A_1$ & $\frac{1}{3}(1,1,1,1)$ & $F_4$ ($B_4$) & 0\\ \hline
       
       $\widetilde{A}_2+A_1$ & $\frac{1}{3}(1,1,1,1)$ & $F_4$ (split) & 1\\ \hline
    \end{tabular}
    \caption{Non-special unipotents: type $F_4$}
    \label{table:nonspecialF4}
\end{table}

\begin{table}[H]
    \begin{tabular}{|c|c|c|c|} \hline
       $\OO$ & $\gamma_0(\OO)$ & $G$ & $\#\mathrm{Unip}_{\OO}(G)$ \\ \hline
       
       $3A_1$ & $\frac{1}{2}(1,1,1,1,1,1)$ & $E_6^{\mathrm{sc}}$ (Hermitian) & 0 \\ \hline
       $3A_1$ & $\frac{1}{2}(1,1,1,1,1,1)$ & $E_6^{\mathrm{sc}}$ (quasi-split) & 0\\ \hline
       $3A_1$ & $\frac{1}{2}(1,1,1,1,1,1)$ & $E_6^{\mathrm{sc}}$ ($F_4$) & 0\\ \hline
       $3A_1$ & $\frac{1}{2}(1,1,1,1,1,1)$ & $E_6^{\mathrm{sc}}$ (split) & 0\\ \hline
       $3A_1$ & $\frac{1}{2}(1,1,1,1,1,1)$ & $E_6^{\mathrm{ad}}$ (Hermitian) & 0\\ \hline
       $3A_1$ & $\frac{1}{2}(1,1,1,1,1,1)$ & $E_6^{\mathrm{ad}}$ (quasi-split) & 0 \\ \hline
       $3A_1$ & $\frac{1}{2}(1,1,1,1,1,1)$ & $E_6^{\mathrm{ad}}$ ($F_4$) & 0\\ \hline
       $3A_1$ & $\frac{1}{2}(1,1,1,1,1,1)$ & $E_6^{\mathrm{ad}}$ (split) & 0\\ \hline
       
       $2A_2+A_1$ & $\frac{1}{3}(1,1,1,1,1,1)$ & $E_6^{\mathrm{sc}}$ (Hermitian) & 0\\ \hline
       $2A_2+A_1$ & $\frac{1}{3}(1,1,1,1,1,1)$ & $E_6^{\mathrm{sc}}$ (quasi-split) & 3 \\ \hline
       $2A_2+A_1$ & $\frac{1}{3}(1,1,1,1,1,1)$ & $E_6^{\mathrm{sc}}$ ($F_4$) & 0\\ \hline
       $2A_2+A_1$ & $\frac{1}{3}(1,1,1,1,1,1)$ & $E_6^{\mathrm{sc}}$ (split) & 1\\ \hline
       $2A_2+A_1$ & $\frac{1}{3}(1,1,1,1,1,1)$ & $E_6^{\mathrm{ad}}$ (Hermitian) & 0\\ \hline
       $2A_2+A_1$ & $\frac{1}{3}(1,1,1,1,1,1)$ & $E_6^{\mathrm{ad}}$ (quasi-split) & 0\\ \hline
       $2A_2+A_1$ & $\frac{1}{3}(1,1,1,1,1,1)$ & $E_6^{\mathrm{ad}}$ ($F_4$) & 0 \\ \hline
       $2A_2+A_1$ & $\frac{1}{3}(1,1,1,1,1,1)$ & $E_6^{\mathrm{ad}}$ (split) & 0\\ \hline
    \end{tabular}
    \caption{Non-special unipotents: type $E_6$}
    \label{table:nonspecialE6}
\end{table}

\begin{table}[H]
    \begin{tabular}{|c|c|c|c|} \hline
       $\OO$ & $\gamma_0(\OO)$ & $G$ & $\#\mathrm{Unip}_{\OO}(G)$ \\ \hline
       
       $(3A_1)'$ & $\frac{1}{2}(1,1,1,1,1,1,2)$ & $E_7^{\mathrm{sc}}$ (Hermitian) & 0 \\ \hline
       $(3A_1)'$ & $\frac{1}{2}(1,1,1,1,1,1,2)$ & $E_7^{\mathrm{sc}}$ (quaternionic) & 0\\ \hline
       $(3A_1)'$ & $\frac{1}{2}(1,1,1,1,1,1,2)$ & $E_7^{\mathrm{sc}}$ (split) & 0\\ \hline
       $(3A_1)'$ & $\frac{1}{2}(1,1,1,1,1,1,2)$ & $E_7^{\mathrm{ad}}$ (Hermitian) & 0 \\ \hline
       $(3A_1)'$ & $\frac{1}{2}(1,1,1,1,1,1,2)$ & $E_7^{\mathrm{ad}}$ (quaternionic) & 0\\ \hline
       $(3A_1)'$ & $\frac{1}{2}(1,1,1,1,1,1,2)$ & $E_7^{\mathrm{ad}}$ (split) & 0\\ \hline
       
       $4A_1$ & $\frac{1}{2}(1,1,1,1,1,1,1)$ & $E_7^{\mathrm{sc}}$ (Hermitian) & 0 \\ \hline
       $4A_1$ & $\frac{1}{2}(1,1,1,1,1,1,1)$ & $E_7^{\mathrm{sc}}$ (quaternionic) & 0\\ \hline
       $4A_1$ & $\frac{1}{2}(1,1,1,1,1,1,1)$ & $E_7^{\mathrm{sc}}$ (split) & 0\\ \hline
       $4A_1$ & $\frac{1}{2}(1,1,1,1,1,1,1)$ & $E_7^{\mathrm{ad}}$ (Hermitian) & 0 \\ \hline
       $4A_1$ & $\frac{1}{2}(1,1,1,1,1,1,1)$ & $E_7^{\mathrm{ad}}$ (quaternionic) & 0 \\ \hline
       $4A_1$ & $\frac{1}{2}(1,1,1,1,1,1,1)$ & $E_7^{\mathrm{ad}}$ (split) & 0\\ \hline
       
       $2A_2+A_1$ & $\frac{1}{3}(1,1,1,1,1,1,1)$ & $E_7^{\mathrm{sc}}$ (Hermitian) & 0 \\ \hline
       $2A_2+A_1$ & $\frac{1}{3}(1,1,1,1,1,1,1)$ & $E_7^{\mathrm{sc}}$ (quaternionic) & 1 \\ \hline
       $2A_2+A_1$ & $\frac{1}{3}(1,1,1,1,1,1,1)$ & $E_7^{\mathrm{sc}}$ (split) & 1 \\ \hline
       $2A_2+A_1$ & $\frac{1}{3}(1,1,1,1,1,1,1)$ & $E_7^{\mathrm{ad}}$ (Hermitian) & 0\\ \hline
       $2A_2+A_1$ & $\frac{1}{3}(1,1,1,1,1,1,1)$ & $E_7^{\mathrm{ad}}$ (quaternionic) & 0\\ \hline
       $2A_2+A_1$ & $\frac{1}{3}(1,1,1,1,1,1,1)$ & $E_7^{\mathrm{ad}}$ (split) & 0\\ \hline
       
       $(A_3+A_1)'$ & $\frac{1}{2}(1,1,0,1,0,1,1)$ & $E_7^{\mathrm{sc}}$ (Hermitian) & 0 \\ \hline
       $(A_3+A_1)'$ & $\frac{1}{2}(1,1,0,1,0,1,1)$ & $E_7^{\mathrm{sc}}$ (quaternionic) & 0\\ \hline
       $(A_3+A_1)'$ & $\frac{1}{2}(1,1,0,1,0,1,1)$ & $E_7^{\mathrm{sc}}$ (split) & 0\\ \hline
       $(A_3+A_1)'$ & $\frac{1}{2}(1,1,0,1,0,1,1)$ & $E_7^{\mathrm{ad}}$ (Hermitian) & 0 \\ \hline
       $(A_3+A_1)'$ & $\frac{1}{2}(1,1,0,1,0,1,1)$ & $E_7^{\mathrm{ad}}$ (quaternionic) & 0\\ \hline
       $(A_3+A_1)'$ & $\frac{1}{2}(1,1,0,1,0,1,1)$ & $E_7^{\mathrm{ad}}$ (split) & 0\\ \hline
       
    \end{tabular}
    \caption{Non-special unipotents: type $E_7$}
    \label{table:nonspecialE7}
\end{table}

\begin{table}[H]
    \begin{tabular}{|c|c|c|c|} \hline
       $\OO$ & $\gamma_0(\OO)$ & $G$ & $\#\mathrm{Unip}_{\OO}(G)$ \\ \hline
       
       $3A_1$ & $\frac{1}{2}(1,1,1,1,1,1,2,2)$ & $E_8$ (quaternionic) & 0\\ \hline
       $3A_1$ & $\frac{1}{2}(1,1,1,1,1,1,2,2)$ & $E_8$ (split) & 0\\ \hline
       
       $4A_1$ & $\frac{1}{2}(1,1,1,1,1,1,1,1)$ & $E_8$ (quaternionic) & 0\\ \hline
       $4A_1$ &  $\frac{1}{2}(1,1,1,1,1,1,1,1)$ & $E_8$ (split) & 0\\ \hline
       
       $A_2+3A_1$ & $\frac{1}{2}(1,1,1,0,1,1,1,1)$ & $E_8$ (quaternionic) & 0\\ \hline
       $A_2+3A_1$ & $\frac{1}{2}(1,1,1,0,1,1,1,1)$ & $E_8$ (split) & 0\\ \hline
       
       $2A_2+A_1$ & $\frac{1}{3}(1,1,1,1,1,1,1,3)$ & $E_8$ (quaternionic) & 1\\ \hline
       $2A_2+A_1$ & $\frac{1}{3}(1,1,1,1,1,1,1,3)$ & $E_8$ (split) & 1\\ \hline
       
       $A_3+A_1$ & $\frac{1}{2}(1,1,0,1,0,1,1,2)$ & $E_8$ (quaternionic) & 0\\ \hline
       $A_3+A_1$ & $\frac{1}{2}(1,1,0,1,0,1,1,2)$ & $E_8$ (split) & 0\\ \hline
       
       $2A_2+2A_1$ & $\frac{1}{3}(1,1,1,1,1,1,1,1)$ & $E_8$ (quaternionic) & 0\\ \hline
       $2A_2+2A_1$ & $\frac{1}{3}(1,1,1,1,1,1,1,1)$ & $E_8$ (split) & 1\\ \hline
       
       $A_3+2A_1$ & $\frac{1}{2}(1,1,1,0,1,0,1,1)$ & $E_8$ (quaternionic) & 0\\ \hline
       $A_3+2A_1$ & $\frac{1}{2}(1,1,1,0,1,0,1,1)$ & $E_8$ (split) & 0\\ \hline
       
       $A_3+A_2+A_1$ & $\frac{1}{2}(1,0,0,1,0,1,1,1)$ & $E_8$ (quaternionic) & 0\\ \hline
       $A_3+A_2+A_1$ & $\frac{1}{2}(1,0,0,1,0,1,1,1)$ & $E_8$ (split) & 0\\ \hline
       
       $2A_3$ & $\frac{1}{4}(1,1,1,1,1,1,1,1)$ & $E_8$ (quaternionic) & 0\\ \hline
       $2A_3$ & $\frac{1}{4}(1,1,1,1,1,1,1,1)$ & $E_8$ (split) & 0\\ \hline
       
       $A_4+A_3$ & $\frac{1}{5}(1,1,1,1,1,1,1,1)$ & $E_8$ (quaternionic) & 0 \\ \hline
       $A_4+A_3$ & $\frac{1}{5}(1,1,1,1,1,1,1,1)$ & $E_8$ (split) & 1\\ \hline
       
       $A_5+A_1$ & $\frac{1}{6}(2,2,1,1,1,1,1,1)$ & $E_8$ (quaternionic) & 0 \\ \hline
       $A_5+A_1$ & $\frac{1}{6}(2,2,1,1,1,1,1,1)$ & $E_8$ (split) & 0\\ \hline
       
       $D_5(a_1)$ & $\frac{1}{4}(1,1,1,0,1,1,1,1)$ & $E_8$ (quaternionic) & 0 \\ \hline
       $D_5(a_1)$ & $\frac{1}{4}(1,1,1,0,1,1,1,1)$ & $E_8$ (split) & 0\\ \hline

    \end{tabular}
    \caption{Non-special unipotents: type $E_8$}
    \label{table:nonspecialE8}
\end{table}

Below, we list the {\fontfamily{cmtt}\selectfont atlas} parameters of the 12 unipotent representations enumerated in the tables above.

\begin{itemize}
    \item $\OO=A_1$, $\gamma_0(\OO) = \frac{1}{3}(3,1)$, $G=G_2(\mathrm{split})$.
    $$\mathrm{Unip}_{\OO}(G) = \{(x=9,\lambda=(1,2), \nu=\frac{1}{3}(3,1))\}.$$
    
    \item $\OO=\widetilde{A}_2+A_1$, $\gamma_0(\OO) = \frac{1}{3}(1,1,1,1)$, $G=F_4(\mathrm{split})$.
    $$\mathrm{Unip}_{\OO}(G) = \{(x=228,\lambda=(2,2,1,2),\nu=\frac{1}{3}(1,1,1,1))\} $$
    
    \item $\OO=2A_2+A_1$, $\gamma_0(\OO) = \frac{1}{3}(1,1,1,1,1,1)$, $G=E_6(\mathrm{quasisplit})$.
    \begin{align*}\mathrm{Unip}_{\OO}(G) = \{&((x=1790,\lambda=(1,2,1,1,1,1),\nu=\frac{1}{3}(1,1,1,1,1,1)),\\
    &(x=1778,\lambda=(1,2,0,2,1,0),\nu=\frac{1}{6}(2,2,-1,5,2,-1)),\\
    &(x=1777,\lambda=(0,2,1,2,0,1),\nu=\frac{1}{6}(-1,2,2,5,-1,2))\}\end{align*}
    
    \item $\OO=2A_2+A_1$, $\gamma_0(\OO)= \frac{1}{3}(1,1,1,1,1,1)$, $G=E_6(\mathrm{split})$. 
    $$\mathrm{Unip}_{\OO}(G) = \{(x=981,\lambda=(2,2,2,1,2,2),\nu=\frac{1}{3}(1,1,1,1,1,1))\} $$
    \item $\OO=2A_2+A_1$, $\gamma_0(\OO) = \frac{1}{3}(1,1,1,1,1,1,1)$, $G=E_7^{\mathrm{sc}}(\mathrm{quaternionic})$. 
    $$\mathrm{Unip}_{\OO}(G) = \{(x=8920,\lambda=(2,1,3,1,-1,1,1),\nu=\frac{1}{3}(1,1,4,1,-2,1,1))\}.$$
    \item $\OO=2A_2+A_1$, $\gamma_0(\OO) = \frac{1}{3}(1,1,1,1,1,1,1)$, $G=E_7^{\mathrm{sc}}(\mathrm{split})$. 
    $$\mathrm{Unip}_{\OO}(G) = \{((x=20925,\lambda=(2,2,2,1,2,2,1),\nu=\frac{1}{3}(1,1,1,1,1,1,1)))\}.$$
    \item $\OO=2A_2+A_1$, $\gamma_0(\OO)=\frac{1}{3}(1,1,1,1,1,1,1,3)$, $G=E_8(\mathrm{quaternionic})$.
    $$\mathrm{Unip}_{\OO}(G) = \{((x=66576,\lambda=(1,1,-4,4,3,2,-4,1),\nu=(1,1,-8,7,4,1,-5,0)/3)
)\}.$$

    \item $\OO=2A_2+A_1$, $\gamma_0(\OO)=\frac{1}{3}(1,1,1,1,1,1,1,3)$, $G=E_8(\mathrm{split})$.
    $$\mathrm{Unip}_{\OO}(G) = \{(x=320205,\lambda=(2,2,2,1,2,2,1,1),\nu=(1,1,1,1,1,1,1,3)/3))\}.$$
    
    \item $\OO=2A_2+2A_1$, $\gamma_0(\OO) = \frac{1}{3}(1,1,1,1,1,1,1,1)$, $G=E_8(\mathrm{split})$.
    $$\mathrm{Unip}_{\OO}(G) = \{(x=320205,\lambda=(2,2,2,1,2,2,1,2),\nu=(1,1,1,1,1,1,1,1)/3))\}$$
    
    \item $\OO=A_4+A_4$, $\gamma_0(\OO) = \frac{1}{5}(1,1,1,1,1,1,1,1)$, $G=E_8(\mathrm{split})$. 
    $$\mathrm{Unip}_{\OO}(G) = \{((x=320205,\lambda=(2,2,2,2,1,2,2,2),\nu=(1,1,1,1,1,1,1,1)/5)
)\}$$
\end{itemize}

We conclude by remarking on two general patterns, which deserve further consideration. First, none of the representations listed above are spherical. This (perhaps) suggests the following general conjecture: if $G$ is a split real group and $X$ is an irreducible unitary spherical representation, then the associated variety of $X$ is the closure of a special nilpotent orbit. We do not know of a counterexample (in classical or exceptional types). Second, we note that if $\gamma_0(\OO)$ has an even integer in its denominator, then $\unip_{\OO}(G)$ is empty. In these cases, there should be interesting unitary representations of the appropriate two-fold covers.
\appendix

\section{Maximality computations}

In Sections \ref{subsec:appendixspin} and \ref{subsec:appendixexceptional} below, we will prove the maximality of all unipotent ideals for spin and exceptional groups. We will now summarize our approach.

Suppose $G$ is a complex reductive algebraic group with Lie algebra $\fg$. Form the Langlands dual group $G^{\vee}$. If we fix a Cartan subalgebra $\fh \subset \fg$, then $\fg^{\vee}$ contains a Cartan subalgebra $\fh^{\vee} \subset \fg^{\vee}$, which is canonically identified with $\fh^*$, and the roots $\Delta(\fg^{\vee},\fh^{\vee})$ for $\fg^{\vee}$ coincide with the co-roots $\Delta^{\vee}(\fg,\fh)$ for $\fg$. Fix an element $\gamma \in \fh^* \simeq \fh^{\vee}$. Consider the subsystems of $\Delta^{\vee}$ consisting of \emph{integral} and \emph{singular} co-roots
$$\Delta^{\vee}_{\gamma} := \{\alpha^{\vee} \in \Delta^{\vee}(\fg,\fh): \langle\gamma, \alpha^{\vee}\rangle \in \ZZ\}, \qquad \Delta^{\vee}_{\gamma,0} := \{\alpha^{\vee} \in \Delta^{\vee}(\fg,\fh): \langle\gamma, \alpha^{\vee}\rangle =0\} \subset \Delta^{\vee}_{\gamma}$$
These root systems define reductive subalgebras $\mathfrak{l}^{\vee}_{\gamma}$ and $\mathfrak{l}^{\vee}_{\gamma,0}$ of $\fg^{\vee}$. Using the bijection $\Delta^{\vee}(\fg,\fh) \simeq \Delta(\fg,\fh)$ between roots and co-roots, we can produce from $\Delta^{\vee}_{\gamma}$ and $\Delta^{\vee}_{\gamma,0}$ two subsystems of $\Delta(\fg,\fh)$
$$\Delta_{\gamma} := \left(\Delta^{\vee}_{\gamma}\right)^{\vee} \subset \Delta, \qquad \Delta_{\gamma,0} := \left(\Delta^{\vee}_{\gamma,0}\right)^{\vee} \subset \Delta_{\gamma}.$$
These root systems define reductive Lie algebras, denoted by  $\mathfrak{l}_{\gamma}$ and $\mathfrak{l}_{\gamma,0}$. Finally, consider the nilpotent orbits
$$\mathbb{O}^{\vee}_{\gamma} := \Ind^{L_{\gamma}^{\vee}}_{L_{\gamma,0}^{\vee}} \{0\} \subset (\fl_{\gamma}^{\vee})^*, \qquad \mathbb{O}_{\gamma} := \mathsf{D}(\mathbb{O}^{\vee}_{\gamma}) \subset \fl_{\gamma}^*,$$
where $\mathsf{D}$ is Barbasch-Vogan duality (see \cite[Appendix]{BarbaschVogan1985}). The criterion below is a standard consequence of the Barbasch-Vogan algorithm for the associated variety of a maximal ideal, see \cite[Prop 3.3.1]{LMBM} for a formal proof.

\begin{prop}\label{prop:maximalitycriterion}
Suppose $I \subset U(\fg)$ is a primitive ideal with infinitesimal character $\gamma \in \fh^*/W$ and associated variety $\overline{\OO}$. Then $I$ is a maximal ideal if and only if
$$\codim(\mathbb{O},\cN) = \codim(\mathbb{O}_{\gamma},\cN_{L_{\gamma}}).$$
\end{prop}

\subsection{Spin groups}\label{subsec:appendixspin}

Let $\fg = \fg(n)$ be a simple rank-$n$ Lie algebra of type $B$ or $D$. We start with a technical lemma. 

\begin{lemma}\label{lem:lemmaspin}
Let $q=(q_1,...,q_m)$ be a partition of $n$. Consider the infinitesimal character defined in standard coordinates by the formula
$$\gamma := \rho^+(h'(q \cup q))$$
(cf. Definition (\ref{def:combinatorics})). Then
$$\codim(\OO_{\gamma},\cN_{L_{\gamma}}) = \sum_{i=1}^m \dim \cN_{\mathrm{GL}(q_i)}.$$
\end{lemma}

\begin{proof}
Let $q^t =: r = (r_1,...,r_{q_1})$. Up to permutations
$$\gamma = \frac{1}{4}((2q_1-1)^{r_{q_1}}, (2q_1-3)^{r_{q_1-1}}, ..., 1^{r_1})$$
where we use superscripts to denote the multiplicities of repeated entries in $\gamma$. It is easy to see that 
$$\mathfrak{l}^{\vee}_{\gamma} \simeq \mathfrak{gl}(n), \qquad \mathfrak{l}^{\vee}_{\gamma,0} \simeq \mathfrak{gl}(r_1) \times \mathfrak{sl}(r_2) \times ... \times \mathfrak{sl}(r_{q_1})$$
Thus
$$\OO_{\gamma}^{\vee} = \Ind^{L^{\vee}_{\gamma}}_{L^{\vee}_{\gamma,0}} \{0\} = \OO_q$$
where $\OO_q$ denotes the nilpotent orbit in $\mathfrak{gl}(n)^*$ corresponding to $q$. Clearly $\mathfrak{l}_{\gamma} \simeq \mathfrak{gl}(n)$ and 
$$\OO_{\gamma} = d(\OO_{\gamma}^{\vee}) = \OO_r$$
Note that $\OO_r$ is induced from the $\{0\}$-orbit of the Levi $\mathfrak{gl}(q_1) \times \mathfrak{sl}(q_2) \times ... \times \mathfrak{sl}(q_m)$. So by (iv) of Proposition \ref{prop:propsofInd}
$$\codim(\OO_{\gamma}, \cN_{L_{\gamma}}) = \codim(\{0\},\cN_{\mathrm{GL}(q_1) \times ... \times \mathrm{GL}(q_m)}) = \sum_{i=1}^m \dim \cN_{\mathrm{GL}(q_i)},$$
as asserted.

\end{proof}

\begin{prop}\label{prop:maximalitySpin}
Let $G$ be a simple rank-$n$ spin-group of type $B$ or $D$ and let $\widetilde{\OO}$ be a $G$-equivariant nilpotent cover. Then $I_0(\widetilde{\OO})$ is a maximal ideal.
\end{prop}

\begin{proof}
Choose a Levi subalgebra
$$\fm = \prod_i \mathfrak{gl}(a_i) \times \mathfrak{g}(m) \subset \fg,$$
and a birationally rigid $M$-equivariant nilpotent cover
$$\{0\} \times ... \times \{0\} \times \widetilde{\OO}'$$
such that $\widetilde{\OO} = \mathrm{Bind}^G_M \{0\} \times ... \times \{0\} \times \widetilde{\OO}'$. 
Here, $m \leq n$ and $a$ is a partition of $n-m$. Write $\gamma = \gamma_0(\widetilde{\OO})$ and $\gamma'=\gamma_0(\widetilde{\OO}')$. By Proposition \ref{prop:nodeltashift}
\begin{equation}\label{eq:gamma}\gamma= (\rho(a),\gamma').\end{equation}
Let $p$ denote the partition corresponding to $\OO'$, the orbit in $\mathfrak{g}(m)$ of which $\widetilde{\OO}'$ is a cover. Define $S_4(p)$, $p \# S_4(p)$, $x=x((p\#S_4(p))^t)$, $y=y((p\#S_4(p))^t)$, and $z=z((p\#S_4(p))^t)$ as in Definition \ref{def:combinatorics}. Since $\widetilde{\OO}'$ is a birationally rigid $\mathrm{Spin}(m)$-equivariant cover, we can construct a Levi subalgebra
$$\fl = \prod_{k \in S_4(p)} \mathfrak{gl}(k) \times \mathfrak{g}(m-|S_4(p)|) \subset \fg(m),$$
such that 
\begin{equation}\label{eq:gammaprime}\gamma' =(\rho^+(h'(z^{1/2})), \gamma_0(\widetilde{\OO}_{p\#S_4(p)})),\end{equation}
where $\widetilde{\OO}_{p \#S_4(p)}$ is the universal $\mathrm{SO}(m-|S_4(p)|)$-equivariant cover of $\OO_{p\#S_4(p)}$. This follows from the proof of Proposition \ref{prop:inflcharspin}. Substituting (\ref{eq:gammaprime}) into (\ref{eq:gamma}), we obtain the following formula for $\gamma$
\begin{equation}\label{eq:gamma2}\gamma = (\rho(a),\rho^+(h'(z^{1/2})), \gamma_0(\widetilde{\OO}_{p\#S_4(p)})).\end{equation}
Form the levi subalgebra of $\fg(n-|S_4(p)|)$
$$\mathfrak{k} = \prod_{i}\mathfrak{gl}(a_i) \times \mathfrak{g}(m-|S_4(p)|) \subset \fg(n-|S_4(p)|),$$
and consider the $\mathrm{SO}(n-|S_4(p)|)$-equivariant nilpotent cover 
$$\widetilde{\OO}'' = \mathrm{Bind}^{\mathrm{SO}(n-|S_4(p)|)}_{K} \{0\} \times ... \times \{0\} \times \widetilde{\OO}_{p\#S_4(p)}.$$
If we write $\gamma'' = \gamma_0(\widetilde{\OO}'')$, then again by Proposition \ref{prop:nodeltashift}
\begin{equation}\label{eq:gammaprimeprime}\gamma''= (\rho(a),\gamma_0(\widetilde{\OO}_{p\#S_4(p)})).\end{equation}
Combining (\ref{eq:gamma2}) and (\ref{eq:gammaprimeprime}) (and permuting entries) we obtain the following formula for $\gamma$
$$\gamma = (\rho^+(h'(z^{1/2})), \gamma'').$$
By \cite[Thm 8.5.1]{LMBM}, $I_0(\widetilde{\OO}'')$ is a maximal ideal in $U(\mathfrak{g}(n-|S_4(p)|)$. So by Proposition \ref{prop:maximalitycriterion}
$$\codim(\OO_{\gamma''}, \cN_{L_{\gamma''}}) = \codim(\OO'', \cN_{\mathrm{SO}(n-|S_4(p)|)}) $$
On the other hand, since $\OO$ is induced from $\{0\} \times ... \times \{0\} \times \OO'' \subset \prod_{k \in S_4(p)} \mathfrak{gl}(k) \times \mathfrak{g}(n-|S_4(p)|)$, we have by Proposition \ref{prop:propsofInd}(iv)
$$\codim(\OO,\cN_G) = \codim(\OO'',\cN_{\mathrm{SO}(n-|S_4(p)|)}) + \sum_{k \in S_4(p)} \dim \cN_{\mathrm{GL}(k)}.$$
In view of Proposition \ref{prop:maximalitycriterion}, it suffices to show that
\begin{equation}\label{eq:eq2}\codim(\OO_{\gamma},\cN_{L_{\gamma}}) = \codim(\OO_{\gamma''},\cN_{L_{\gamma''}}) + \sum_{k \in S_4(p)}\dim \cN_{\mathrm{GL}(k)}.\end{equation}
Note that $\gamma''$ is a tuple in $\frac{1}{2}\ZZ$, while $\rho^+(h'(z^{1/2}))$ is a tuple in $\frac{1}{4} + \frac{1}{2}\ZZ$.
In particular, the sum or difference of an entry in $\gamma''$ with an entry in $\rho^+(h'(z^{1/2}))$ is never contained in $\ZZ$. So $\fl^{\vee}_{\gamma}$ splits as a product
$$\fl^{\vee}_{\gamma} = \fl^{\vee}_{\gamma''} \times \fl^{\vee}_{\rho^+(h'(z^{1/2}))},$$
where $\fl^{\vee}_{\rho^+(h'(z^{1/2}))}$ is the subalgebra in $\fg(|S_4(p)|)$ corresponding to the integral co-roots for $\fl^{\vee}_{\rho^+(h'(z^{1/2}))}$, regarded as an infinitesimal character for $\fg(|S_4(p)|)$. As an immediate consequence we obtain
$$\fl_{\gamma} = \fl_{\gamma''} \times \fl_{\rho^+(h'(z^{1/2}))}, \qquad \OO_{\gamma} = \OO_{\gamma''} \times \OO_{\rho^+(h'(z^{1/2}))},$$
Since $z^{1/2} = S_4(p) \cup S_4(p)$, Lemma \ref{lem:lemmaspin} (applied to $q=S_4(p)$) implies
\begin{align*}
\codim(\OO_{\gamma}, \cN_{L_{\gamma}}) &= \codim(\OO_{\gamma''}, \cN_{L_{\gamma''}}) + \codim(\OO_{\rho^+(h'(z^{1/2}))}, \cN_{L_{\rho^+(h'(z^{1/2}))}}) \\
&= \codim(\OO_{\gamma''}, \cN_{L_{\gamma''}}) + \sum_{k \in S_4(p)} \dim \cN_{\mathrm{GL}(k)}.\end{align*}
This proves (\ref{eq:eq2}) and thus completes the proof.
\end{proof}

\subsection{Exceptional groups}\label{subsec:appendixexceptional}

\begin{prop}\label{prop:maximalityexceptional}
Let $G$ be a simple simply connected group of exceptional type and let $\widetilde{\OO}$ be a $G$-equivariant nilpotent cover. Then $I_0(\widetilde{\OO})$ is a maximal ideal.
\end{prop}

The proof will come after a lemma.

\begin{lemma}\label{lem:E7}
Let $\OO$ be the nilpotent orbit $E_7(a_4)$ in $E_7$ and let $\widetilde{\OO}$ be a 2-fold cover of $\OO$. Then $\Spec(\CC[\widetilde{\OO}])$ contains a codimension 2 leaf.
\end{lemma}

\begin{proof}
    Assume the opposite. We note that $\pi_1(\OO)\simeq \ZZ_2\times \ZZ_2$, and $\pi_1(\widetilde{\OO})\simeq \ZZ_2$. There are three codimension $2$ orbits in $\overline{\OO}$, namely $\OO_1=A_6$, $\OO_2=D_5+A_1$ and $\OO_3=D_6(a_1)$. For each of these orbits the corresponding singularity $\Sigma_j$ is of type $A_1$. Let $\widehat{\OO}$ be the universal cover of $\OO$. The preimage of each $\Sigma_i$ in $\widehat{X}$ is the union of two copies of $\CC^2$. Let $K\subset \pi_1(\OO)$ be such that $\pi_1(\OO)\simeq \pi_1(\widetilde{\OO})\oplus K$, and let $\widecheck{\OO}$ be the corresponding cover of $\OO$. Since the map $\pi_1(\OO)\to S_2$ permuting the components of the preimage of $\Sigma_i$ has kernel $\pi_1(\widetilde{\OO})$, it follows that the preimage of $\Sigma_i$ in $\widecheck{X}$ is the union of two copies of $\Sigma_i$ for each $i=1,2,3$. Therefore, $\dim \fP^{\widecheck{X}}\ge 3$. Using \cite[Tables]{deGraafElashvili}, we see that there is a single pair of Levi $L$ of corank $3$ and a nilpotent orbit $\OO_L\subset \fl^*$, such that $\OO$ is induced from $\OO_L$, namely $(L(A_2+2A_1; 2,3,5,6), \{0\})$. Thus, there is at most one cover birationally induced from a corank $3$ Levi. However, we have $\dim \fP^X=3$, and $\dim \fP^{\widecheck{X}}\ge 3$, so we get a contradiction. 
\end{proof}

\begin{proof}[Proof of Proposition \ref{prop:maximalityexceptional}]
Let $L \subset G$ be a Levi subgroup of $G$. The universal cover $\widetilde{L}$ of $L$ is a product $T \times \widetilde{L}_1 \times ... \times \widetilde{L}_t$, where $T$ is a torus and $\widetilde{L}_1,...,\widetilde{L}_t$ are simple simply connected groups. For each $i$, choose a birationally rigid $\widetilde{L}_i$-equivariant nilpotent cover $\widetilde{\OO}_i$ and regard  $\gamma_0(\widetilde{\OO}_i)$ as an infinitesimal character for $\fg$ via the natural embedding $\fl_i \subset \fl$. Consider the infinitesimal character
\begin{equation}\label{eq:fictionalgamma}\gamma = \sum_{i=1}^t \gamma_0(\widetilde{\OO}_i) \in \fh^*/W\end{equation}
Any infinitesimal character which arises in this way will be called \emph{pseudo-unipotent}. By Proposition \ref{prop:nodeltashift}, every unipotent infinitesimal character is pseudo-unipotent, but there are many others. 

In the tables below, we list all pseudo-unipotent infinitesimal character for simple exceptional groups. We will explain the procedure and give an example:
\begin{itemize}
    \item List all standard Levi subgroups $L \subset G$. 
    \item For each $L$, determine the simple factors $\fl_i$ of $\fl$. This is evident from the Dynkin diagram.
    \item For each simple factor $\fl_i$ compute the fundamental weights for $\fl_i$ in terms of fundamental weights for $\fg$. 
    \item For each simple factor $\fl_i$ of $\fl$, list, in fundamental weight coordinates, all infinitesimal characters $\gamma_0(\widetilde{\OO}_i)$ attached to birationally rigid covers for the simply connected group $\widetilde{L}_i$. If $\fl_i$ is exceptional, we use the tables in Sections \ref{subsec:exceptionalorbitinflchar} and \ref{subsec:exceptionalcoverinflchar}. If $\fl_i$ is classical, we use \cite[Prop 8.2.8]{LMBM}. 
    \item For each simple factor $\fl_i$ and $\gamma_0(\widetilde{\OO}_i)$, write the infinitesimal character $\gamma_0(\widetilde{\OO}_i)$ in terms of fundamental weights for $G$. 
    \item Compute $\gamma=\sum_{i=1}^t \gamma_0(\widetilde{\OO}_i)$.
\end{itemize}

\begin{examplenonumber}
Let $G=E_8$ and choose $L=L(D_4+A_2;2,3,4,5,7,8)$. Let $\fl_1$ denote the $D_4$ factor and let $\fl_2$ denote the $A_2$ factor. In terms of fundamental weights for $\fg$, the fundamental weights for $\fl_1$ are
$$\frac{1}{2}(-1,2,0,0,0,-1,0,0), \frac{1}{2}(-2,0,2,0,0,-1,0,0), (-1,0,0,1,0,-1,0,0), \frac{1}{2}(-1,0,0,0,2,-2,0,0),$$
where the third weight corresponds to the central node. The fundamental weights for $\fl_2$ are
$$\frac{1}{3}(0,0,0,0,0,-2,3,0), \frac{1}{3}(0,0,0,0,0,-1,0,3).$$
Choose $\widetilde{\OO}_1$ to be the minimal nilpotent orbit in $\fl_1$, corresponding to the partition $(2^2,1^4)$. By \cite[Prop 8.2.8]{LMBM}, $\gamma_0(\widetilde{\OO}_1)$ is the sum of the three fundamental weights corresponding to the non-central nodes, i.e.
\begin{align*}
\gamma_0(\widetilde{\OO}_1) &=  \frac{1}{2}(-1,2,0,0,0,-1,0,0) + \frac{1}{2}(-2,0,2,0,0,-1,0,0) + \frac{1}{2}(-1,0,0,0,2,-2,0,0)\\
&= (-2,1,1,0,1,-2,0,0)\end{align*}
Choose $\widetilde{\OO}_2$ to be the universal cover of the principal nilpotent orbit in $\fl_2$. By \cite[Prop 8.2.8]{LMBM}, $\gamma_0(\widetilde{\OO}_2)$ is $\rho(\fl_2)/3$, i.e.
$$\gamma_0(\widetilde{\OO}_2) = \frac{1}{3}(\frac{1}{3}(0,0,0,0,0,-2,3,0)+ \frac{1}{3}(0,0,0,0,0,-1,0,3)) = \frac{1}{3}(0,0,0,0,0,-1,1,1).$$
So
$$\gamma = \gamma_0(\widetilde{\OO}_1)+\gamma_0(\widetilde{\OO}_2) = \frac{1}{3}(-6,3,3,0,3,-7,1,1).$$
This is $W$-conjguate to the dominant weight
$$\gamma = \frac{1}{3}(0,1,0,0,0,0,2,0).$$
\end{examplenonumber}

If $\gamma$ is a pseudo-unipotent infinitesimal character corresponding to the data $(L \subset G, \widetilde{\OO}_1,...,\widetilde{\OO}_t)$, define the integer
$$n(\gamma) := \sum_{i=1}^t \codim(\OO_i, \cN_{\fl_i})$$
For each pseudo-unipotent infinitesimal character, we will record $n(\gamma)$ in the tables below. Suppose $\gamma$  is the infinitesimal character of a unipotent ideal $I_0(\widetilde{\OO})$. Then $\OO$ is induced from $\OO_1 \times ... \times \OO_t \subset \fl^*$ and
$$n(\gamma) = \codim(\OO, \cN_G).$$
So if
\begin{equation}\label{eq:eq4}n(\gamma) = \codim(\OO_{\gamma}, \cN_{L_{\gamma}}),\end{equation}
then
$$\codim(\OO,\cN_G) = \codim(\OO_{\gamma}, \cN_{L_{\gamma}}),$$
which implies that $I_0(\widetilde{\OO})$ is maximal by Proposition \ref{prop:maximalitycriterion}. Thus, it suffices to show that (\ref{eq:eq4}) holds for every unipotent $\gamma$.

In fact, we compute $\codim(\OO_{\gamma},\cN_{L_{\gamma}})$ for every \emph{pseudo-unipotent} $\gamma$ in the tables below. In types $G_2$, $F_4$, $E_6$, and $E_7$, the {\fontfamily{cmtt}\selectfont atlas} command `$\mathrm{GK\_dim\_maximal\_ideal}(G,\gamma)$' is helpful, but one can also compute by hand, as illustrated below.

\begin{examplenonumber}
Consider the pseudo-unipotent infinitesimal character computed in the example above
$$\gamma = \frac{1}{3}(0,1,0,0,0,0,2,0).$$
The following is a simple system for $\Delta^{\vee}_{\gamma}$
$$\alpha_1^{\vee}, \alpha_3^{\vee},\alpha_4^{\vee},\alpha_5^{\vee},\alpha_6^{\vee},\alpha_8^{\vee}, \alpha_2^{\vee}+\alpha_4^{\vee}+\alpha_5^{\vee}+\alpha_6^{\vee}+\alpha_7^{\vee}.$$
And below is a simple system for $\Delta^{\vee}_0$
$$\alpha_1^{\vee}, \alpha_3^{\vee},\alpha_4^{\vee},\alpha_5^{\vee},\alpha_6^{\vee},\alpha_8^{\vee}$$
Thus, $\fl^{\vee}_{\gamma}$ is of type $E_7$ and $\fl^{\vee}_{\gamma,0}$ is a Levi of type $A_5+A_1$ in $\fl^{\vee}_{\gamma}$ (the embedding $\fl_{\gamma,0}^{\vee} \subset \fl_{\gamma}^{\vee}$ matters---$\fl_{\gamma,0}^{\vee}$ can be regarded as the standard Levi in $E_7$ obtained by deleting the node adjacent to the end point of the short leg of the Dynkin diagram). Now $\OO^{\vee}_{\gamma}$ is the Richardson orbit in $\fl^{\vee}_{\gamma}$ induced from the $\{0\}$-orbit in $\fl^{\vee}_{\gamma,0}$. Hence by \cite[Sec 4]{deGraafElashvili}, $\OO_{\gamma}^{\vee} = D_4(a_1)$. By \cite[Sec 13.4]{Carter1993}, $\OO_{\gamma} = \mathsf{D}(\OO_{\gamma}^{\vee}) = E_7(a_5) \subset E_7$. And by \cite[Sec 8.4]{CM}, $\codim(E_7(a_5), \cN_{E_7}) = 14$.
\end{examplenonumber}

Comparing $\codim(\OO_{\gamma},\cN_{L_{\gamma}})$ and $n(\gamma)$ for every pseudo-unipotent $\gamma$, we find that (\ref{eq:eq4}) holds in all but one case (highlighted in red in the tables below). The case in question is
\begin{equation}\label{eq:badpseudo}G=E_8, \qquad \gamma = \frac{1}{4}(0,1,0,0,1,0,0,1)\end{equation}
This infinitesimal character corresponds to the following choices: 
$$L = L(E_7), \qquad \widetilde{\OO}_1 =  \text{universal cover of } E_7(a_4) \subset \cN_L$$
Indeed, in fundamental weight coordinates for $\fl$
$$\gamma_0(\widetilde{\OO}_1) = \frac{1}{4}(1,2,0,0,1,0,1),$$
see Table \ref{table:covers:E7}. Converting into fundamental weight coordinates for $\fg$, we obtain the weight
\begin{equation}\label{eq:badpseudo1}\gamma_0(\widetilde{\OO}_1) = \frac{1}{4}(1,2,0,0,1,0,1,-8).\end{equation}
The dominant weight in (\ref{eq:badpseudo}) is obtained from (\ref{eq:badpseudo1}) by conjugating by $W$. Note that we have
$$n(\gamma) =\codim(E_7(a_4), \cN_{E_7}) = 10,$$
whereas
$$\codim(\OO_{\gamma},L_{\gamma}) = \codim(\OO_{(4,3,1)} \times \{0\}, \cN_{\mathfrak{sl}(2)\times \mathfrak{sl}(8)}) = 12 $$
To show that $\gamma$ is \emph{not} unipotent, we must show that the orbit $\OO_L = E_7(a_4) \subset \fl^*$ does not admit an $L$-equivariant cover in the same equivalence class as the universal cover of $\OO_L$. An {\fontfamily{cmtt}\selectfont atlas} computation shows that $\pi_1^L(\OO_L) \simeq \ZZ_2$. Write $\widetilde{\OO}_L$ for the nontrivial $L$-equivariant cover of $\OO_L$. By Lemma \ref{lem:E7}, $\Spec(\CC[\widetilde{\OO}_L])$ contains a codimension 2 leaf. On the other hand, by Corollary \ref{cor:criterionbirigidcover}, $\Spec(\CC[\widehat{\OO}_L])$ does not. Consequently, $[\widetilde{\OO}_L] \neq [\widehat{\OO}_L]$, as desired. This completes the proof.
\end{proof}

\begin{longtable}{|c|c|c|c|c|c|c|c|}
\caption{Pseudo-unipotent infinitesimal characters: type $G_2$}\label{table:fictitiousG2}\\ \hline
central char & codim & central char & codim & central char & codim & central char & codim \\ \hline

\cellcolor{blue!20}$(0,0)$ & $0$ & 
\cellcolor{blue!20}$\frac{1}{2}(1,0)$ & $2$ &
\cellcolor{blue!20}$\frac{1}{2}(0,1)$ & $2$ &
\cellcolor{blue!20}$(1,0)$ & $2$ \\ \hline

$\frac{1}{2}(1,1)$ & $4$ &
$\frac{1}{3}(3,1)$ & $6$ &
\cellcolor{blue!20}$(1,1)$ & $12$ & & \\ \hline

\end{longtable}

\begin{longtable}{|c|c|c|c|c|c|c|c|}
\caption{Pseudo-unipotent infinitesimal characters: type $F_4$}\label{table:fictitiousG2}\\ \hline
central char & codim & central char & codim & central char & codim & central char & codim \\ \hline

\cellcolor{blue!20}$(0,0,0,0)$ & $0$ & 
$\frac{1}{3}(1,0,0,0)$ & $0$ & 
$\frac{1}{3}(0,0,0,1)$ & $0$ &
$\frac{1}{4}(1,0,0,0)$ & $0$ \\ \hline

$\frac{1}{4}(0,0,1,0)$ & $0$ &
$\frac{1}{4}(0,0,0,1)$ & $0$ &
$\frac{1}{12}(1,0,3,0)$ & $0$ &
$\frac{1}{12}(0,1,2,0)$ & $0$ \\ \hline

\cellcolor{blue!20}$\frac{1}{2}(1,0,0,0)$ & $2$ &
\cellcolor{blue!20}$\frac{1}{2}(0,0,0,1)$ & $2$ &
$\frac{1}{4}(1,0,1,0)$ & $2$ &
$\frac{1}{4}(0,1,0,0)$ & $2$ \\ \hline

$\frac{1}{4}(1,0,0,2)$ & $2$ &
$\frac{1}{6}(0,1,1,0)$ & $2$ &
$\frac{1}{6}(1,0,2,0)$ & $2$ &
$\frac{1}{2}(1,0,0,1)$ & $4$ \\ \hline

\cellcolor{blue!20}$\frac{1}{2}(0,0,1,0)$ & $4$ &
\cellcolor{blue!20}$(0,0,0,1)$ & $6$ & 
\cellcolor{blue!20}$(1,0,0,0)$ & $6$ &
$\frac{1}{4}(3,0,1,0)$ & $6$ \\ \hline

$\frac{1}{4}(1,0,2,0)$ & $6$ &
$\frac{1}{4}(0,1,0,2)$ & $6$ &
\cellcolor{blue!20}$(0,0,1,0)$ & $8$ &
\cellcolor{blue!20}$\frac{1}{2}(1,0,0,2)$ & $8$ \\ \hline

\cellcolor{blue!20}$\frac{1}{2}(0,1,0,0)$ & $8$ &
\cellcolor{blue!20}$\frac{1}{2}(0,1,0,1)$ & $8$ &
\cellcolor{blue!20}$\frac{1}{2}(1,0,1,0)$ & $8$ &
$\frac{1}{2}(1,1,0,0)$ & $10$ \\ \hline

$\frac{1}{2}(0,1,0,2)$ & $12$ &
$\frac{1}{3}(1,1,1,1)$ & $12$ &
$\frac{1}{4}(1,0,2,4)$ & $12$ &
$\frac{1}{4}(1,1,2,2)$ & $14$ \\ \hline

\cellcolor{blue!20}$\frac{1}{2}(2,1,0,1)$ & $18$ &
\cellcolor{blue!20}$(0,0,1,1)$ & $18$ &
\cellcolor{blue!20}$(1,0,1,0)$ & $20$ &
\cellcolor{blue!20}$(1,0,1,1)$ & $26$ \\ \hline

$\frac{1}{2}(1,1,2,2)$ & $32$ &
\cellcolor{blue!20}$(1,1,1,1)$ & $48$ & & & &\\ \hline

\end{longtable}

\begin{longtable}{|c|c|c|c|c|c|}
\caption{Pseudo-unipotent infinitesimal characters: type $E_6$}\label{table:fictitiousE8}\\ \hline
central char & codim & central char & codim & central char & codim  \\ \hline

\cellcolor{blue!20}$(0,0,0,0,0,0)$ & $0$ &
$\frac{1}{3}(0,1,0,0,0,0)$ & $0$ &
$\frac{1}{3}(1,0,0,0,0,1)$ & $0$ \\ \hline

$\frac{1}{4}(0,1,0,0,0,0)$ & $0$ & 
$\frac{1}{4}(1,0,0,0,0,1)$ & $0$ &
$\frac{1}{4}(0,0,0,1,0,0)$ & $0$ \\ \hline

$\frac{1}{5}(1,1,0,0,0,1)$ & $0$ &
$\frac{1}{8}(1,0,1,0,1,1)$ & $0$ &
$\frac{1}{8}(1,2,0,0,0,1)$ & $0$ \\ \hline

$\frac{1}{12}(1,0,0,3,0,1)$ & $0$ &
$\frac{1}{12}(2,1,1,0,1,2)$ & $0$ &
$\frac{1}{12}(3,1,0,0,0,3)$ & $0$ \\ \hline

$\frac{1}{12}(0,0,1,2,1,0)$ & $0$ &
$\frac{1}{20}(1,0,3,1,3,1)$ & $0$ &
\cellcolor{blue!20}$\frac{1}{2}(0,1,0,0,0,0)$ & $2$ \\ \hline

$\frac{1}{4}(0,0,1,0,1,0)$ & $2$ &
$\frac{1}{4}(1,1,0,0,0,1)$ & $2$ &
$\frac{1}{6}(2,1,0,0,0,2)$ & $2$ \\ \hline

$\frac{1}{6}(0,0,1,1,1,0)$ & $2$ &
$\frac{1}{8}(2,0,1,0,1,2)$ & $2$ &
$\frac{1}{10}(2,0,1,1,1,2)$ & $2$ \\ \hline

$\frac{1}{12}(1,0,2,1,2,1)$ & $2$ &
\cellcolor{blue!20}$\frac{1}{2}(1,0,0,0,0,1)$ & $4$ &
$\frac{1}{4}(1,2,0,0,0,1)$ & $4$ \\ \hline

$\frac{1}{4}(1,0,0,1,0,1)$ & $4$ &
$\frac{1}{4}(0,1,1,0,1,0)$ & $4$ &
$\frac{1}{6}(1,0,0,2,0,1)$ & $4$ \\ \hline

\cellcolor{blue!20}$(0,1,0,0,0,0)$ & $6$ &
\cellcolor{blue!20}$\frac{1}{2}(0,0,0,1,0,0)$ & $6$ &
$\frac{1}{3}(0,1,1,0,1,0)$ & $6$ \\ \hline

$\frac{1}{3}(1,2,0,0,0,1)$ & $6$ &
$\frac{1}{4}(1,3,0,0,0,1)$ & $6$ &
$\frac{1}{4}(0,2,1,0,1,0)$ & $6$ \\ \hline

$\frac{1}{4}(1,0,1,0,1,1)$ & $6$ &
$\frac{1}{6}(2,1,1,0,1,2)$ & $6$ &
$\frac{1}{12}(1,5,3,0,3,1)$ & $6$ \\ \hline

\cellcolor{blue!20}$\frac{1}{2}(1,1,0,0,0,1)$ & $8$ &
$\frac{1}{2}(0,1,0,1,0,0)$ & $8$ &
$\frac{1}{4}(1,1,1,0,1,1)$ & $8$ \\ \hline

$\frac{1}{6}(1,1,2,0,2,1)$ & $6$ &
$\frac{1}{6}(1,2,1,1,1,1)$ & $6$ & 
\cellcolor{blue!20}$\frac{1}{2}(0,0,1,0,1,0)$ & $10$ \\ \hline

$\frac{1}{4}(1,2,1,0,1,1)$ & $10$ &
\cellcolor{blue!20}$(1,0,0,0,0,1)$ & $12$ &
\cellcolor{blue!20}$\frac{1}{2}(1,2,0,0,0,1)$ & $12$ \\ \hline

$\frac{1}{4}(1,2,0,2,0,1)$ & $12$ &
$\frac{1}{4}(1,3,1,0,1,1)$ & $12$ &
$\frac{1}{4}(3,0,0,1,0,3)$ & $12$ \\ \hline

$\frac{1}{4}(2,1,1,0,1,2)$ & $12$ &
\cellcolor{blue!20}$(0,0,0,1,0,0)$ & $14$ &
\cellcolor{blue!20}$\frac{1}{2}(0,1,1,0,1,0)$ & $14$ \\ \hline

\cellcolor{blue!20}$\frac{1}{2}(1,0,0,1,0,1)$ & $14$ &
$\frac{1}{2}(1,1,0,1,0,1)$ & $16$ &
$\frac{1}{3}(1,1,1,1,1,1)$ & $18$ \\ \hline

\cellcolor{blue!20}$(1,1,0,0,0,1)$ & $20$ &
$\frac{1}{4}(3,3,1,0,1,3)$ & $20$ & 
\cellcolor{blue!20}$\frac{1}{2}(1,1,1,0,1,1)$ & $22$ \\ \hline

\cellcolor{blue!20}$(0,1,0,1,0,0)$ & $24$ &
$\frac{1}{3}(1,3,1,1,1,1)$ & $24$ &
$\frac{1}{4}(1,4,1,2,1,1)$ & $24$ \\ \hline

\cellcolor{blue!20}$\frac{1}{2}(1,2,1,0,1,1)$ & $26$ &
\cellcolor{blue!20}$(1,0,0,1,0,1)$ & $30$ &
\cellcolor{blue!20}$\frac{1}{2}(2,1,1,0,1,2)$ & $30$ \\ \hline

$\frac{1}{2}(1,1,1,1,1,1)$ & $32$ &
\cellcolor{blue!20}$(1,1,0,1,0,1)$ & $40$ &
\cellcolor{blue!20}$(1,1,1,0,1,1)$ & $50$ \\ \hline

\cellcolor{blue!20}$(1,1,1,1,1,1)$ & $72$ & & & & \\ \hline

\end{longtable}

\begin{longtable}{|c|c|c|c|c|c|}
\caption{Pseudo-unipotent infinitesimal characters: type $E_7$}\label{table:fictitiousE8}\\ \hline
central char & codim & central char & codim & central char & codim  \\ \hline

\cellcolor{blue!20}$(0,0,0,0,0,0,0)$ & $0$ & $\frac{1}{2}(0,0,0,0,0,0,1)$ & $0$ & $\frac{1}{3}(1,0,0,0,0,0,0)$ & $0$ \\ \hline

$\frac{1}{3}(0,0,0,0,0,1,0)$ & $0$ & $\frac{1}{4}(1,0,0,0,0,0,0)$ & $0$ & $\frac{1}{4}(0,0,1,0,0,0,0)$ & $0$\\ \hline

$\frac{1}{4}(0,1,0,0,0,0,1)$ & $0$ & $\frac{1}{5}(1,0,0,0,0,1,0)$ & $0$ & $\frac{1}{6}(0,2,0,0,0,0,1)$ & $0$ \\ \hline

$\frac{1}{6}(1,0,0,0,0,1,1)$ & $0$ & $\frac{1}{7}(0,0,0,1,0,1,0)$ & $0$ & $\frac{1}{8}(1,0,0,0,0,1,2)$ & $0$ \\ \hline

$\frac{1}{8}(0,0,0,1,0,1,0)$ & $0$ & $\frac{1}{8}(0,0,1,0,1,0,1)$ & $0$ & $\frac{1}{8}(2,0,0,0,0,1,0)$ & $0$ \\ \hline

$\frac{1}{12}(0,0,1,1,1,0,1)$ & $0$ & $\frac{1}{12}(0,0,2,1,0,0,0)$ & $0$ & $\frac{1}{12}(1,0,0,1,0,2,0)$ & $0$ \\ \hline

$\frac{1}{12}(0,0,3,0,0,1,0)$ & $0$ & $\frac{1}{12}(1,0,0,0,0,3,0)$ & $0$ &
$\frac{1}{15}(0,0,2,1,0,2,0)$ & $0$ \\ \hline

$\frac{1}{20}(0,0,1,3,0,1,0)$ & $0$ &
$\frac{1}{24}(0,2,3,0,3,0,1)$ & $0$ &
$\frac{1}{24}(0,0,2,1,0,5,0)$ & $0$ \\ \hline

\cellcolor{blue!20}$\frac{1}{2}(1,0,0,0,0,0,0)$ & $2$ &
\cellcolor{blue!20}$\frac{1}{2}(0,1,0,0,0,0,0)$ & $2$ & 
$\frac{1}{4}(1,0,0,0,0,0,2)$ & $2$ \\ \hline

$\frac{1}{4}(0,0,0,1,0,0,0)$ & $2$ &
$\frac{1}{4}(1,0,0,0,0,1,0)$ & $2$ &
$\frac{1}{4}(0,0,0,0,1,0,1)$ & $2$ \\ \hline

$\frac{1}{6}(0,0,1,1,0,0,0)$ & $2$ &
$\frac{1}{6}(1,0,0,0,0,2,0)$ & $2$ &
$\frac{1}{6}(0,1,1,0,0,1,0)$ & $2$ \\ \hline

$\frac{1}{8}(0,0,0,1,0,2,0)$ & $2$ &
$\frac{1}{8}(1,0,0,0,0,2,2)$ & $2$ &
$\frac{1}{8}(0,0,1,1,0,0,2)$ & $2$ \\ \hline

$\frac{1}{8}(0,2,1,0,0,1,0)$ & $2$ &
$\frac{1}{10}(0,0,1,1,0,2,0)$ & $2$ &
$\frac{1}{12}(0,1,0,0,3,0,2)$ & $2$ \\ \hline

$\frac{1}{12}(0,0,1,2,0,1,0)$ & $2$ &
$\frac{1}{24}(0,0,5,1,2,0,4)$ & $2$ &
\cellcolor{blue!20} $\frac{1}{2}(0,0,0,0,0,1,0)$ & $4$ \\ \hline

$\frac{1}{2}(1,0,0,0,0,0,1)$ & $4$ &
$\frac{1}{4}(2,0,0,0,0,1,0)$ & $4$ &
$\frac{1}{4}(0,0,1,0,0,0,2)$ & $4$ \\ \hline

$\frac{1}{4}(1,0,0,0,1,0,1)$ & $4$ &
$\frac{1}{4}(0,1,0,0,1,0,0)$ & $4$ &
$\frac{1}{4}(0,0,0,0,0,1,2)$ & $4$ \\ \hline

$\frac{1}{4}(1,0,0,1,0,0,0)$ & $4$ &
$\frac{1}{4}(0,0,1,0,0,1,0)$ & $4$ &
$\frac{1}{6}(1,0,0,0,2,0,1)$ & $4$ \\ \hline

$\frac{1}{6}(0,0,2,0,0,1,0)$ & $4$ &
$\frac{1}{8}(0,2,3,0,2,0,0)$ & $4$ &
$\frac{1}{12}(2,0,0,3,0,1,0)$ & $4$ \\ \hline

$\frac{1}{12}(0,1,3,0,0,0,5)$ & $4$ &
\cellcolor{blue!20}$(0,0,0,0,0,0,1)$ & $6$ &
\cellcolor{blue!20}$(1,0,0,0,0,0,0)$ & $6$ \\ \hline

\cellcolor{blue!20}$\frac{1}{2}(0,0,0,0,1,0,0)$ & $6$ &
$\frac{1}{2}(1,1,0,0,0,0,0)$ & $6$ &
\cellcolor{blue!20}$\frac{1}{2}(0,0,1,0,0,0,0)$ & $6$ \\ \hline

$\frac{1}{3}(1,0,0,0,0,1,1)$ & $6$ &
$\frac{1}{3}(0,1,0,0,0,0,2)$ & $6$ &
$\frac{1}{3}(1,0,0,1,0,0,0)$ & $6$ \\ \hline

$\frac{1}{3}(2,0,0,0,0,1,0)$ & $6$ &
$\frac{1}{4}(2,0,0,1,0,0,0)$ & $6$ &
$\frac{1}{4}(3,0,0,0,0,1,0)$ & $6$ \\ \hline

$\frac{1}{4}(1,0,0,0,0,1,2)$ & $6$ &
$\frac{1}{4}(0,0,0,1,0,1,0)$ & $6$ &
$\frac{1}{4}(0,0,0,1,0,0,2)$ & $6$ \\ \hline

$\frac{1}{4}(0,1,1,0,0,0,1)$ & $6$ &
$\frac{1}{4}(0,1,0,0,0,0,3)$ & $6$ &
$\frac{1}{4}(2,0,0,1,0,0,0)$ & $6$ \\ \hline

$\frac{1}{5}(2,0,0,1,0,1,0)$ & $6$ &
$\frac{1}{6}(1,0,0,1,0,2,0)$ & $6$ &
$\frac{1}{8}(4,0,0,1,0,2,0)$ & $6$ \\ \hline

$\frac{1}{8}(3,2,0,1,0,1,0)$ & $6$ &
$\frac{1}{12}(0,0,1,2,1,0,5)$ & $6$ &
$\frac{1}{12}(5,0,0,3,0,1,0)$ & $6$ \\ \hline

$\frac{1}{12}(1,0,0,3,0,1,4)$ & $6$ &
\cellcolor{blue!20}$\frac{1}{2}(1,0,0,0,0,1,0)$ & $8$ &
\cellcolor{blue!20}$\frac{1}{2}(0,1,0,0,0,0,1)$ & $8$ \\ \hline

$\frac{1}{2}(1,0,1,0,0,0,0)$ & $8$ &
$\frac{1}{4}(2,1,1,0,0,0,1)$ & $8$ &
$\frac{1}{4}(1,2,0,0,0,1,0)$ & $8$ \\ \hline

$\frac{1}{4}(1,0,0,1,0,1,0)$ & $8$ &
$\frac{1}{4}(0,0,1,0,1,0,1)$ & $8$ & 
$\frac{1}{6}(3,0,1,1,0,1,0)$ & $8$ \\ \hline

$\frac{1}{6}(0,0,1,1,1,0,1)$ & $8$ &
$\frac{1}{6}(1,0,0,2,0,1,0)$ & $8$ &
$\frac{1}{8}(1,2,0,2,0,1,0)$ & $8$ \\ \hline

$\frac{1}{2}(0,1,0,0,0,1,0)$ & $10$ &
$\frac{1}{2}(1,1,0,0,0,0,1)$ & $10$ &
\cellcolor{blue!20}$\frac{1}{2}(0,0,0,1,0,0,0)$ & $10$ \\ \hline

$\frac{1}{2}(1,0,0,0,1,0,0)$ & $10$ &
$\frac{1}{4}(2,0,0,1,0,1,0)$ & $10$ &
$\frac{1}{4}(1,2,0,0,1,0,1)$ & $10$ \\ \hline

$\frac{1}{4}(0,0,0,1,0,2,0)$ & $10$ &
$\frac{1}{4}(0,2,0,1,0,0,0)$ & $10$ &
\cellcolor{blue!20}$(0,0,0,0,0,1,0)$ & $12$ \\ \hline

\cellcolor{blue!20}$(0,1,0,0,0,0,0)$ & $12$ &
\cellcolor{blue!20}$\frac{1}{2}(2,0,0,0,0,1,0)$ & $12$ &
$\frac{1}{2}(1,0,0,0,0,1,1)$ & $12$ \\ \hline

$\frac{1}{4}(4,0,0,0,0,1,2)$ & $12$ &
$\frac{1}{4}(2,0,2,0,0,1,0)$ & $12$ &
$\frac{1}{4}(2,1,1,0,1,0,0)$ & $12$ \\ \hline

$\frac{1}{4}(0,2,1,0,0,1,0)$ & $12$ &
$\frac{1}{4}(1,0,0,1,1,0,1)$ & $12$ &
$\frac{1}{4}(1,0,0,1,0,1,2)$ & $12$ \\ \hline

$\frac{1}{4}(0,0,1,0,0,3,0)$ & $12$ &
$\frac{1}{4}(3,0,0,1,0,1,0)$ & $12$ & 
$\frac{1}{4}(0,1,1,0,1,0,2)$ & $12$ \\ \hline

$\frac{1}{4}(1,0,0,1,0,2,0)$ & $12$ &
$\frac{1}{4}(3,0,0,1,0,0,2)$ & $12$ &
$\frac{1}{6}(4,0,0,2,0,1,0)$ & $12$ \\ \hline

$\frac{1}{12}(8,0,0,3,1,0,5)$ & $12$ &
\cellcolor{blue!20}$(1,0,0,0,0,0,1)$ & $14$ &
\cellcolor{blue!20}$(0,0,1,0,0,0,0)$ & $14$ \\ \hline

\cellcolor{blue!20}$\frac{1}{2}(1,0,0,1,0,0,0)$ & $14$ &
\cellcolor{blue!20}$\frac{1}{2}(0,0,1,0,0,1,0)$ & $14$ &
$\frac{1}{2}(0,0,1,1,0,1,1)$ & $14$ \\ \hline

$\frac{1}{2}(0,0,0,1,0,0,1)$ & $14$ &
$\frac{1}{2}(0,0,1,0,1,0,0)$ & $14$ &
$\frac{1}{3}(2,0,0,0,1,0,2)$ & $14$ \\ \hline

$\frac{1}{4}(2,0,0,1,1,0,1)$ & $14$ &
$\frac{1}{4}(2,2,0,0,1,0,1)$ & $14$ &
$\frac{1}{4}(0,0,2,1,0,1,0)$ & $14$ \\ \hline

$\frac{1}{4}(3,0,0,0,1,0,3)$ & $14$ &
$\frac{1}{4}(0,1,3,0,0,0,1)$ & $14$ &
\cellcolor{blue!20}$\frac{1}{2}(1,0,0,0,1,0,1)$ & $16$ \\ \hline

\cellcolor{blue!20}$\frac{1}{2}(0,1,1,0,0,0,1)$ & $16$ &
$\frac{1}{2}(1,0,1,0,0,1,0)$ & $16$ &
$\frac{1}{4}(2,1,1,0,1,1,0)$ & $16$ \\ \hline

\cellcolor{blue!20}$\frac{1}{2}(0,0,0,1,0,1,0)$ & $18$ &
$\frac{1}{2}(1,0,0,1,0,0,1)$ & $18$ &
$\frac{1}{2}(1,1,0,0,1,0,0)$ & $18$ \\ \hline

$\frac{1}{3}(1,0,1,1,0,1,0)$ & $18$ &
$\frac{1}{4}(2,1,1,0,1,0,2)$ & $18$ &
$\frac{1}{4}(0,0,0,1,2,1,0)$ & $18$ \\ \hline

$\frac{1}{6}(2,1,2,1,1,1,1)$ & $18$ &
\cellcolor{blue!20}$(1,0,0,0,0,1,0)$ & $20$ &
\cellcolor{blue!20}$(0,0,0,0,1,0,0)$ & $20$ \\ \hline

$\frac{1}{3}(2,0,0,1,0,2,0)$ & $20$ &
$\frac{1}{4}(3,0,0,1,0,3,0)$ & $20$ &
\cellcolor{blue!20}$\frac{1}{2}(1,0,0,1,0,1,0)$ & $22$ \\ \hline

\cellcolor{blue!20}$(1,0,1,0,0,0,0)$ & $24$ &
$\frac{1}{2}(2,0,1,0,1,0,0)$ & $24$ &
$\frac{1}{3}(3,0,1,1,0,1,0)$ & $24$ \\ \hline

$\frac{1}{4}(1,1,0,2,1,1,3)$ & $24$ & 
$\frac{1}{4}(2,1,1,0,1,2,2)$ & $24$ &
$\frac{1}{4}(4,0,2,1,0,1,0)$ & $24$ \\ \hline

$(0,0,0,1,0,0,0)$ & $26$ &
\cellcolor{blue!20}$\frac{1}{2}(2,1,1,0,0,0,1)$ & $26$ &
$\frac{1}{2}(1,0,0,1,0,1,1)$ & $26$ \\ \hline

\cellcolor{blue!20}$\frac{1}{2}(2,0,0,1,0,1,0)$ & $26$ &
$\frac{1}{4}(4,1,1,0,2,0,1)$ & $26$ &
$\frac{1}{4}(4,0,0,1,2,1,0)$ & $26$ \\ \hline

\cellcolor{blue!20}$(1,0,0,0,1,0,0)$ & $28$ &
$\frac{1}{2}(1,1,1,0,0,1,1)$ & $28$ &
\cellcolor{blue!20}$(1,0,0,0,0,1,1)$ & $30$ \\ \hline

\cellcolor{blue!20}$(0,0,1,0,0,1,0)$ & $30$ &
$\frac{1}{2}(1,0,0,1,0,2,0)$ & $30$ &
$\frac{1}{4}(3,0,0,1,0,3,4)$ & $30$ \\ \hline

\cellcolor{blue!20}$(0,0,0,1,0,0,1)$ & $32$ &
$\frac{1}{2}(1,0,1,1,0,1,0)$ & $32$ &
\cellcolor{blue!20}$\frac{1}{2}(1,0,0,1,0,1,2)$ & $32$ \\ \hline

$\frac{1}{2}(1,1,1,0,1,0,1)$ & $32$ &
\cellcolor{blue!20}$\frac{1}{2}(0,1,1,0,1,0,2)$ & $32$ &
$\frac{1}{2}(1,1,0,1,0,1,1)$ & $34$ \\ \hline

$\frac{1}{3}(1,1,1,1,1,1,1)$ & $36$ &
\cellcolor{blue!20}$(1,0,1,0,0,1,0)$ & $40$ &
$\frac{1}{2}(2,0,1,1,0,1,1)$ & $40$ \\ \hline

$\frac{1}{4}(4,1,3,0,1,3,0)$ & $40$ &
$\frac{1}{4}(4,1,3,0,1,2,2)$ & $40$ &
\cellcolor{blue!20}$(0,0,0,1,0,1,0)$ & $42$ \\ \hline

\cellcolor{blue!20}$\frac{1}{2}(2,1,1,0,1,1,0)$ & $42$ &
$\frac{1}{2}(1,1,1,0,1,1,1)$ & $42$ &
\cellcolor{blue!20}$\frac{1}{2}(2,1,1,0,1,0,2)$ & $42$ \\ \hline

\cellcolor{blue!20}$(1,0,0,1,0,0,1)$ & $44$ &
\cellcolor{blue!20}$(1,0,0,1,0,1,0)$ & $50$ &
$\frac{1}{2}(1,1,1,1,1,1,1)$ & $56$ \\ \hline

\cellcolor{blue!20}$(1,0,0,1,0,1,1)$ & $60$ &
\cellcolor{blue!20}$\frac{1}{2}(2,1,1,0,1,2,2)$ & $60$ &
$\frac{1}{2}(1,1,1,1,1,1,2)$ & $62$ \\ \hline

\cellcolor{blue!20}$(1,0,1,1,0,1,0)$ & $72$ &
$\frac{1}{2}(2,1,2,1,1,1,1)$ & $72$ &
\cellcolor{blue!20}$(1,1,1,0,1,0,1)$ & $74$ \\ \hline

\cellcolor{blue!20}$(1,1,1,0,1,1,1)$ & $92$ &
\cellcolor{blue!20}$(1,1,1,1,1,1,1)$ & $126$ & & \\ \hline

\end{longtable}

\begin{longtable}{|c|c|c|c|c|c|}
\caption{Pseudo-unipotent infinitesimal characters: type $E_8$}\label{table:fictitiousE8}\\ \hline
central char & codim & central char & codim & central char & codim  \\ \hline

\cellcolor{blue!20}$(0,0,0,0,0,0,0,0)$ & $0$ &
$\frac{1}{3}(0,0,0,0,0,0,0,1)$ & $0$ &
$\frac{1}{3}(1,0,0,0,0,0,0,0)$ & $0$ \\ \hline

$\frac{1}{4}(0,0,0,0,0,0,0,1)$ & $0$ &
$\frac{1}{4}(1,0,0,0,0,0,0,0)$ & $0$ &
$\frac{1}{4}(0,0,0,0,0,0,1,0)$ & $0$ \\ \hline

$\frac{1}{4}(0,1,0,0,0,0,0,0)$ & $0$ &
$\frac{1}{5}(1,0,0,0,0,0,0,1)$ & $0$ &
$\frac{1}{7}(1,0,0,0,0,1,0,0)$ & $0$ \\ \hline

$\frac{1}{8}(1,0,0,0,0,0,0,2)$ & $0$ &
$\frac{1}{8}(0,0,1,0,0,0,1,0)$ & $0$ &
$\frac{1}{8}(1,0,0,0,0,1,0,0)$ & $0$ \\ \hline

$\frac{1}{8}(1,0,0,0,1,0,0,0)$ & $0$ &
$\frac{1}{12}(2,0,0,0,0,1,0,1)$ & $0$ &
$\frac{1}{12}(0,2,0,0,1,0,0,0)$ & $0$ \\ \hline

$\frac{1}{12}(1,0,0,0,0,0,3,0)$ & $0$ &
$\frac{1}{12}(3,0,0,0,0,0,0,1)$ & $0$ &
$\frac{1}{12}(0,2,1,0,0,0,0,0)$ & $0$ \\ \hline

$\frac{1}{12}(0,0,1,0,0,1,1,0)$ & $0$ &
$\frac{1}{12}(0,0,0,0,0,1,2,0)$ & $0$ &
$\frac{1}{15}(2,0,0,0,0,1,2,0)$ & $0$ \\ \hline

$\frac{1}{16}(1,0,0,1,0,1,1,0)$ & $0$ &
$\frac{1}{20}(0,1,3,0,0,0,1,0)$ & $0$ &
$\frac{1}{20}(1,0,0,0,0,3,1,0)$ & $0$ \\ \hline

$\frac{1}{24}(5,0,0,0,0,1,2,0)$ & $0$ &
$\frac{1}{24}(0,0,1,0,2,0,3,0)$ & $0$ &
$\frac{1}{28}(0,0,3,1,0,0,3,0)$ & $0$ \\ \hline

$\frac{1}{40}(2,0,0,1,4,0,3,0)$ & $0$ &
$\frac{1}{60}(0,3,4,0,5,0,3,0)$ & $0$ &
\cellcolor{blue!20}$\frac{1}{2}(0,0,0,0,0,0,0,1)$ & $2$ \\ \hline

$\frac{1}{4}(1,0,0,0,0,0,0,1)$ & $2$ &
$\frac{1}{4}(0,0,1,0,0,0,0,0)$ & $2$ & 
$\frac{1}{4}(0,0,0,0,0,1,0,0)$ & $2$ \\ \hline

$\frac{1}{6}(2,0,0,0,0,0,0,1)$ & $2$ &
$\frac{1}{6}(0,0,0,0,0,1,1,0)$ & $2$ &
$\frac{1}{8}(0,0,0,1,0,0,1,0)$ & $2$ \\ \hline

$\frac{1}{8}(2,0,0,0,0,1,0,0)$ & $2$ &
$\frac{1}{10}(2,0,0,0,0,1,1,0)$ & $2$ &
$\frac{1}{12}(1,0,0,0,0,2,1,0)$ & $2$ \\ \hline

$\frac{1}{12}(0,0,2,0,1,0,0,0)$ & $2$ &
$\frac{1}{12}(1,0,0,1,0,0,2,0)$ & $2$ &
$\frac{1}{12}(0,1,2,0,0,0,1,0)$ & $2$ \\ \hline

$\frac{1}{14}(0,0,0,1,1,0,2,0)$ & $2$ &
$\frac{1}{20}(0,0,1,1,2,0,1,0)$ & $2$ &
$\frac{1}{24}(0,0,2,1,0,0,5,0)$ & $2$ \\ \hline

$\frac{1}{30}(0,1,3,0,2,0,4,0)$ & $2$ &
\cellcolor{blue!20}$\frac{1}{2}(1,0,0,0,0,0,0,0)$ & $4$ &
$\frac{1}{4}(1,0,0,0,0,0,0,2)$ & $4$ \\ \hline

$\frac{1}{4}(0,0,0,0,1,0,0,0)$ & $4$ &
$\frac{1}{4}(0,0,0,0,0,1,0,1)$ & $4$ &
$\frac{1}{4}(1,0,0,0,0,0,1,0)$ & $4$ \\ \hline

$\frac{1}{4}(0,0,1,0,0,0,0,1)$ & $4$ &
$\frac{1}{6}(1,0,0,0,0,0,2,0)$ & $4$ &
$\frac{1}{6}(0,1,0,0,1,0,0,0)$ & $4$ \\ \hline

$\frac{1}{8}(0,0,0,0,1,0,2,0)$ & $4$ &
$\frac{1}{8}(0,0,0,1,0,1,0,1)$ & $4$ &
$\frac{1}{10}(0,1,0,0,1,0,2,0)$ & $4$ \\ \hline

$\frac{1}{12}(0,1,0,0,2,0,1,0)$ & $4$ &
$\frac{1}{12}(1,0,0,0,0,3,0,2)$ & $4$ &
$\frac{1}{12}(0,0,2,1,0,0,0,2)$ & $4$ \\ \hline

$\frac{1}{20}(0,0,1,3,0,1,0,2)$ & $4$ &
\cellcolor{blue!20} $(0,0,0,0,0,0,0,1)$ & $6$ &
\cellcolor{blue!20} $\frac{1}{2}(0,0,0,0,0,0,1,0)$ & $6$ \\ \hline

$\frac{1}{3}(1,0,0,0,0,0,0,2)$ & $6$ &
$\frac{1}{3}(0,0,0,0,0,1,0,1)$ & $6$ &
$\frac{1}{4}(1,0,0,0,0,0,0,3)$ & $6$ \\ \hline

$\frac{1}{4}(1,0,0,0,0,1,0,0)$ & $6$ &
$\frac{1}{4}(0,0,1,0,0,0,0,2)$ & $6$ &
$\frac{1}{4}(0,0,0,1,0,0,0,0)$ & $6$ \\ \hline

$\frac{1}{4}(0,0,0,0,0,1,0,2)$ & $6$ &
$\frac{1}{4}(0,1,0,0,0,0,1,0)$ & $6$ &
$\frac{1}{5}(1,0,0,0,0,1,0,2)$ & $6$ \\ \hline

$\frac{1}{6}(0,2,0,0,0,0,1,0)$ & $6$ &
$\frac{1}{6}(2,0,0,0,0,1,0,1)$ & $6$ &
$\frac{1}{8}(2,0,0,0,0,1,0,4)$ & $6$ \\ \hline

$\frac{1}{8}(0,0,0,1,0,1,0,3)$ & $6$ &
$\frac{1}{12}(0,0,2,1,0,0,0,5)$ & $6$ &
$\frac{1}{12}(1,0,0,0,0,3,0,5)$ & $6$ \\ \hline

$\frac{1}{12}(0,0,3,0,0,1,0,4)$ & $6$ & 
$\frac{1}{12}(1,0,0,2,0,1,0,1)$ & $6$ &
$\frac{1}{12}(0,0,2,1,0,0,0,5)$ & $6$ \\ \hline

$\frac{1}{12}(0,0,1,2,0,0,1,0)$ & $6$ &
$\frac{1}{20}(0,0,1,3,0,1,0,7)$ & $6$ &
\cellcolor{blue!20}$\frac{1}{2}(1,0,0,0,0,0,0,1)$ & $8$ \\ \hline

\cellcolor{blue!20}$\frac{1}{2}(0,1,0,0,0,0,0,0)$ & $8$ &
$\frac{1}{2}(0,0,0,0,0,0,1,1)$ & $8$ & 
$\frac{1}{4}(1,0,0,0,0,1,0,1)$ & $8$ \\ \hline

$\frac{1}{4}(0,0,0,1,0,0,0,1)$ & $8$ &
$\frac{1}{4}(0,0,1,0,0,0,1,0)$ & $8$ &
$\frac{1}{4}(0,1,0,0,0,0,1,2)$ & $8$ \\ \hline

$\frac{1}{4}(1,0,0,0,1,0,0,0)$ & $8$ &
$\frac{1}{6}(0,2,0,0,0,0,1,3)$ & $8$ &
$\frac{1}{6}(0,0,1,0,0,1,1,0)$ & $8$ \\ \hline

$\frac{1}{6}(1,0,0,0,0,1,1,3)$ & $8$ &
$\frac{1}{6}(1,0,0,0,0,2,0,1)$ & $8$ &
$\frac{1}{6}(0,0,1,1,0,0,0,1)$ & $8$ \\ \hline

$\frac{1}{8}(0,0,0,1,0,2,0,1)$ & $8$ &
$\frac{1}{8}(1,0,0,1,0,1,1,0)$ & $8$ &
$\frac{1}{10}(0,0,1,1,0,2,0,1)$ & $8$ \\ \hline

$\frac{1}{12}(0,0,1,2,0,1,0,2)$ & $8$ &
$\frac{1}{12}(0,0,1,1,1,0,1,6)$ & $8$ &
$\frac{1}{2}(0,1,0,0,0,0,0,1)$ & $10$ \\ \hline

\cellcolor{blue!20} $\frac{1}{2}(0,0,0,0,0,1,0,0)$ & $10$ &
$\frac{1}{4}(1,0,0,0,0,1,0,2)$ & $10$ &
$\frac{1}{4}(0,0,1,0,0,1,0,0)$ & $10$ \\ \hline

\cellcolor{red!20} $\frac{1}{4}(0,1,0,0,1,0,0,1)$ & $10$ &
$\frac{1}{4}(0,0,0,1,0,0,0,2)$ & $10$ &
$\frac{1}{4}(2,0,0,0,0,1,0,0)$ & $10$ \\ \hline

$\frac{1}{4}(1,0,0,1,0,0,0,0)$ & $10$ &
$\frac{1}{6}(0,0,2,0,0,1,0,0)$ & $10$ &
$\frac{1}{6}(0,1,1,0,0,1,0,2)$ & $10$ \\ \hline

$\frac{1}{12}(0,0,1,2,0,1,0,5)$ & $10$ &
\cellcolor{blue!20}$(1,0,0,0,0,0,0,0)$ & $12$ &
\cellcolor{blue!20}$\frac{1}{2}(0,0,1,0,0,0,0,0)$ & $12$ \\ \hline

\cellcolor{blue!20}$\frac{1}{2}(1,0,0,0,0,0,0,2)$ & $12$ &
$\frac{1}{4}(0,0,0,1,0,0,1,0)$ & $12$ &
$\frac{1}{4}(0,0,1,0,0,1,0,1)$ & $12$ \\ \hline

$\frac{1}{4}(2,0,0,0,1,0,0,0)$ & $12$ &
$\frac{1}{4}(1,0,0,0,0,1,0,3)$ & $12$ &
$\frac{1}{4}(0,0,0,0,1,0,1,2)$ & $12$ \\ \hline

$\frac{1}{4}(1,0,0,0,0,0,2,2)$ & $12$ &
$\frac{1}{4}(0,0,0,1,0,0,0,3)$ & $12$ &
$\frac{1}{4}(2,0,0,0,0,1,0,1)$ & $12$ \\ \hline

$\frac{1}{4}(1,0,0,1,0,0,0,1)$ & $12$ &
$\frac{1}{4}(3,0,0,0,0,0,1,0)$ & $12$ &
$\frac{1}{4}(0,1,1,0,0,0,1,0)$ & $12$ \\ \hline

$\frac{1}{6}(1,0,0,0,0,2,0,4)$ & $12$ &
$\frac{1}{8}(0,0,0,1,0,2,0,5)$ & $12$ &
$\frac{1}{10}(0,0,1,1,0,2,0,6)$ & $12$ \\ \hline

$\frac{1}{12}(0,0,1,2,0,1,0,8)$ & $12$ &
$\frac{1}{12}(0,1,0,0,3,0,2,6)$ & $12$ &
\cellcolor{blue!20}$(0,0,0,0,0,0,1,0)$ & $14$ \\ \hline

\cellcolor{blue!20}$\frac{1}{2}(0,0,0,0,0,1,0,1)$ & $14$ &
$\frac{1}{2}(1,1,0,0,0,0,0,0)$ & $14$ &
\cellcolor{blue!20}$\frac{1}{2}(1,0,0,0,0,0,1,0)$ & $14$ \\ \hline

$\frac{1}{3}(0,1,0,0,0,0,2,0)$ & $14$ &
$\frac{1}{4}(0,0,1,0,0,1,0,2)$ & $14$ &
$\frac{1}{4}(1,0,0,0,1,0,1,0)$ & $14$ \\ \hline

$\frac{1}{4}(0,1,0,0,0,0,3,0)$ & $14$ &
$\frac{1}{4}(0,0,0,1,0,0,2,0)$ & $14$ & 
$\frac{1}{4}(0,0,0,1,0,1,0,0)$ & $14$ \\ \hline

$\frac{1}{4}(1,0,0,0,0,1,2,0)$ & $14$ &
$\frac{1}{6}(0,0,2,0,0,1,0,3)$ & $14$ &
\cellcolor{blue!20}$\frac{1}{2}(0,0,1,0,0,0,0,1)$ & $16$ \\ \hline

\cellcolor{blue!20}$\frac{1}{2}(0,1,0,0,0,0,1,0)$ & $16$ &
$\frac{1}{2}(1,0,0,0,0,0,1,1)$ & $16$ &
\cellcolor{blue!20}$\frac{1}{2}(0,0,0,0,1,0,0,0)$ & $16 $\\ \hline

$\frac{1}{4}(0,0,0,1,0,1,0,1)$ & $16$ &
$\frac{1}{4}(1,0,0,1,0,0,1,0)$ & $16$ &
$\frac{1}{4}(1,0,0,0,1,0,1,2)$ & $16$ \\ \hline

$\frac{1}{4}(0,0,0,1,0,1,0,1)$ & $16$ &
$\frac{1}{6}(1,0,0,0,2,0,1,3)$ & $16$ &
\cellcolor{blue!20}$\frac{1}{2}(1,0,0,0,0,1,0,0)$ & $18$ \\ \hline

$\frac{1}{2}(0,0,0,0,1,0,0,1)$ & $18$ &
$\frac{1}{3}(1,0,0,0,0,1,1,1)$ & $18$ &
$\frac{1}{4}(1,0,0,1,0,1,0,0)$ & $18$ \\ \hline

$\frac{1}{4}(1,2,0,0,0,1,0,0)$ & $18$ &
$\frac{1}{4}(0,0,0,1,0,1,0,2)$ & $18$ &
$\frac{1}{4}(0,1,1,0,0,0,1,2)$ & $18$ \\ \hline

$\frac{1}{12}(1,0,0,3,0,1,4,4)$ & $18$ &
$\frac{1}{12}(2,1,1,0,1,2,4,4)$ & $18$ &
\cellcolor{blue!20}$(1,0,0,0,0,0,0,1)$ & $20$ \\ \hline

\cellcolor{blue!20}$(0,1,0,0,0,0,0,0)$ & $20$ &
\cellcolor{blue!20}$\frac{1}{2}(0,0,0,1,0,0,0,0)$ & $20$ &
$\frac{1}{3}(2,0,0,0,0,1,0,2)$ & $20$ \\ \hline

$\frac{1}{3}(1,1,0,0,0,1,0,1)$ & $20$ &
$\frac{1}{4}(1,0,0,1,0,1,0,1)$ & $20$ &
$\frac{1}{4}(2,0,0,1,0,0,0,3)$ & $20$ \\ \hline

$\frac{1}{4}(3,0,0,0,0,1,0,3)$ & $20$ &
$\frac{1}{6}(0,0,1,1,1,0,1,2)$ & $20$ &
$\frac{1}{8}(4,0,0,1,0,2,0,5)$ & $20$ \\ \hline

$\frac{1}{12}(5,0,0,3,0,1,0,8)$ & $20$ &
\cellcolor{blue!20}$\frac{1}{2}(1,0,0,0,0,1,0,1)$ & $22$ &
$\frac{1}{2}(0,1,0,0,0,1,0,0)$ & $22$ \\ \hline

$\frac{1}{4}(1,0,0,1,0,1,0,2)$ & $22$ &
$\frac{1}{6}(1,0,0,2,0,1,0,3)$ & $22$ &
\cellcolor{blue!20}$(0,0,0,0,0,0,1,1)$ & $24$ \\ \hline

\cellcolor{blue!20}$\frac{1}{2}(0,0,0,1,0,0,0,1)$ & $24$ &
\cellcolor{blue!20}$\frac{1}{2}(1,0,0,0,1,0,0,0)$ & $24$ &
$\frac{1}{3}(1,0,0,0,0,1,1,3)$ & $24$ \\ \hline

$\frac{1}{3}(0,1,0,0,0,0,2,3)$ & $24$ &
$\frac{1}{4}(2,1,1,0,0,0,1,2)$ & $24$ &
$\frac{1}{4}(1,0,0,1,0,1,1,0)$ & $24$ \\ \hline

$\frac{1}{4}(0,0,0,1,0,2,0,1)$ & $24$ &
$\frac{1}{4}(0,1,0,0,0,0,3,4)$ & $24$ &
$\frac{1}{4}(0,0,0,1,0,0,2,4)$ & $24$ \\ \hline

$\frac{1}{4}(1,0,0,0,0,1,2,4)$ & $24$ &
$\frac{1}{12}(0,0,1,2,1,0,5,12)$ & $24$ &
$\frac{1}{12}(1,0,0,3,0,1,4,12)$ & $24$ \\ \hline

\cellcolor{blue!20}$(0,0,0,0,0,1,0,0)$ & $26$ &
\cellcolor{blue!20}$\frac{1}{2}(0,1,0,0,0,0,1,2)$ & $26$ &
\cellcolor{blue!20}$\frac{1}{2}(1,0,0,0,0,1,0,2)$ & $26$ \\ \hline

$\frac{1}{4}(0,0,1,0,1,0,1,4)$ & $26$ &
$\frac{1}{4}(0,0,1,0,0,3,0,0)$ & $26$ &
$\frac{1}{4}(2,0,0,1,0,1,0,2)$ & $26$ \\ \hline

$\frac{1}{4}(1,0,0,1,0,1,0,4)$ & $26$ &
$\frac{1}{4}(1,2,0,0,0,1,0,4)$ & $26$ &
$\frac{1}{6}(1,0,0,2,0,1,0,6)$ & $26$ \\ \hline

$\frac{1}{6}(0,0,1,1,1,0,1,6)$ & $26$ &
\cellcolor{blue!20}$\frac{1}{2}(0,0,0,1,0,0,0,2)$ & $28$ &
\cellcolor{blue!20}$\frac{1}{2}(0,0,1,0,0,1,0,0)$ & $28$ \\ \hline

$\frac{1}{2}(1,1,0,0,0,0,1,1)$ & $28$ &
\cellcolor{blue!20}$(1,0,0,0,0,0,1,0)$ & $30$ &
\cellcolor{blue!20}$(0,1,0,0,0,0,0,1)$ & $30$ \\ \hline

\cellcolor{blue!20}$\frac{1}{2}(2,0,0,0,0,1,0,1)$ & $30$ &
$\frac{1}{4}(2,0,0,1,0,1,2,0)$ & $30$ &
$\frac{1}{4}(0,2,1,0,0,1,0,3)$ & $30$ \\ \hline

$\frac{1}{4}(3,0,0,1,0,1,0,2)$ & $30$ &
$\frac{1}{4}(3,0,0,0,1,0,3,0)$ & $30$ &
$\frac{1}{4}(1,0,0,1,1,0,1,2)$ & $30$ \\ \hline

\cellcolor{blue!20}$(0,0,0,0,1,0,0,0)$ & $32$ &
\cellcolor{blue!20}$\frac{1}{2}(1,0,0,0,1,0,1,0)$ & $32$ &
\cellcolor{blue!20}$\frac{1}{2}(1,0,0,1,0,0,0,1)$ & $32$ \\ \hline

\cellcolor{blue!20}$\frac{1}{2}(0,1,1,0,0,0,1,0)$ & $32$ &
\cellcolor{blue!20}$\frac{1}{2}(0,0,1,0,0,1,0,1)$ & $32$ &
$\frac{1}{2}(1,0,0,0,0,1,1,1)$ & $32$ \\ \hline

$\frac{1}{2}(0,1,1,0,0,0,1,0)$ & $32$ &
\cellcolor{blue!20}$\frac{1}{2}(0,0,0,1,0,1,0,0)$ & $32$ &
\cellcolor{blue!20}$\frac{1}{2}(0,0,0,1,0,0,1,0)$ & $32$ \\ \hline

$\frac{1}{4}(0,1,1,0,1,0,2,2)$ & $32$ &
$\frac{1}{4}(1,0,0,1,0,1,2,2)$ & $32$ &
$\frac{1}{2}(0,0,1,0,1,0,0,1)$ & $34$ \\ \hline

$\frac{1}{2}(1,1,0,0,0,1,0,1)$ & $34$ &
$\frac{1}{2}(0,0,0,1,0,0,1,1)$ & $34$ &
$\frac{1}{2}(1,0,0,1,0,0,1,0)$ & $36$ \\ \hline

$\frac{1}{3}(0,1,1,0,1,0,1,1)$ & $36$ &
$\frac{1}{4}(2,1,1,0,1,1,0,1)$ & $36$ &
$\frac{1}{6}(2,1,1,0,1,2,2,2)$ & $36$ \\ \hline

$\frac{1}{4}(1,1,1,0,1,1,1,1)$ & $38$ &
$\frac{1}{6}(2,2,1,1,1,1,1,1)$ & $38$ &
\cellcolor{blue!20}$(1,0,0,0,0,0,1,1)$ & $40$ \\ \hline

$\frac{1}{3}(2,0,0,0,1,0,2,3)$ & $40$ &
$\frac{1}{4}(3,0,0,0,1,0,3,4)$ & $40$ &
$\frac{1}{4}(2,1,1,0,0,0,3,4)$ & $40$ \\ \hline

$\frac{1}{4}(2,0,0,1,0,1,2,4)$ & $40$ &
$\frac{1}{5}(1,1,1,1,1,1,1,1)$ & $40$ &
\cellcolor{blue!20}$(1,0,0,0,0,1,0,0)$ & $42$ \\ \hline

$\frac{1}{2}(1,0,0,1,0,0,1,1)$ & $42$ &
\cellcolor{blue!20}$\frac{1}{2}(1,0,0,0,1,0,1,2)$ & $42$ &
\cellcolor{blue!20}$\frac{1}{2}(0,1,1,0,0,0,1,2)$ & $42$ \\ \hline

$\frac{1}{2}(0,1,1,0,0,0,1,2)$ & $42$ &
$\frac{1}{4}(2,1,1,0,1,0,2,2)$ & $42$ &
$\frac{1}{4}(3,0,0,1,0,3,0,0)$ & $42$ \\ \hline

\cellcolor{blue!20}$\frac{1}{2}(0,0,0,1,0,1,0,2)$ & $44$ &
\cellcolor{blue!20}$\frac{1}{2}(1,0,0,1,0,1,0,0)$ & $44$ &
$\frac{1}{2}(1,1,0,0,1,0,0,2)$ & $44$ \\ \hline

\cellcolor{blue!20}$(0,0,0,0,1,0,0,1)$ & $46$ &
\cellcolor{blue!20}$\frac{1}{2}(1,0,0,1,0,1,0,1)$ & $48$ &
\cellcolor{blue!20}$(1,0,0,0,0,1,0,1)$ & $50$ \\ \hline

$\frac{1}{4}(3,0,0,1,0,3,0,4)$ & $50$ &
\cellcolor{blue!20}$\frac{1}{2}(1,0,0,1,0,1,0,2)$ & $52$ &
$\frac{1}{4}(1,1,1,1,1,1,1,1)$ & $52$ \\ \hline

\cellcolor{blue!20}$(0,0,0,1,0,0,0,1)$ & $56$ &
\cellcolor{blue!20}$\frac{1}{2}(1,0,0,1,0,1,1,0)$ & $56$ &
$\frac{1}{4}(2,1,1,0,1,2,2,2)$ & $56$ \\ \hline

$\frac{1}{2}(1,0,0,1,0,1,1,1)$ & $58$ &
\cellcolor{blue!20}$\frac{1}{2}(2,1,1,0,0,0,1,2)$ & $60$ &
\cellcolor{blue!20}$\frac{1}{2}(2,0,0,1,0,1,0,2)$ & $60$ \\ \hline

$\frac{1}{4}(4,1,1,0,2,0,1,4)$ & $60$ &
\cellcolor{blue!20}$(1,0,0,0,1,0,0,1)$ & $62$ &
$\frac{1}{2}(1,1,1,0,0,1,1,1)$ & $62$ \\ \hline

\cellcolor{blue!20}$(0,0,0,1,0,0,1,0)$ & $64$ &
$\frac{1}{2}(1,1,1,0,1,0,1,1)$ & $68$ &
\cellcolor{blue!20}$(1,0,0,0,0,1,1,1)$ & $72$ \\ \hline

$\frac{1}{3}(1,1,1,1,1,1,1,1)$ & $72$ &
$\frac{1}{4}(3,0,0,1,0,3,4,4)$ & $72$ &
$\frac{1}{4}(2,1,1,0,1,2,4,4)$ & $72$ \\ \hline

\cellcolor{blue!20}$(0,0,0,1,0,0,1,1)$ & $74$ &
\cellcolor{blue!20}$\frac{1}{2}(1,0,0,1,0,1,2,2)$ & $74$ &
\cellcolor{blue!20}$\frac{1}{2}(0,1,1,0,1,0,2,2)$ & $74$ \\ \hline

$\frac{1}{2}(1,1,0,1,0,1,1,2)$ & $76$ &
$\frac{1}{3}(1,1,1,1,1,1,1,3)$ & $78$ &
\cellcolor{blue!20}$(1,0,0,1,0,0,1,0)$ & $84$ \\ \hline

\cellcolor{blue!20}$\frac{1}{2}(2,1,1,0,1,1,0,1)$ & $84$ &
$\frac{1}{2}(1,1,1,0,1,1,1,1)$ & $86$ &
\cellcolor{blue!20}$\frac{1}{2}(2,1,1,0,1,0,2,2)$ & $92$ \\ \hline

\cellcolor{blue!20}$(1,0,0,1,0,0,1,1)$ & $94$ &
\cellcolor{blue!20}$(1,0,0,1,0,1,0,1)$ & $104$ &
$\frac{1}{2}(1,1,1,1,1,1,1,1)$ & $112$ \\ \hline

\cellcolor{blue!20}$(1,0,0,1,0,1,1,1)$ & $126$ &
\cellcolor{blue!20}$\frac{1}{2}(2,1,1,0,1,2,2,2)$ & $126$ &
$\frac{1}{2}(1,1,1,1,1,1,2,2)$ & $128$ \\ \hline

\cellcolor{blue!20}$(1,1,1,0,1,0,1,1)$ & $148$ &
\cellcolor{blue!20}$(1,1,1,0,1,1,1,1)$ & $182$ &
\cellcolor{blue!20}$(1,1,1,1,1,1,1,1)$ & $240$ \\ \hline

\end{longtable}

\begin{sloppypar} \printbibliography[title={References}] \end{sloppypar}

\end{document}